\documentclass[10pt,a4paper]{article}

\usepackage{preamble}
\usepackage{cobordisms}
\usepackage{pinlabel}

\title{2-cabling and tangle operators in Khovanov theory}
\author{Mihai Marian}

\begin{document}
\maketitle

\begin{abstract}
We describe an operator on 4-ended tangles that is induced by 2-cabling of a strongly invertible knot. By passing to the 4-ended tangle Khovanov theory of Kotelskiy--Watson--Zibrowius, this induces an operator on the category of type D structures over the Bar-Natan algebra $\mc{B}$, as well as on a Fukaya category of the 4-punctured 2-sphere. We provide a full description of this operator's restriction to cap-trivial tangles. Finally, we extract geography results that are inspired by a recent concordance invariant of Lewark--Zibrowius.
\end{abstract}

\section{Introduction}

%\textbf{Story of HFK.} 
Given a knot invariant, it is natural to ask how the invariant behaves under cabling. For the Alexander polynomial $\Delta$, this is well-known (see, e.g., \cite[Theorem 6.15]{Lic97}):
\[\Delta_{K_{p,q}}(t) = \Delta_{T(p,q)}(t) \cdot \Delta_K(t^p),\]
where $K_{p,q}$ is the $(p,q)$-cable of a knot $K$ and $T(p,q)$ is the $(p,q)$ torus knot. This formula led Hedden to analyse the knot Floer homology of cables \cite{Hed05,Hed09}, starting with consideration of $(2,n)$-cables. The culmination\footnote{Until such a theorem is proved for the other versions of Heegaard Floer theory.} of this line of work is the description of $(p,q)$-cabling as an operator on the bordered Heegaard Floer homology of knot complements \cite{HanWat23b}:
\[\wh{\HF}(K^c) \xmapsto{\mathscr{F}_{p,q}} \wh{\HF}(K_{p,q}^c),\]
where the complement $K^c$ has boundary parametrized by the meridian and Seifert longitude. The invariant $\wh{\HF}(K^c)$ is an immersed curve (with local system) in the punctured torus $T^2_*$ and it is equivalent to the bordered type D structure $\wh{\CFD}(K^c)$ \cite{HRW23}. The operator $\mathscr{F}_{p,q}$ may be described as a Lagrangian correspondence from $T^2_*$ to itself that is particularly simple: it is given by the graph of a $p$-valued function $T^2_* \ra T^2_*$. We illustrate it for $K = T(2,3)$ and $(p,q) = (2,1)$ in \cref{fig:cabTrefHF}. See \cite{HanWat23b} for other beautiful figures.

\begin{figure}[h]
\labellist
\hair 2pt
\pinlabel $\mathscr{F}_{2,1}$ at 360 175
\endlabellist
\centering
\includegraphics[width=0.8\textwidth]{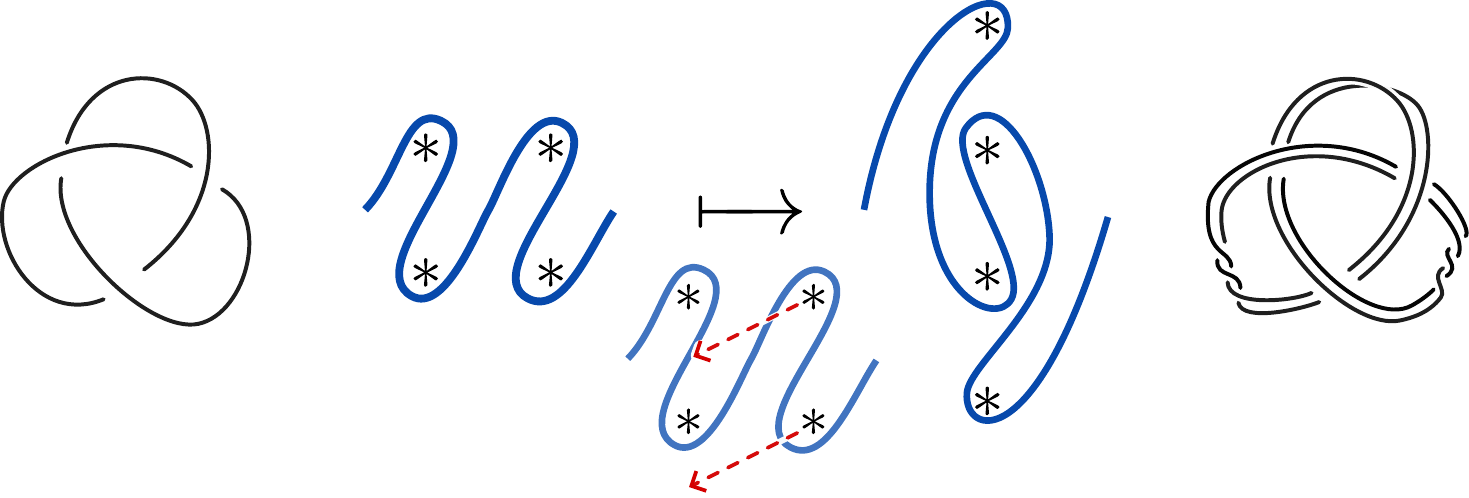}
\caption{The immersed curve invariants $\wh{\HF}(K^c)$ and $\wh{\HF}(K_{2,1}^c)$, lifted to the universal Abelian cover $\R^2 \setminus \Z^2 \ra T^2_*$, where $K$ is the Seifert-framed right-handed trefoil. In the cover, the operator $\mathscr{F}_{p,q}$ is given by a collection of finger moves that slide marked lattice points along lines of slope $q/p$, as indicated by the red dotted arrows.}
\label{fig:cabTrefHF}	
\end{figure}

\begin{wrapfigure}[15]{R}{0.22\textwidth}
\includegraphics[width=0.2\textwidth]{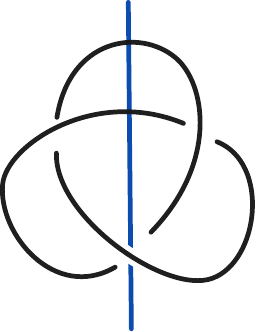}
\caption{The trefoil with the axis of a strong inversion.}
\label{fig:trefSI}
\end{wrapfigure}
%\textbf{Invariant of SI Knots.} 
In this work, we consider Khovanov theory, but our goal is not the obvious one of describing the Khovanov homology of a cable, although this is also interesting. Indeed, the Jones polynomial of cables is itself not well understood and formulas for it would be instrumental in proving instances of the volume conjecture, see \cite{McPS}. Instead, given the spectral sequence from the Khovanov homology of a link to the Heegaard Floer homology of its branched double-cover \cite{OS05} and what is known about the knot Floer homology of cables, it is natural to consider strongly invertible knots. Before presenting our main results, we define strongly invertible knots, explain how to construct a Khovanov-type invariant for them, and how $(2,n)$-cabling acts on this invariant.

A knot $K$ in $S^3$ is (smoothly) strongly invertible if there is a smooth, orientation preserving involution $h \co S^3 \ra S^3$ that fixes $K$ and reverses a choice of orientation on $K$. For example, rotation by $\pi$ about the axis in \cref{fig:trefSI}. In fact, we may assume without loss of generality that a given strong inversion really is rotation by $\pi$ about an axis that intersects $K$ in two points \cite{Wal69}; see also \cite{BRW23}. Given a strongly invertible knot $(K,h)$, there is a standard way to construct a 4-ended tangle $T_h$ that is an invariant of $(K,h)$: under the quotient map $S^3 \ra S^3\!/(x \sim h(x))$, the image of the pair $(K^c, \mathrm{Fix}(h) \cap K^c)$ is a 4-ended tangle $(B^3, T_h)$ (by the smooth Sch\"onflies theorem). Thus $K^c$ is the 2-cover of $B^3$, branched at $T_h$, a space that is denoted $\Sigma(B^3, T_h)$: 
\begin{notn} The branched 2-cover of a manifold $M$ with branch set $B$ is denoted $\Sigma(M, B)$. If $M = S^3$ and $B$ is a knot $K$, then it is denoted $\Sigma(K)$.
\end{notn}
See \cite{Sak85} for an early study of strongly invertible knots via the image of $\mathrm{Fix}(h)$ in the quotient (in the PL setting). The construction of $T_h$ is depicted in \cref{fig:SITangle} for the trefoil. Therein we also keep track of the Seifert longitude, in order to be able to use the so-called ``Montesinos trick", one incarnation of which is the following observation: by closing the tangle $T_h$ with the $r$-rational tangle (defined in \ref{def:cutPaste} for $r \in \Z$, which is the only case we consider in this article) and taking the branched cover over the resulting link $\Sigma(T_h(r))$, one obtains $S^3_r(K)$, the $r$-framed Dehn surgery on $K$ (if $T_h$ is correctly parametrized) \cite{Mon75}. Note that, strictly speaking, $T_h$ is associated with the strongly invertible knot \textit{together with a choice of longitude}. Absent this choice, it is a family of tangles $T_h$ that is associated with a strongly invertible knot, where two tangles in the family differ by the number of crossings on the right-hand side (according to our diagram conventions). We call this number of crossings the framing of the tangle $T_h$, and the Seifert framing is the one for which the Montesinos trick works as advertised.
%The Montesinos trick underlies the spectral sequence of Ozsv\'ath--Szab\'o.

\begin{figure}[h]
\labellist
\hair 2pt
\pinlabel $\color{rgb:red,14;green,21;blue,10}\lambda$ at 199 219
\pinlabel $\color{rgb:red,14;green,21;blue,10}\conj{\lambda}$ at 120 84
\endlabellist
	\centering
	\includegraphics[width=0.9\textwidth]{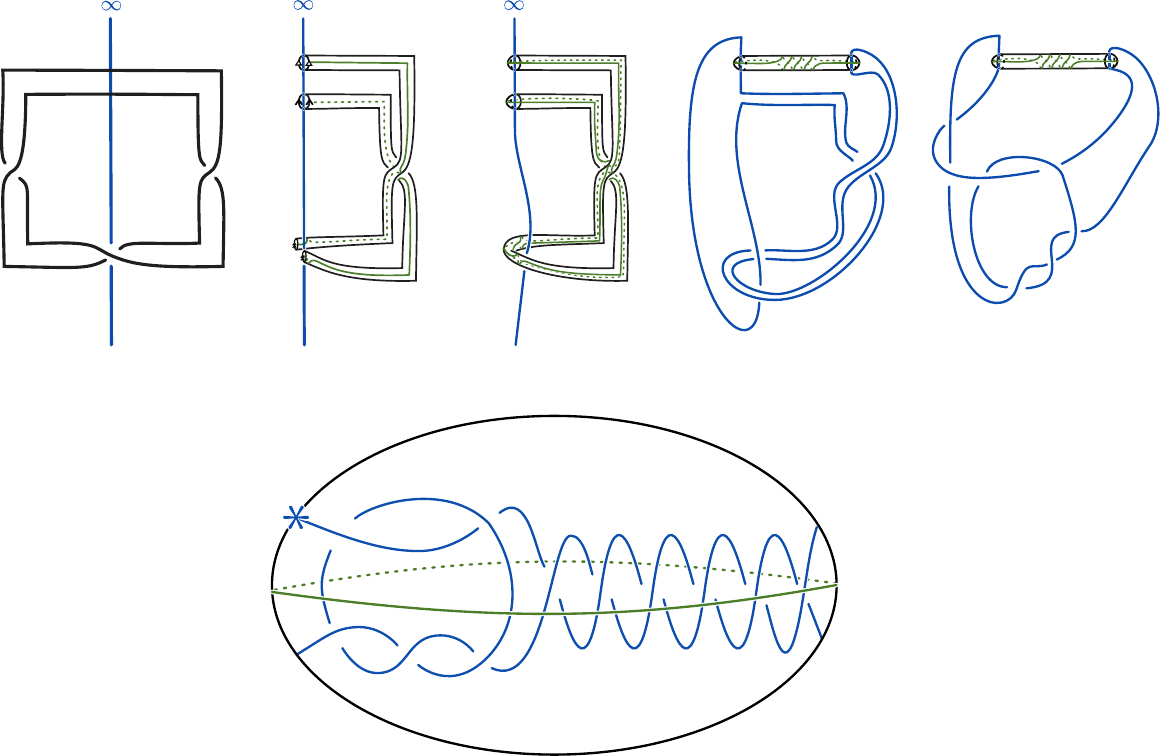}
	\caption{The tangle $T_h$ associated to the strong inversion on the right-handed trefoil. The Seifert longitude $\lambda$ in the fundamental domain of the quotient and its image $\conj{\lambda}$ are shown in green.}
	\label{fig:SITangle}
\end{figure}

Here is what is meant by the closure of a tangle:

\begin{defn}\label{def:cutPaste} Let $n \in \Z \cup \{\infty\}$. First, the rational $n$-tangle $Q_n$ is the one in \cref{fig:nTangle} for $n>0$. If $n<0$, then $Q_n = mQ_{-n}$, where $m$ denotes the mirror. And if $n = 0, \infty$, we set $Q_0 = \oRes$ and $Q_\infty = \iRes$. Second, given two 4-ended tangles $T_1$ and $T_2$, the link $\mc{L}(T_1, T_2)$ is obtained by identifying endpoints as in \cref{fig:glue} below. Finally, let the $n$-closure $T(n)$ of a 4-ended tangle $T$ be $\mc{L}(T, Q_{-n})$.
\end{defn}

\begin{figure}[ht]
	\centering
	\begin{minipage}{.5\textwidth}
		\labellist
		\hair 2pt
		\pinlabel $n$ at 33 10
		\endlabellist
		\centering
		\includegraphics[height = 0.1\textheight]{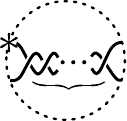}
		\caption{The tangle $Q_r$} \label{fig:nTangle}
	\end{minipage}%
	\begin{minipage}{.5\textwidth}
		\centering
		\includegraphics[height=0.1\textheight]{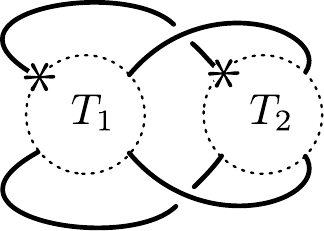}
		\caption{The link $\mc{L}(T_1, T_2)$.}	
		\label{fig:glue}
	\end{minipage}
\end{figure}

We will mostly be working with cap-trivial tangles:

\begin{notn} We call a tangle $T$ cap-trivial if $T(\infty)$ is an unknot. 
\end{notn}

By the following remark, we use the terms ``cap-trivial tangle" and ``tangle associated with a strong inversion" interchangeably.

\begin{rmk} Cap-trivial tangles are in natural bijection with tangles associated with strongly invertible knots. 
\end{rmk}
\begin{proof} If $T$ is cap-trivial, then $\Sigma(T(0)) = S^3$ and $\Sigma(B^3, T) \subset S^3$ is a knot complement with an involution given by the nontrivial deck transformation. This involution induces a strong inversion on the knot core. Conversely, it is clear that tangles associated with strong inversions are cap-trivial.
\end{proof}

\begin{rmk} A word of warning: in small examples, such as the one in \cref{fig:SITangle}, it seems alluringly straightforward to diagrammatically follow the Seifert longitude through the construction in order to obtain the Seifert-framed tangle $T_h$. This procedure is treacherous in general. It is more secure to verify the framing by computing determinants: by the Montesinos trick, the branched double cover $\Sigma(T_h(0))$ is 0-surgery on $K$ (if $T_h$ is Seifert-framed), so the link determinant must satisfy $\det(T_h(0)) = 0$. A statement that uses only determinants of knots is: $T_h$ is Seifert-framed if and only if $\det(T_h(\pm1)) = 1$. Moreover, the determinant may be computed using either the Alexander polynomial $\Delta$ or the Jones polynomial $V$:
\[\det(K) = |\Delta_K(-1)| = |V_K(-1)| = \left| \chi_{gr}\left( \wt{\Kh}(K) \right) \big|_{t=-1}\right|,\]
where $\chi_{gr}$ is the graded Euler characteristic. For example, $T_h(1) = m10_{124}$ for the tangle in \cref{fig:SITangle}. The Alexander polynomial of $10_{124}$ is listed in \cite{Rol03} as
\[1 -(t+t^{-1}) + (t^3+t^{-3}) - (t^4 + t^{-4}),\]
so we see that $\Delta_{T_h(1)}(-1) = 1$.
\end{rmk}

The Khovanov-type invariant of $(K,h)$ that we consider is the immersed curve invariant of the 4-ended tangle $T_h$, defined in \cite{KWZ19}. It is an object
\[\wt{\BN}(T_h) \in \wrapFuk(S^2_{4,*}),\]
where $\wrapFuk(S^2_{4,*})$ is a (partially wrapped) Fukaya category of the 4-punctured sphere, one of the punctures being distinguished. We will describe this invariant in more detail below. For now we note only that $\wt{\BN}(T_h)$ is a (possibly disconnected) immersed curve in $S^2_{4,*}$, with some additional decorations. In particular, these curves are bigraded in an appropriate sense. However, for this introduction, we can ignore the additional data and think of $\wt{\BN}(T_h)$ as an immersed curve in $S^2_{4,*}$ that is, up to regular homotopy, an invariant of the isotopy class of $T_h$. 

Now, given a strongly invertible knot $(K,h)$, there is a uniquely induced strong inversion $h_{p,q}$ on the cable $K_{p,q}$, so there is a well-defined tangle operator 
\[\Cb^0 \co \cat{Tan}(4) \ra \cat{Tan}(4)\]
that takes the 4-ended tangle $T_h$ into the 4-ended tangle $T_{h_{2,1}}$. To see the operator, consider the involution on the $(2,1)$ pattern knot in the solid torus $D^2 \times S^1$ depicted in \cref{fig:21PatternSI}. Remove a tubular neighbourhood of the pattern knot and the result is a manifold with boundary consisting of two tori: inner and outer, drawn in green and red, respectively. Given a strongly invertible knot $(K,h)$, we construct $(K_{2,1}^c, h_{2,1})$ by identifying the outer boundary of $D^2 \times S^1$ with $\del K^c$ and by extending $h$ through the solid torus using the indicated involution. \Cref{fig:21PatternSI} therefore constructs the tangle operator $\Cb^0$: the image $\Cb^0(T)$ is simply the tangle obtained by placing $T$ inside the red inner sphere.

\begin{figure}[h]
	\centering
	\includegraphics[width=0.9\textwidth]{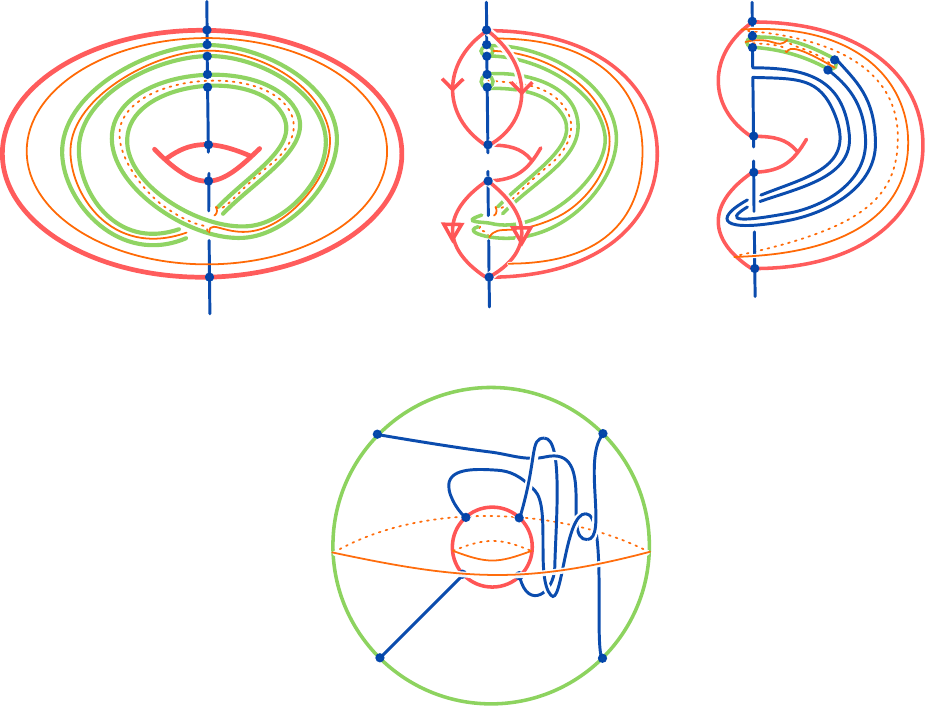}
	\caption{The strong inversion on the $(2,1)$ pattern and the induced tangle operator $\Cb^0$. The orange arcs are images of Seifert longitudes.}
	\label{fig:21PatternSI}
\end{figure}

%\textbf{Main Results.} 
Thus there is an induced operator on the image of $\wt{\BN}$ inside $\wrapFuk(S^2_{4,*})$:
\[\Cb^0 \co \left\{\wt{\BN}(T) \in \wrapFuk(S^2_{4,*}) \right\} \ra \left\{ \wt{\BN}(T) \in \wrapFuk(S^2_{4,*}) \right\},\]
given by
\[\Cb^0 \big( \wt{\BN}(T_h) \big) := \wt{\BN}\big( \Cb^0(T_h) \big).\]
We allow ourselves the repeated use of the symbol $\Cb$, since it should be clear from context whether we are talking about an operator acting on $\cat{Tan}(4)$ or one acting on the Fukaya category $\wrapFuk(S^2_{4,*})$. We will see below that our construction in fact yields an operator on the whole category $\wrapFuk(S^2_{4,*})$. However, we restrict our attention to Bar-Natan invariants of cap-trivial tangles. Our main result is a full description of $\Cb^0\big(\wt{\BN}(T_h)\big)$ in terms of $\wt{\BN}(T_h)$, which has the following restrictive geography result as an immediate corollary:

\begin{thm}\label{thm:geography} Given a cap-trivial tangle $T$, the unique non-compact component of the immersed curve $\Cb^0(\wt{\BN}(T))$ is, up to mirroring and framing, one of
\[\wt{\BN}(\vcenter{\hbox{\protect\includegraphics[scale=0.12]{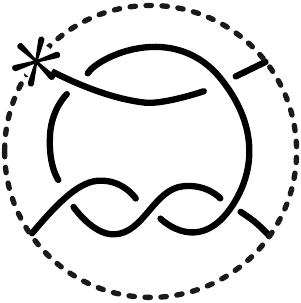}}}) \qquad \text{or} \qquad \wt{\BN}(\oRes).\]
\end{thm}

It is also natural to compare $\Cb$ with $\mathscr{F}_{2,1}$. Such operators on Fukaya categories are expected to arise as (generalized) Lagrangians in $X \times X$, where $X$ is either $T^2_*$ or $S^2_{4,*}$, with the action of the operator given as the composite $\pi_2 \circ \pi_1^{-1}$ of the projections $X \times X \ra X$. From the symplectic perspective, our operator seems more exotic than $\mc{F}_{p,q}$:

\begin{thm} The operator $\Cb^0 \co \wrapFuk(S^2_{4,*}) \ra \wrapFuk(S^2_{4,*})$ is not induced from a Lagrangian correspondence that is the graph of a $p$-valued continuous function of $S^2_{4, *}$, for any value of $p \in \N$.
\end{thm}

For example, let $T_{3_1}$ denote the tangle associated with the unique (see, e.g., \cite[\S10.6]{Kaw95}) strong inversion on the right-handed trefoil $3_1$. \Cref{fig:example1} illustrates $\wt{\BN}(T_{3_1})$ and its image under $\Cb^0$. The figure shows lifts of the immersed curves to the covering space 
\[ \R^2\setminus ({\scalebox{1}{$\frac{1}{2}$}} \Z)^2 \xra{\alpha} T^2_{4,*}  \xra{\beta} S^2_{4, *},\]
where $\alpha$ is the usual hyperelliptic projection and $\beta$ is the universal Abelian cover.

\begin{figure}[h]
\labellist
\hair 2pt
%\pinlabel {$\wt{\BN}(T_{3_1})$} at 150 0
\pinlabel $\Cb$ [t] at 307 310
%\pinlabel $\bigsqcup$ at 565 270
\endlabellist
\centering
\includegraphics[width=0.8\textwidth]{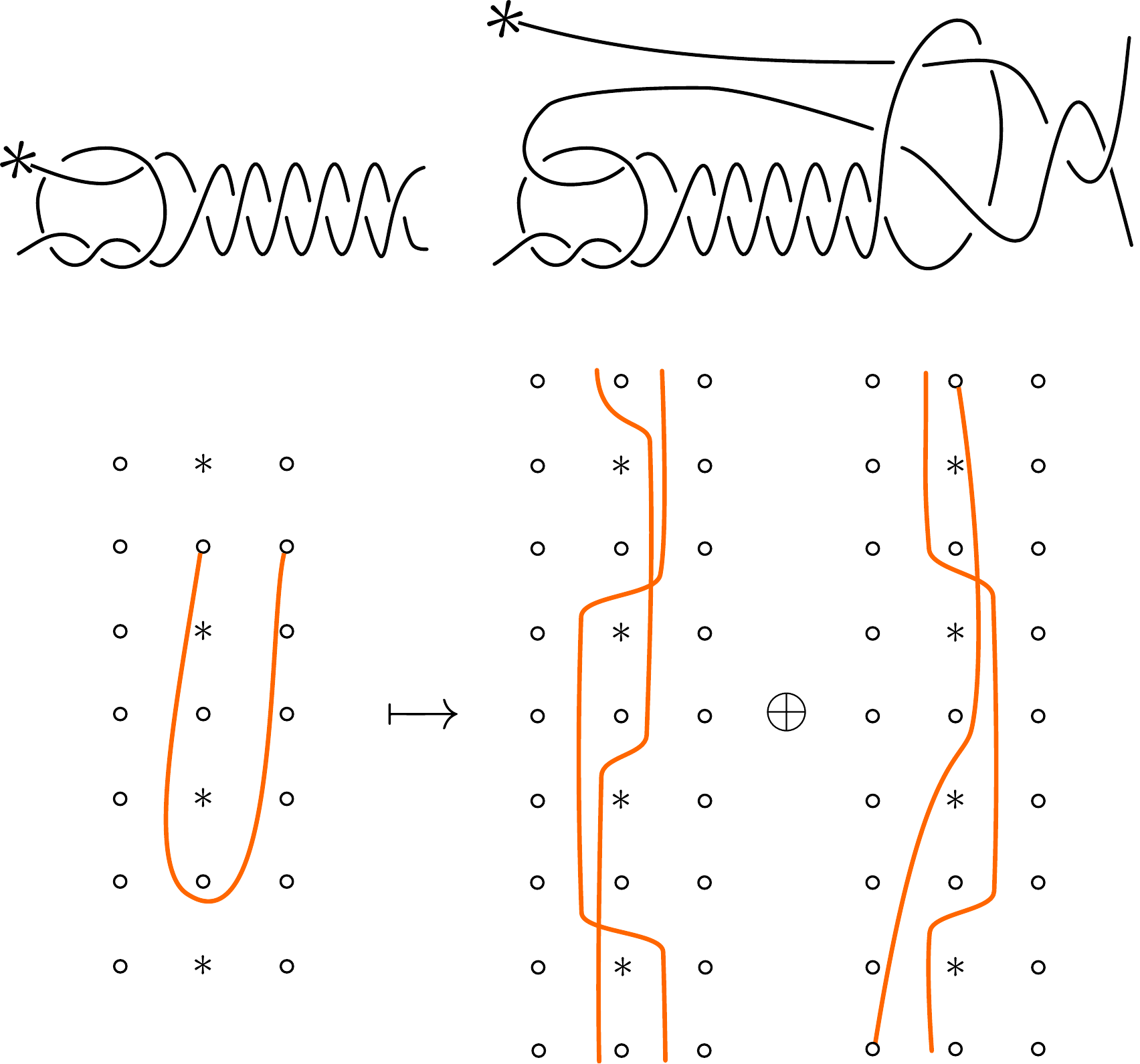}
\caption{The operator $\Cb^0$ acting on the curve $\wt{\BN}(T_{3_1})$ assigned to the Seifert-framed right-handed trefoil. The curve on the right is split into two summands, which correspond to a decomposition of the curve on the left as in \cref{fig:cabTrefoilFactorization}. Finally, the two curves in the left-hand summand of $\Cb^0(\wt{\BN}(T_{3_1}))$ are homotopic in $S^2_{4,*}$, but we picked two different lifts to the covering space, to suggest a difference in grading.}
\label{fig:example1}
\end{figure}

The examples in \cref{fig:example1,fig:cabTrefHF} display behaviour that is generic in the following sense. On the knot Floer side, $\mathscr{F}_{p,q}$ always preserves the non-compact component of $\wh{\HF}(K^c)$, whereas $\Cb^0$ almost never does---see \cref{thm:curveFactorization}. In both cases, iterating the operator results in exponential growth of the rank of the invariant $\wh{\HF}(K^c)$ or $\wt{\BN}(T_h)$ (defined as the rank of the module underlying the type D structure that corresponds to the immersed curve). The precise computation of $\Cb(\wt{\BN}(T_{3_1})$ is in \cref{sec:trefoilComputation}.

Most of our work follows from the following factorization result together with some model computations. To state it, we use the notion of a peg-board diagram from \cite[\S7.1]{HRW23}: thicken every puncture of $S^2_{4,*}$ into an open disc of small radius $\epsilon_k$, for $k \in \{1, 2, 3, 4\}$, a so-called ``peg". The covering space of $S^2_{4,*}$ is then a ``peg-board" 
\[\R^2 \setminus \bigcup_{(i,j) \in \Z^2} D^2_{\epsilon_{k}}(i/2 , j/2),\]
where $k$ depends on $(i, j)$. If $\gamma$ is an element of $\wrapFuk(S^2_{4,*})$, always pick a representative of $\gamma$ that, away from its ends, is geodesic with respect to the flat metric on the peg-board covering space (the ends of $\gamma$ necessarily join the boundary). Let $\gamma|_{D_\bullet=0}$ denote the limit of the peg-board representative of $\gamma$ as the radii of the pegs corresponding to non-special punctures along columns containing $\ast$ go to 0. We call this new curve the one obtained from $\gamma$ by ``pulling tight at $D_\bullet$". We can now make the following statement; see \cref{def:tightAtD} and \cref{thm:factorization} for the more refined version that takes gradings into account (by stating it in terms of type D structures), and \cref{fig:cabTrefoilFactorization} for an illustration.

\begin{thm}\label{thm:curveFactorization} Ignoring grading information, the operator $\Cb^0 \co \wrapFuk(S^2_{4,*}) \ra \wrapFuk(S^2_{4,*})$, restricted to immersed curve invariants of cap-trivial tangles, factors through pulling tight at $D_\bullet$: if $T$ is cap-trivial then
\[\Cb^0(\wt{\BN}(T)) = \Cb^0\big(\wt{\BN}(T)|_{D_\bullet=0}\big).\]
\end{thm}

\begin{figure}[h]
\labellist
\hair 2pt
\pinlabel $\Cb^0$ at 370 600
\pinlabel $D_\bullet{=0}$ at 224 350
\pinlabel $\Cb^0$ at 520 348
\endlabellist
	\centering
	\includegraphics[width=0.7\textwidth]{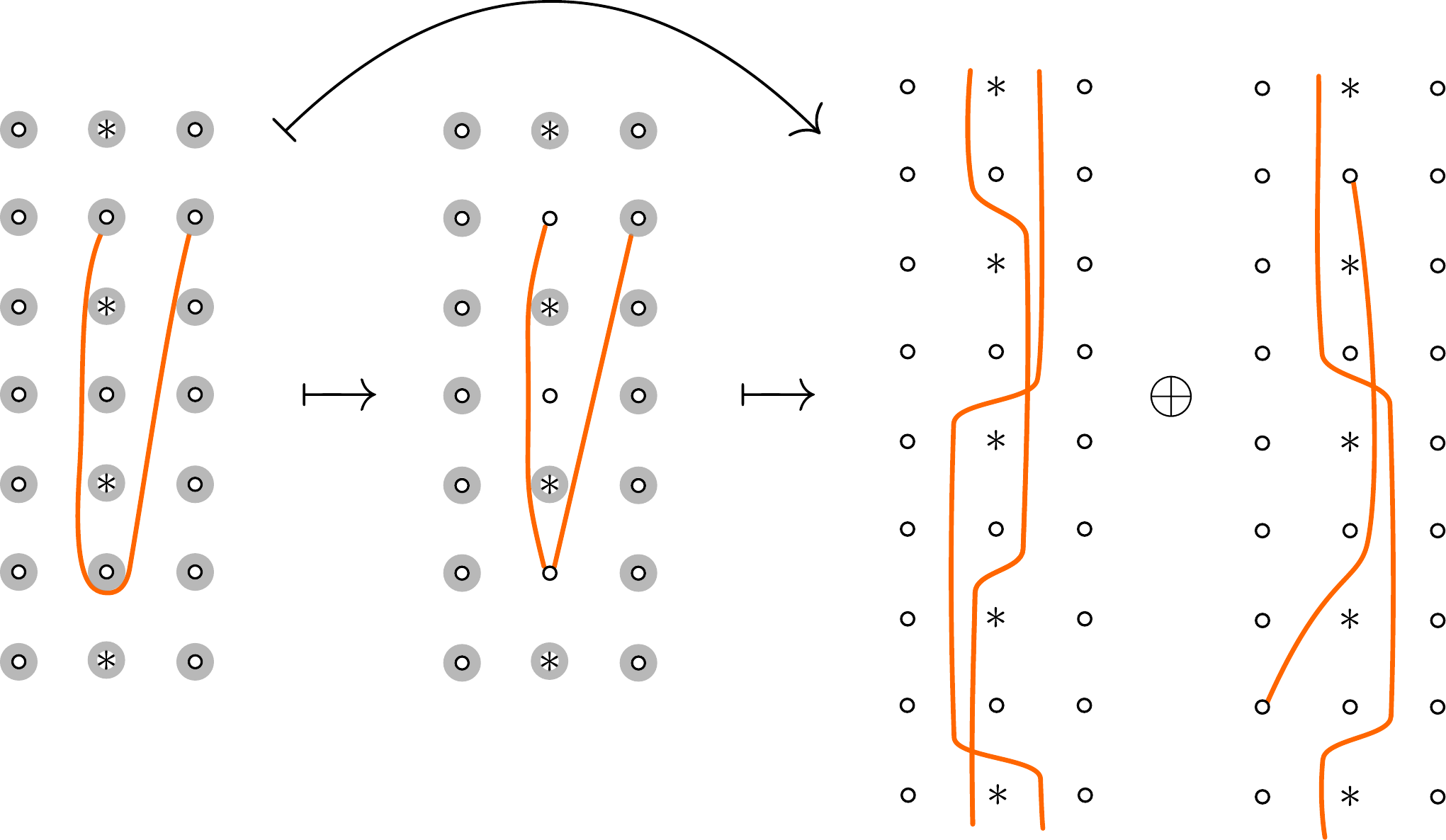}
	\caption{Factorization of $\Cb^0$ through pulling tight at $D_\bullet$. The summands on the right correspond to the two components of the pulled-tight curve in the middle.}
	\label{fig:cabTrefoilFactorization}
\end{figure}
%The Bar-Natan invariant is a bordered theory, meaning that it is a tangle invariant that enjoys pairing theorems allowing for efficient recovery of either the Bar-Natan or Khovanov homology of a link obtained by gluing the ends of two 4-ended tangles.

%\textbf{Annular cobordisms and bimodules.} 
So far, we have stated our results in the language of the Fukaya category. In practice, the immersed curve invariant $\wt{\BN}(T)$ can be identified with a particularly nice representative of a homotopy class of chain complexes over a two-object $k$-linear category $\mc{B}$, where $k$ is a field. Equivalently, it is a type D structure over the algebra of morphisms of $\mc{B}$, which is also called $\mc{B}$. The category of bigraded type D structures over $\mc{B}$ is denoted $\cat{Mod}^\mc{B}$, and it is in this more algebraic category that we truly work. We discuss this all in \cref{sec:BNAlgebra,sec:curves}. Our cabling operator is then more honestly described as an endofunctor
\[\Cb^0 \co \cat{Mod}^{\mc{B}} \ra \cat{Mod}^{\mc{B}}.\]
(Again, note that we reuse the symbol $\Cb$). This endofunctor is an example of a planar algebra operation as defined in \cite[\S5]{BN05} and this is where the tractability of our computation comes from. Crucially, $\Cb^0 \co \cat{Tan}(4) \ra \cat{Tan}(4)$ takes the form of the annular tangle depicted at the bottom of \cref{fig:21PatternSI}. Annular tangles or, more generally, tangles with diagrams living in $D^2$ with sub-discs removed, have the structure of a planar algebra that is precisely described by Bar-Natan. Moreover, the particular annular tangle we consider has a sufficiently small number of crossings to allow for pencil-and-paper computations. One would also obtain tangle operators for any other $(p,q)$-cabler of strongly invertible knots, but the crossing-numbers suggest needing computer assistance to do anything serious in more generality. We will further reduce the crossing number by working with the tangle operator $\Cb$ in \cref{fig:planAlgOperator}. The difference between the two operators is how the output is framed. We will show in \cref{prop:detCb} that the operator in \cref{fig:21PatternSI} takes Seifert-framed cap-trivial tangles to Seifert-framed cap-trivial tangles, so $\Cb^0$ is, topologically, the ``correct" operator.

One perspective on the present work is to see it as a step in the study of the category of complexes of annular tangles with 4 endpoints on each boundary circle. Elements of this category induce bimodules on tangle invariants, in the sense of \cite{LOT15}, so it is the avenue for understanding the transformations that apply to immersed curve invariants of 4-ended tangles. While this 4-4 annular cobordism category is expected to be quite large, a more tractable category is
\[\cat{Kob}\left(\vcenter{\hbox{\protect\includegraphics[scale=0.07]{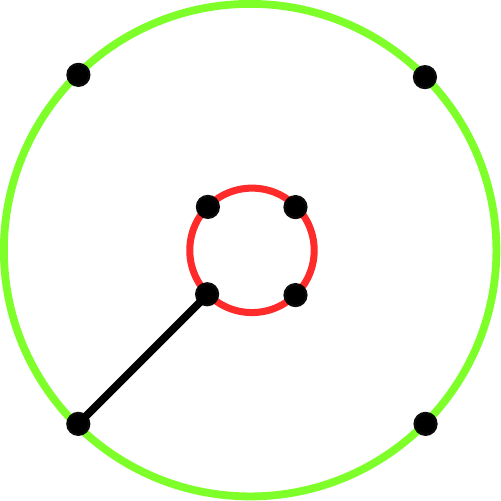}}}\right),\]
which is where our operator lives and, indeed, where live all operators on cap-trivial tangles induced by cabling of strongly invertible knots. To the author's knowledge, the only other step in this direction is done in \cite[\S8]{KWZ19}, in which case the subcategory under consideration is
\[\cat{Kob}\left(\vcenter{\hbox{\protect\includegraphics[scale=0.07]{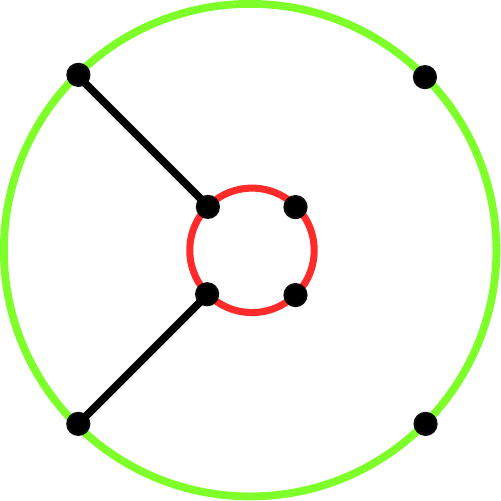}}}\right).\]
This latter category contains the operator induced by reframing a tangle or, equivalently, by applying a (half) Dehn twist to the rightmost endpoints. In \cite{KWZ19}, the study of the above category led to the very useful \cref{thm:MCG}. 

From this perspective, we may summarize our work as the computation of a partial bimodule over $\mc{B}$. Whereas such bimodules may end up being quite difficult to compute in general, our results show that it is possible to get tractability, if not ease, of computation by restricting to submodules of interest.

%\subsection{VarTheta}

Finally, we touch on the work of \cite{LZ24}. Therein, Lewark--Zibrowius define and study a new concordance invariant $\vartheta_c$, parametrized by a prime number $c$. One critical class of knots for the behaviour of $\vartheta_c$ is termed the $\vartheta_c$-rational knots. Precisely, a knot $K$ is $\vartheta_c$-rational if and only if the tangle associated with the strong inversion on $K \# K$ has the non-compact component of its immersed curve invariant equal to $\wt{\BN}(\oRes)$ \cite{Mar25}. More generally, we can define a $\vartheta_c$-rational tangle $T$ to be a cap-trivial tangle with non-compact immersed curve invariant equal to $\wt{\BN}(\oRes)$. On the one hand, given the restrictiveness of this condition, it is expected that tangles should generically not be $\vartheta_c$-rational. On the other hand, for every Seifert-framed cap-trivial tangle $T$, the tangle $\Cb^0(T)$ is $\vartheta_c$-rational. Therefore, iteration of $\Cb^0$ produces $\vartheta_c$-rational tangles in infinite families:

\begin{prop}\label{prop:CbRational} The class of $\vartheta_c$-rational tangles is fixed under the action of the operator $\Cb^0$.
\end{prop}

\subsection{Organization}

\Cref{sec:background} summarizes Bar-Natan's tangle theory \cite{BN05} and the algebraic part of \cite{KWZ19}, with a small change in notation: our handle addition operation is denoted $G$ and differs by a sign from the one in \cite{KWZ19}. \Cref{sec:construction} uses Bar-Natan's planar algebra structure to construct the operator $\Cb \co \cat{Mod}^\mc{B} \ra \cat{Mod}^\mc{B}$. \Cref{sec:properties} deals with the geometric part of \cite{KWZ19} as a means to obtain geography restrictions on the Bar-Natan invariants of tangles that arise from strong inversions and, ultimately, proves the technical lemmas at the heart of this project. The main insights in this paper can be gleaned from our computations in \cref{subsec:unknot,sec:trefoilComputation}.

\section{Background on Bar-Natan homology}\label{sec:background}

In this section we lay down notation for Bar-Natan's relative version of Khovanov's homology theory \cite{Kho00}, following \cite{BN05,KWZ19}.

\begin{wrapfigure}{R}{0.25\textwidth}\includegraphics[width=0.23\textwidth]{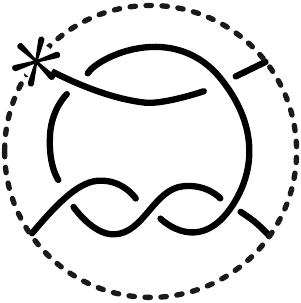}\end{wrapfigure} 

Tangles are often defined to be proper embeddings $T$ of a compact 1-manifold into the 3-ball $B^3$:
\[T \co {S^1}^{\sqcup l} \sqcup I^{\sqcup n} \hookrightarrow B^3,\]
where $l,n \in \N$, and up to ambient isotopies that fix the boundary $\del B^3$ set-wise. However when we speak of tangles, we will always mean pointed tangles that are \textit{framed}, meaning that one endpoint of $T$ is marked $\ast$, and the endpoints of $T$ are fixed equidistantly on an equatorial circle of $\del B^3$, such as in the adjacent figure. For such tangles, isotopies are required to fix the boundary of $B^3$ pointwise. When we draw a diagram of a tangle, if the marked endpoint is not indicated, it is assumed to be the top-left one. The collection of $2n$-ended tangles is denoted $\cat{Tan}(2n)$.

%Is the following necessary for the operator?
%{
%\begin{defn} A Conway tangle is a 4-ended tangle that has no closed component.
%\end{defn}
%}

%The following subsection describes complexes of cobordisms. This is quite important for Prop 3.9.

\subsection{Complexes of cobordisms}
\label{sec:notnReview}

Let $R$ be a commutative ring---in practice, either $\Z$ or a field $\k$.

\begin{defn} The category $\cat{Cob}(B)$ consists of crossingless 1-manifolds properly embedded in the disc $D$, with boundary equal to $B \subset \del D$, and up to isotopy. Morphisms are $R$-linear combinations of cobordisms that fix the boundary, up to isotopy.
\end{defn}

We often write $\cat{Cob}(4)$ to mean $\cat{Cob}(B)$, for some $B \subset D$ of cardinality 4. We also omit $B$ from the notation if we wish to speak in general terms, as in the following definition.

\begin{defn} The category $\cat{Cob}_{/l}$ has the same objects as $\cat{Cob}$, but its morphisms are quotients of morphisms in $\cat{Cob}$ by the $/l$ relations, depicted in \cref{fig:locRel}. Under these relations, a cobordism with a sphere component is equivalent to 0, a cobordism with a torus component is equivalent to twice the cobordism with the torus removed, and one can replace a cobordism containing a handle with a linear combination of cobordisms with the handle moved to different attaching regions in a ball neighbourhood, as in the figure.
\end{defn}

\begin{figure}[h]
	\labellist
	\hair 2pt
	\pinlabel $2$ at 124 102
	\pinlabel $0$ at 310 102
	%\pinlabel (T) at 98 120
	%\pinlabel (S) at 285 120
	%\pinlabel (4Tu) at 182 53
	\endlabellist
	\centering
	\includegraphics[scale=0.6]{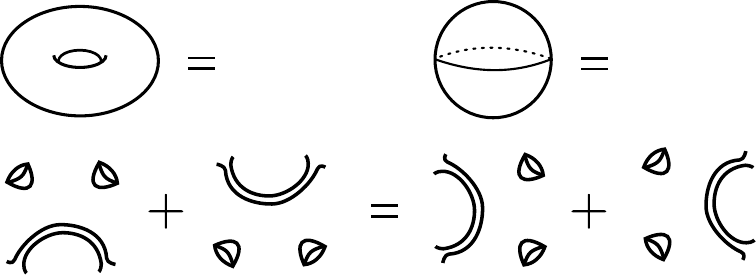}
	\caption{The $/l$ relations.}
	\label{fig:locRel}
\end{figure}

\begin{rmk} Both $\cat{Cob}$ and $\cat{Cob}_{/l}$ are preadditive categories, by design. They are formally upgraded into additive categories in the following manner.
\end{rmk}

\begin{defn} For any preadditive category $\cat{C}$, its \textbf{additive closure} $\cat{Mat}(\cat{C})$ consists of formal finite direct sums of objects in $\cat{C}$, with morphisms being matrices of morphisms in $\cat{C}$. 
\end{defn}

Here is an example, to make the notation completely explicit. A morphism in $\cat{Mat}(\cat{Cob})$ from $\mc{O}_1\oplus \mc{O}_2$ to $\mc{O}_1' \oplus \mc{O}_2'$ is a collection of four (possibly non-trivial $R$-linear combinations of) cobordisms $C_{ij} \co \mc{O}_j \ra \mc{O}'_i$, written as a matrix:
\[\begin{array}{c} \mc{O}_1 \\ \oplus \\ \mc{O}_2 \end{array} 
\xra{
	\begin{pmatrix} 
		C_{11} & C_{12} \\ 
		C_{21} & C_{22}
	\end{pmatrix}}
\begin{array}{c} \mc{O}'_1 \\ \oplus \\ \mc{O}'_2 \end{array}.
\]
Composition is given by the usual convention for matrix multiplication.

\begin{rmk} The category $\cat{Mat}(\cat{Cob}(B))$ is additive and contains a 0 object for formal reasons, as the empty direct sum. The 0 object is different from the empty manifold $\Oset \in \cat{Cob}(\Oset)$.
\end{rmk}

\begin{defn} If $\mc{A}$ is an additive category, then $\cat{Kom}(\mc{A})$ is the category of chain complexes over $\mc{A}$, and $\cat{Kom}_{/h}(\mc{A})$ is the na\"ive homotopy category of chain complexes over $\mc{A}$: it has the same objects as $\cat{Kom}(\mc{A})$, but morphisms in $\cat{Kom}_{/h}(\mc{A})$ are chain homotopy classes of morphisms in $\cat{Kom}(\mc{A})$.
\end{defn} 

\begin{notn} As a shorthand, let $\cat{Kob}(B) = \cat{Kom}(\cat{Mat}(\cat{Cob}^3_{/l}(B)))$ (for ``\textbf{Ko}mplexes of c\textbf{ob}ordisms", perhaps). Finally, let $\cat{Kob}_{/h}(B)$ be the category $\cat{Kob}(B)$ modulo chain homotopy, where $B$ may be empty. We also write write $\cat{Kob}$ and $\cat{Kob}_{/h}$ if $B$ is understood, or if we wish to talk about $\cup_k \cat{Kob}(k)$ or $\cup_k \cat{Kob}_{/h}(k)$.
\end{notn}

Importantly, Bar-Natan's cobordism categories are graded in the following sense.

\begin{defn}\label{def:gradCat} A \textbf{graded category} is a preadditive category where the hom-groups are graded Abelian groups such that $\mathrm{deg}(f\circ g) = \mathrm{deg}(f)+ \mathrm{deg}(g)$ whenever composition makes sense. Moreover, there is a $\Z$-action on objects 
\[(m, \mc{O}) \mapsto \prescript{m}{}{\mc{O}}\]
that does not change the hom-groups as Abelian groups. The $\Z$-action affects the gradings as follows: if $f \in \Hom(\mc{O}_1, \mc{O}_2)$ has degree $d$, then the corresponding morphism $f^{m_1}_{m_2} \in \Hom(\prescript{m_1}{}{\mc{O}_1}, \prescript{m_2}{}{\mc{O}_2})$ has degree $d + m_2 - m_1$.
\end{defn}

\begin{defn} \label{def:cobQuantGr} Let $C$ be a morphism in $\cat{Cob}(B)$. Define the quantum grading of $C$ to be 
\[q(C) = \chi(C) - \frac{1}{2}|B|.\]
\end{defn}

\begin{prop}[\!\! \cite{BN05} Exercise 6.3] \label{prop:kobIsGr} With the above definition, $\cat{Cob}(B)$ is a graded category. Moreover, the $/l$ relations are degree-homogeneous, so the grading descends to the quotient and $\cat{Cob}_{/l}(B)$ is graded as well. It follows that so are $\cat{Kob}(B)$ and $\cat{Kob}_{/h}(B)$.
\end{prop} 

It will be convenient to have a bit more structure, namely to make $\cat{Cob}$ into a \textit{bigraded} category. We do this by endowing cobordisms $C$ with a constant 0 grading,
\[h(C) = 0,\] 
called the \textit{homological} grading. We denote the corresponding $\Z$-action by
\[(n, \mc{O}) \mapsto \mc{O}_n.\]

\begin{notn} To allow ourselves the option of using indexing subscripts, we also sometimes use the homological algebra notation 
\[\mc{O}[n]\{m\} := \prescript{m}{}{\mc{O}_n}.\]
\end{notn}

From this point onwards, the categories $\cat{Cob}, \cat{Cob}_{/l}, \cat{Kob}$, etc. are all considered bigraded. Thus each of them is enriched over the category of bigraded groups or, as in \cite{KWZ19}, over bigraded chain complexes. 

We can now define the basic tangle invariant that we consider in this paper (assuming familiarity with Khovanov's construction of the cube of resolutions). Given an $n$-crossing tangle diagram $D$ with crossings labelled $1, \dots, n$, let $n_+, n_-$ be the number of positive and negative crossings in $D$, and let $\bsymb{\delta} = (\delta_1, \dots, \delta_n)$ be a coordinate of the cube $\{0,1\}^n$. Let
\[D^{\bsymb{\delta}} \in \cat{Cob}_{/l}\]
be the crossingless diagram obtained by performing the $\delta_i$-resolution at the $i^\text{th}$ crossing.

\begin{defn} \label{def:brGr} The bracket $\br{D} \in \cat{Kom}(\cat{Mat}(\cat{Cob}))$ of a diagram $D$ is the cube of resolutions of $D$, with the following bigrading shift on the object $\br{D}$ that is at the coordinate $\bsymb{\delta} = (\delta_1, \dots, \delta_n)$:
\[\br{D}(\bsymb{\delta}) = D^{\bsymb{\delta}}\left[n_- + \sum_i \delta_i\right]\left\{n_+ - 2n_- + \sum_i \delta_i\right\}.\]
\end{defn}
Note the following potentially confusing point. Each morphism in the cube $\br{D}$ is a saddle, and so it has inherent homological grading 0 and quantum grading $-1$ (independent of $|B|$), according to our definition. However, after we apply the bigrading shifts that define $\br{D}$, as an element of
\[\Hom(\br{D}(\bsymb{\delta}), \br{D}(\bsymb{\delta}^+)),\]
(where $\bsymb{\delta}^+$ is just $\bsymb{\delta}$ with one of the zero entries incremented by 1) a saddle has homological grading 1 and quantum grading 0.

\begin{defn} Let $T$ be an oriented tangle with a diagram $D$. The Bar-Natan bracket $\br{T}_{/l} \in \cat{Kob}$ is the quotient of $\br{D}$ by the $/l$-relations.
\end{defn}

Bar-Natan shows that this construction is a bigraded tangle invariant:

\begin{thm}[\!\! \cite{BN05} Theorem 3] \label{thm:qDegZero} All differentials in $\br{T}_{/l}$ have quantum grading 0. If $D_1$ and $D_2$ are diagrams for the same tangle $T$, then there is a chain homotopy $F \co \br{D_1}_{/l} \ra \br{D_2}_{/l}$ that has $q$-grading 0.
\end{thm}

Thus, $\br{T}_{/l}$ is, up to chain homotopy, and invariant of the isotopy class of $T$. Moreover, by composing with an appropriate TQFT, we can recover Khovanov homology, as well as other knot homology theories from the Bar-Natan bracket invariant \cite[\S7 and \S9]{BN05}.

%The following subsection is the two algebraic cranks I use over and over

\subsection{Homological algebra} \label{sec:homAlg}

Here, we restate, without proof, Lemmas 2.16 and 2.17 of \cite{KWZ19}, which are the main homological algebra tools that we use in the computations that form the heart of our work; see \cref{eg:CiRes} an application of both lemmas. Let $C$ be any bigraded category.

\begin{lem}[Cancellation Lemma]\label{lem:cancel}
Let $(X, \delta)$ be a chain complex in $\cat{Kom}(\cat{Mat}(C))$. Suppose that the object $X$ splits in $\cat{Mat}(\cat{Cob})$ as a direct sum $Z \oplus Y_1 \oplus Y_2$ such that the components of the differential $\delta$ that map $Y_1$ to $Y_2$ form an isomorphism $f \co Y_1 \ra Y_2$ in $\cat{Mat}(\cat{Cob})$. Write the complex $(X, \delta)$ as
\[\begin{tikzcd}[column sep = small]
													&	Z \ar[loop above, "\zeta"] \ar[dl, bend left = 10, "b"]\ar[dr, bend left = 10, "e"]	&	\\
Y_1 \ar[loop left,"\epsilon_1"] \ar[rr, bend left = 10, "f"] \ar[ur, bend left = 10, "a"]	&												& Y_2 \ar[loop right, "\epsilon_2"] \ar[ul, bend left= 10, "c"] \ar[ll, bend left = 10, "e"]
\end{tikzcd},\]
where each arrow is the component of the differential $\delta$ with its domain and codomain restricted. Then the there is a chain-homotopy 
\[(X, \delta)  \simeq (Z, \zeta - af^{-1}e) .\]
\qed
\end{lem}
\begin{rmk}[Comment on \cref{lem:cancel}.] There are two small differences between the lemma as stated here and the original statement. First, our bigraded categories are assumed to have vanishing internal differential ($\del$ in \cite[\S2]{KWZ19}). Second, we do not make use of the category of ``precomplexes". Indeed, this is not needed, as the proof never requires that the morphism $f$ interact in some way with the ``pre-differentials" $\epsilon_1$ and $\epsilon_2$. 
\end{rmk}

\begin{lem}[Clean-up Lemma]\label{lem:cleanUp} Let $(X,d)$ be an object in $\cat{Kom}(\cat{Mat}(C))$. Then for any morphism $\eta \in \Hom((X, d), (X, d))$ of bigrading $(0, 0)$ for which $\eta^2$ and $\eta(d\eta - \eta d)$ both vanish, the complex $(X, d)$ is chain-isomorphic to $(X, d + d\eta - \eta d)$. 
\qed
\end{lem}

\subsection{4-ended tangles and the Bar-Natan algebra $\mc{B}$}\label{sec:BNAlgebra}

We restrict now our attention to $\cat{Cob}_{/l}(4)$, a category whose structure has been clarified in \cite{KWZ19}. We will sketch in this subsection their argument that shows the following. $R$ is still a commutative ring.

\begin{thm}[\S4 of \cite{KWZ19}]\label{thm:skeletalEquiv} There is an equivalence of categories
\[\cat{Kob}(4) \ra \cat{Mod}^\mc{B},\]
where the right-hand side is the category of type D structures over the Bar-Natan algebra $\mc{B}$.
\end{thm} 

Recall that our tangles are pointed. It follows then from the ``boundary-preserving" property of cobordisms in the definition of $\cat{Cob}$ that cobordisms $C$ in $\cat{Cob}_{/l}$ have a unique connected component that contains the marked point $\ast \in B$ (strictly speaking, $C$ contains $\{\ast\} \times I$, but we can think of $\{\ast\}\times I$ as a single point); indeed, the $/l$ relations cannot eliminate the marked point, as it always lies on a non-closed component of $C$. Thus, we can use the marked cobordism as a touchstone to describe all other cobordisms.

\begin{defn}\label{def:dot} Using the marked point $\ast$, we define operators $G, \blacktriangle$ and $\bullet$ on the morphisms of $\cat{Cob}_{/l}$. The operator $G$ adds a handle to the marked component of a cobordism, depicted below. This is the same as the operator $H$ in \cite{Nao06}, but differs by a sign from the convention in \cite{KWZ19}: 
\[G = H_\text{[Nao]} = -H_\text{[KWZ]}.\]
The operator $\blacktriangle$ is a marking on a component of a cobordism that is short-hand for a tube to the marked component (this operator is called $\bullet$ in \cite{Nao06}), and the operator $\bullet$, likewise, is a marking on a component, and it is given by
\[\bullet = \blacktriangle - G.\]
This latter use of $\bullet$ is consistent with the notation in \cite{KWZ19}. These three operators are depicted in \cref{fig:tubeOperators}.
\begin{figure}[h]
	\centering
	\includegraphics[scale=0.6]{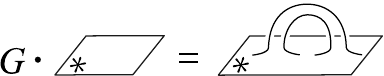}
	\hspace{0.6cm}
	\includegraphics[scale=0.6]{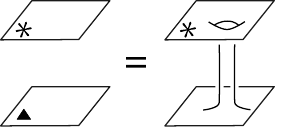}
	\hspace{0.6cm}
	\includegraphics[scale=0.6]{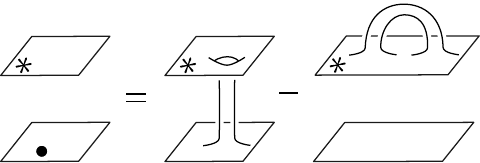}
	\caption{The handle-addition, ``tube-to-mark" and dot cobordisms.}
	\label{fig:tubeOperators}
\end{figure}
\end{defn}

Owing to the following isomorphism, the category $\cat{Mat}(\cat{Cob}_{/l}(4))$ is equivalent to the full subcategory of $\cat{Mat}(\cat{Cob}_{/l}(4))$ generated by the objects $\oRes$ and $\iRes$. This category may be identified with the $\mc{I}$-algebra
\[\mathrm{End}_{/l}\left(\oRes \oplus \iRes\right),\]
where $\mc{I}$ is the subring generated by the identity cobordisms \cite[Observation 2.9]{KWZ19}.

\begin{prop}[\!\! \cite{Nao06}] \label{prop:deLoop} The morphism 
\[L \co \begin{array}{c} \Oset\{-1\}\\ \oplus\\ \Oset\{+1\}\end{array} \xra{\row{\birthTri - G\cdot \birth \hspace{0.4cm}}{\birth}} \bigcirc\]
(in the category $\cat{Mat}(\cat{Cob}_{/l}(B))$), is an isomorphism, with inverse
\[L^{-1} \co \bigcirc \xra{\col{\deathTri \vspace{0.2cm}}{\death}} \begin{array}{c}\Oset\{-1\}\\ \oplus\\ \Oset\{+1\}\end{array}.\]
We call them the looping and delooping isomorphism, respectively.
\end{prop}

To interpret this isomorphism globally, between an object that contains a closed component, such as
\[\mc{O} = \mc{O}' \sqcup \bigcirc,\] 
one extends the disc corbordisms in $L$ and $L^{-1}$ by the identity cobordism $\mc{O'} \times [0,1]$ on the complement of the closed component.

\begin{rmk} The (de)looping isomorphism extends also to an isomorphism in $\cat{Kob}$. 
%This means that if a complex $\br{T}_{/l} \in \cat{Kob}$ is of the form
%\[ \xra{d_{in}} \mc{O}' \sqcup \bigcirc \xra{d_{out}},\]
%where the morphisms $d_{in}$ and $d_{out}$ are the components of the differential that map into and out of $\mc{O}' \sqcup \bigcirc$, then $\br{T}_{/l}$ is isomorphic to
%\[ \xra{L^{-1} \circ d_{in}} \mc{O}' \xra{d_{out} \circ L}.\]
\end{rmk}

Now we are ready to give an overview of the algebraic structure of $\cat{Kob}(4)$. First, we use our definition of the handle-adding operator $G$ to endow $\mathrm{End}_{/l}\left(\oRes \oplus \iRes \right)$ with the structure of an $R[G]$-algebra. Multiplication of two cobordisms is defined to be either their composition, or 0 if the cobordisms are not composable. Second, as Naot shows in \cite{Nao06}, the hom-set $\mathrm{Hom}_{\cat{Cob}_{/l}}(\mc{O}_1, \mc{O}_2)$ is freely generated over $R[G]$ by cobordisms without closed components and such that every open component is contractible and contains at most one marking (either $\bullet$ or $\blacktriangle$). This describes the left-hand category of the equivalence in \cref{thm:skeletalEquiv}. The right-hand is defined in terms of the path algebra of a quiver; see \cite{CB92} for general quiver notions.

\begin{defn} \label{def:quiv} The Bar-Natan algebra $\mc{B}$ is the path algebra (i.e. the $R$-algebra generated by paths, with formal addition, and with multiplication given by concatenation) over the following quiver
\[\begin{tikzcd}
	\bullet \ar[loop left, "D_\bullet"] \rar[bend right, "S_\bullet"']	&\circ \ar[l, bend right, "S_\circ"'] \ar[loop right, "D_\circ"],
\end{tikzcd}\]
subject to the relations
\[D_\circ S_\bullet = S_\bullet D_\bullet = 0 \hspace{0.5cm}\text{ and } \hspace{0.5cm} D_\bullet S_\circ = S_\circ D_\circ = 0.\]
We concatenate paths from right to left, so the path $S_\circ S_\bullet$ starts and ends at $\bullet$, for example. We also define the following quantum grading on $\mc{B}$:
\[q(D_\bullet) = q(D_\circ) = -2 \hspace{0.5cm}\text{ and } \hspace{0.5cm} q(S_\bullet) = q(S_\circ) = -1.\]
 The equation $q(xy) = q(x) + q(y)$ extends the grading to to non-zero concatenations of paths. It is notationally economical to also give elements of $\mc{B}$ a constant homological grading of 0. Finally, we let $1_\bullet$ and $1_\circ$ denote the constant paths at $\bullet$ and $\circ$, respectively.
\end{defn}

We sometimes drop subscripts and write $D$ instead of $D_\bullet$ or $D_\circ$, and likewise for $S$, when the context disambiguates between the two. Thus the relations defining $\mc{B}$ can be rewritten compactly as $DS = SD = 0$.

\begin{rmk}[cf. Observation 2.9 in \cite{KWZ19}] The idempotents $\{1_\bullet, 1_\circ\} \subset \mc{B}$ generate a subring $\mc{I}_\mc{B}$, so that $\mc{B}$ is an $\mc{I}_\mc{B}$-algebra. The quiver in \cref{def:quiv} describes the $\mc{I}_\mc{B}$-algebra $\mc{B}$ as a category enriched over $R$-modules, $\cat{Mod}_R$, with two objects, $\bullet$ and $\circ$. The hom-sets are infinitely generated as $R$-modules, but if we let $G_\circ = S_\bullet S_\circ - D_\circ$ and $G_\bullet = S_\circ S_\bullet - D_\bullet$, then we obtain the following description of the hom-sets as free $R[G_\circ]$- or $R[G_\bullet]$-modules:
\[\begin{split}
	\Hom_\mc{B}(\bullet, \bullet) 	&= 1_\bullet \cdot \mc{B} \cdot 1_\bullet = R[G_\bullet]\gp{1_\bullet, D_\bullet}\\
	\Hom_{\mc{B}}(\bullet, \circ) 	&= 1_\circ \cdot \mc{B} \cdot 1_\bullet = R[G_\circ]\gp{S_\bullet}\\
	\Hom_\mc{B}(\circ, \bullet) 	&= 1_\bullet \cdot \mc{B} \cdot 1_\circ = R[G_\bullet]\gp{S_\circ}\\
	\Hom_\mc{B}(\circ, \circ) 		&= 1_\circ \cdot \mc{B} \cdot 1_\circ = R[G_\circ]\gp{1_\circ, D_\circ}.
\end{split}\]
In fact, taking into account the quantum and homological gradings described in \cref{def:quiv}, $\mc{B}$ is enriched over the category of bigraded $R$-modules.
\end{rmk}

As in the above remark, we think of $\mc{B}$ as both an algebra and a bigraded additive category with two objects.

\begin{notn} The category $\cat{Mod}^\mc{B}$ is defined to be $\cat{Kom}(\mc{B})$, and complexes therein are called type D structures over $\mc{B}$.
\end{notn}

\begin{rmk} If our definition of type D structure seems strange, the reader is encouraged to read Proposition 2.13 and Remark 2.14 of \cite{KWZ19}.
\end{rmk}

\begin{thm}[Theorem 4.21 in \cite{KWZ19}] \label{thm:endo} There is an\footnote{This isomorphism is pinned down by the choice of marked endpoint $\ast$. Other choices yield other isomorphisms.} isomorphism of bigraded categories
\[\mc{B} \xra{\sim} \mathrm{End}_{/l}(\oRes\oplus\iRes),\]
given on objects by
\[\bullet 	\mapsto \oRes \hspace{0.5cm} \circ 	\mapsto \iRes\] 
and on morphisms by
\[\begin{split}
	S_\bullet 	\mapsto \iSad  	&\hspace{0.5cm} 	S_\circ	\mapsto \oSad \\
	D_\bullet	\mapsto \iDot 	&\hspace{0.5cm} 		D_\circ	\mapsto \oDot,
\end{split}\]
where we use the cobordism notation in \cref{def:sadNot}. Moreover, under this isomorphism, the $G$-actions are identified.
\end{thm}

\begin{defn}\label{def:sadNot} A planar diagram consisting of a crossingless tangle $T$ and a red arc connecting two points of $T$ symbolizes an elementary saddle cobordism starting from $T$, with the saddle along the red arc\footnote{The red arc is the core of a 1-handle added in the cell structure determined by a Morse function.}. For example, we have the following, where the left end of the cobordism is the domain, and the right end is the image.
\begin{figure}[h]
	\centering
	\includegraphics[scale=0.3]{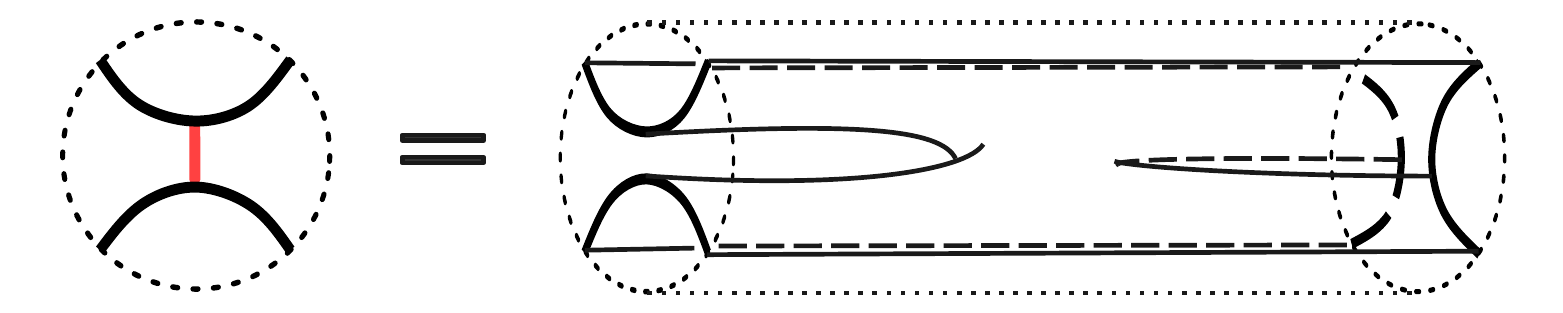}
\end{figure}

Moreover, if $C$ is an elementary saddle cobordism, then we will use $C^k$ to denote the cobordism obtained by taking iterated saddles, as in the following picture:
\begin{figure}[h]
	\centering
	\includegraphics[scale=0.3]{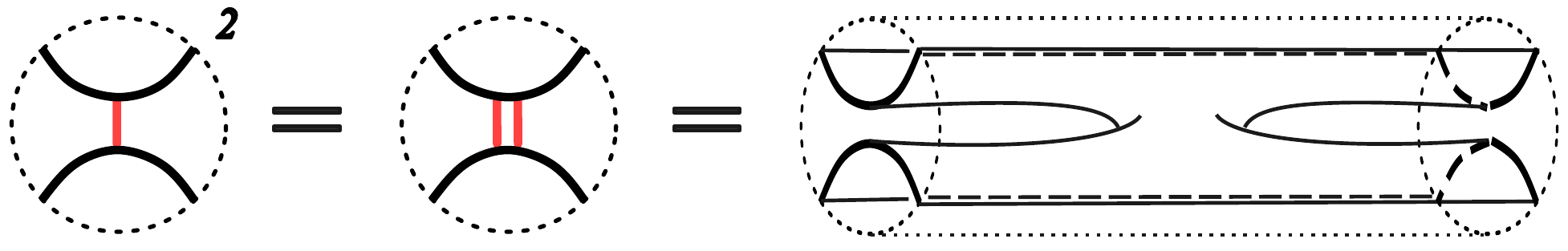}
\end{figure}

Note that in general we have $C^2 \neq C \circ C$; indeed, the saddle $\oSad$ does not have the same domain and codomain, so it cannot be iterated as a morphism in $\cat{Cob}$. Nonetheless, we use this convenient notation, and remark that a pair of parallel red arcs indicates a tube between two sheets in a cobordism. We will also use this notation with dotted cobordisms (\cref{def:dot}), when some component of the tangle is marked with an asterisk $\ast$.  As another example, we have
\begin{figure}[h]
	\centering
	\includegraphics[scale=0.3]{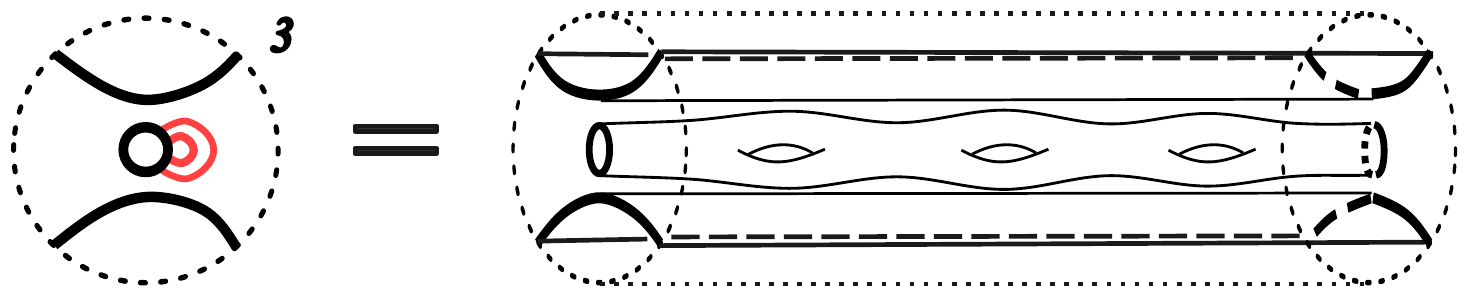}
\end{figure}
\end{defn}

We are now ready to sketch the argument in \cite{KWZ19}.

\begin{proof}[Proof sketch of \cref{thm:skeletalEquiv}] First, by delooping, the inclusion
\[\cat{Kom}\left(\mathrm{End}_{/l}\left(\oRes \oplus \iRes \right)\right) \subset \cat{Kob}(4)\]
is an equivalence of categories. Indeed, it is full and faithful by definition, and delooping provides essential surjectivity. Finally, the isomorphism in \cref{thm:endo} identifies the bigraded additive categories:
\[\mathrm{End}_{/l}\left(\oRes \oplus \iRes\right) \iso \mc{B}.\]
\end{proof}

Finally, we can restate the definition of the tangle invariant in Definition 4.23 of \cite{KWZ19}:

\begin{defn}\label{def:cyrD} \Cref{thm:skeletalEquiv} provides a bigraded isomorphism of categories 
\[\Omega \co \cat{Mod}^\mc{B} \ra \cat{Kom}\left(\cat{End}_{/l}\left(\oRes\oplus\iRes\right)\right.\]
For $T$ a pointed (and, as always, framed) 4-ended tangle, the type D structure $\Rd(T) \in \cat{Mod}^\mc{B}$ is defined as
\[\Rd(T) = \Omega^{-1}(\br{T}_{/l}).\]
Up to homotopy, $\Rd(T)$ is an invariant of the isotopy class of $T$.
\end{defn}

\begin{rmk} Here is how we work with type D structures in practice. Throughout this paper, our type D structures have finite rank, so we describe them as finite labelled directed multi-graphs. The labels on the vertices are $\bullet$ or $\circ$, as they are objects of the category $\mc{B}$ and the arrows are labelled with elements of hom-sets in $\mc{B}$. See for instance \cref{eg:CoRes}.
\end{rmk}

%{\color{blue}
%With our convention for the handle operator, under the isomorphism with the algebra $\mc{B}$, we have $D = SS - G1$, so we have the relations
%\[SSS = S(D+G1) = GS \hspace{0.5cm} \text{ and } \hspace{0.5cm} DD = D(SS - G1) = -GD,\]
%both of which follow from the relation $DS = SD = 0$.
%}

\section{Construction of the operator}\label{sec:construction}

In this section we use Bar-Natan's planar algebra to define a model $\Cb_D$ for the induced operator $\Cb$ on $\cat{Mod}^\mc{B}$. By ``model", we mean precisely that $\Cb_D$ produces explicit representatives of homotopy classes of type D structures.

\subsection{On tangles}

As discussed in the introduction, we define an operator 
\[ \Cb \co \cat{Tan}(4) \ra \cat{Tan}(4)\]
to be the annular tangle $\Cb_T$ depicted in \cref{fig:planAlgOperator}, where the red inner circle is the input (in this section, we will sometimes use the subscript $T$ for the annular tangle). The tangle $\Cb_T$ is oriented so that a cap-trivial input tangle, oriented compatibly with the $\infty$-closure, is taken to a cap-trivial output tangle that is also oriented compatibly with the $\infty$-closure. Finally, $\Cb_T$ has one marked input tangle end and one marked output tangle end. We think of $\Cb$ as a planar algebra operation, which we now review; see \cite[\S5]{BN05}.

Let $s$ be a string of ``in" $(\ina)$ and ``out" $(\outa)$ symbols. The set $\mc{T}^0(s)$ is defined to be the collection of $|s|$-ended oriented tangle diagrams with a marked tangle end, such that the orientations on the tangle ends are given by the string $s$, when reading them counterclockwise, starting from the marked point. Note that the outer marked point will also serve as the basepoint for cobordisms. The set $\mc{T}(s)$ is obtained as a quotient of $\mc{T}^0(s)$ by the Reidemeister moves. Paraphrasing Bar-Natan, a tangle operation 
\[D \co \mc{T}(s_1) \times \dots \mc{T}(s_d) \ra \mc{T}(s)\]
is a crossingless oriented tangle diagram in a big disc with $d$ smaller ``input" discs removed, each of which is marked on the boundary, such that this tangle diagram has $|s_i|$ tangle ends on the $i^{th}$ input circles, oriented according to the string $s_i$, and $|s|$ tangle ends on the big boundary circle, oriented according to the string $s$.

\begin{figure}[h]
	\centering
	\includegraphics[height=.3\linewidth]{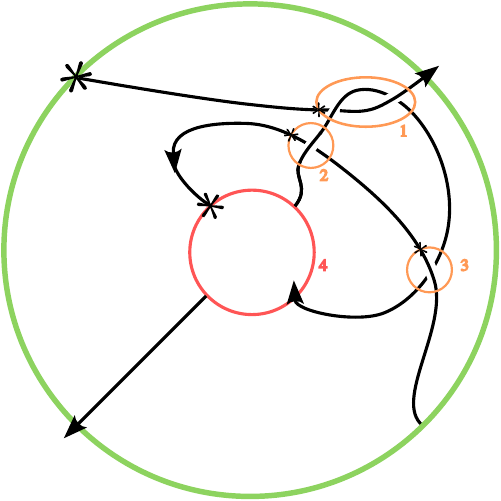}
	\caption{The operator $\Cb$ as a partially filled planar algebra operation $D_{\Cb}(\inputi, \inputii, \inputiii, -)$.}
	\label{fig:planAlgOperator}
\end{figure}

Our operator $\Cb$ is the planar algebra operation 
\[D_{\Cb} \co \mc{T}(\ina\ina\outa\outa) \times \mc{T}(\outa\ina\ina\outa) \times \mc{T}(\outa\outa\ina\ina) \times \mc{T}(\ina\outa\ina\outa) \ra \mc{T}(\ina\outa\ina\outa),\] 
with the first three input discs filled, as in \cref{fig:planAlgOperator}. In other words, $\Cb$ is a function $\mc{T}(\ina\outa\ina\outa) \ra \mc{T}(\ina\outa\ina\outa)$ given by 
\[\Cb(T) = D_\Cb( \inputi, \inputii, \inputiii, T).\]

We will see in the next subsections how this allows us to induce an operator on $\cat{Mod}^\mc{B}$, but first, note that this operator comes in a family. Let $\tau$ be the Dehn-twist tangle operator in \cref{fig:DehnTwist}. Then $\Cb$ is related to $\Cb^0$ by application of $\tau$. If we let $\Cb^n = \tau^n \Cb^0$, then $\Cb = \Cb^{-2}$. We observe finally that the difference between the variously framed operators $\Cb^n$ is captured by a linking number.

\begin{figure}[h]
	\centering
	\includegraphics[height=.2\linewidth]{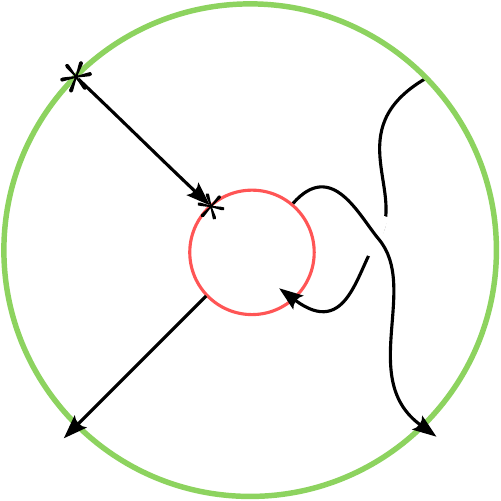}
	\caption{The right half-Dehn twist $\tau$.}
	\label{fig:DehnTwist}
\end{figure}

\begin{defn} Given an oriented 4-ended tangle $T$ without closed components, let its linking number $\lk(T)$ be the number of positive minus the number of negative crossings between the two arcs of $T$.
\end{defn}

\begin{defn} Let $T$ be a 4-ended tangle and let $T'$ be the tangle obtained from $T$ by removing every closed component. The connectivity $\mathrm{conn}(T)$ of $T \in \cat{Tan}(4)$ is the (boundary-preserving) \textit{homotopy} class of $T'$, namely $\connO,\ \connI$ or $\connX$.
\end{defn}

For example, cap-trivial tangles have connectivity $\connX$ or $\connO$. 

\begin{rmk} If $T$ is oriented without closed components and with connectivity $\connO$, then $\lk \Cb^n(T) = n$.
\end{rmk}

\subsection{As a planar algebra operation on $\cat{Kob}$}

In this subsection we use $\Cb_T$ to induce an operator $\Cb_K \co \cat{Kob}(4) \ra \cat{Kob}(4)$. We use in an essential way the planar algebra structure of $\cat{Kob}$ and $\cat{Kob}_{/h}$, which is summarized by the following theorem. %Is it a representation of the little discs operad??

\begin{thm}[\!\!\cite{BN05}, Theorem 2]\label{thm:BNOpAlg}\ 
	\begin{enumerate} 
		\item The collection $\{\cat{Kob}(k)\}_{\{k \in \N\}}$ has a natural planar algebra structure.
		\item The operations $D$ on $\{\cat{Kob}(k)\}$ send homotopy equivalent complexes to homotopy equivalent complexes, so there is also a planar algebra structure on $\cat{Kob}_{/h}$.
		\item The bracket $\br{-}_{/l}$ induces an oriented planar algebra morphism $\{\mc{T}(s)\} \ra \{\cat{Kob}(s)\}$.
	\end{enumerate}
	Moreover, if $D$ is a $d$-input planar arc diagram where the $i^{th}$ input disc has $k_i$ tangle ends, and if $(\Om^i, d^i) \in \cat{Kob}(k_i)$ is a complex for every $i = 1, \dots, d$, then the complex $(\Om, d) = D(\Om^1, \dots, \Om^d)$ is defined, in each (co)homological grading $r$, by
	\[\begin{split}
	\Om_r	&:= \bigoplus_{r_1 + \dots + r_d = r} D(\Om^1_{r_1}, \dots, \Om^d_{r_d})\\
	d|_{D(\Om^1_{r_1}, \dots, \Om^d_{r_d})}
			&:= \sum_{j=1}^d(-1)^{\sum_{j<i}r_j} D(\Id_{\Om^1_{r_1}}, \dots, d^i, \dots, \Id_{\Om^d_{r_d}}).
	\end{split}\]
\end{thm}

We interpret the theorem thus: if $D$ is a tangle diagram, then we can compute a chain homotopy representative of $\br{D}_{/l}$ by cutting $D$ up into subtangles $T_i$, computing chain homotopy representatives of $\br{T_i}_{/l}$, and then piecing these together. The reader who is unfamiliar with \cite{BN05} and wishes to understand the theorem statement is advised to read the first half of that paper. For the reader who is familiar, \cref{sec:notnReview} contains a rapid review of the notation.

\begin{defn} We define the operator 
\[\Cb_K \co \cat{Kob}(4) \ra \cat{Kob}(4)\]
in terms of the 4-input planar arc diagram from the previous subsection:
\[\Cb_K(\Omega) := D_{\Cb}\left(\br{\inputi}_{/l}, \br{\inputii}_{/l}, \br{\inputiii}_{/l}, \Omega\right).\]
\end{defn}

In our case, up to homotopy, we have
\[\begin{split}
	\brl{\inputi}	&= \prescript{\color{ForestGreen}-5}{}{\iRes_{-2}} \xra{\iSS - \iH} \prescript{-3}{}{\iRes_{-1}} \xra{\iSad} \prescript{-2}{}{\oRes_0}\\
	\brl{\inputii}	&= \prescript{1}{}{\iRes_0} \xra{\iSad} \prescript{2}{}{\oRes_1} \\
	\brl{\inputiii}	&= \prescript{1}{}{\oRes_0} \xra{\oSad} \prescript{2}{}{\iRes_1}.
\end{split}\]
See Example 4.27 in \cite{KWZ19} for a proof. We let $(\Om^1, d^1) = \brl{\inputi}, (\Om^2, d^2) = \brl{\inputii}$ and $(\Om^3, d^3) = \brl{\inputiii}$. 

%Indeed, the above claim is justified by passing to the category $\cat{Mod}^\mc{B}$, where we have the homotopy equivalences 
%\[\Rd(\inputi)  \htp \prescript{-5}{}{\circ_{-2}} \xra{D_\circ} \prescript{-3}{}{\circ_{-1}} \xra{S_\circ} \prescript{-2}{}{\bullet_0},\]
%and the equalities
%\[\begin{split}
%	\Rd(\inputii)	& = \prescript{1}{}{\circ_0} \xra{S_\circ} \prescript{2}{}{\bullet_1}, \\
%	\Rd(\inputiii)	& = \prescript{1}{}{\bullet_0} \xra{S_\bullet} \prescript{2}{}{\circ_1}.
%\end{split}\]

Consider now a complex $(\Om^4, d^4) = \brl{T} \in \cat{Kob}(4)$ that is to be placed in the 4$^{th}$ input disc of $D_\Cb$. By \cref{thm:BNOpAlg}, we have
\[\Cb_K\left(\brl{T}\right)_r = \bigoplus_{r_1 + r_2 + r_3 + r_4 = r} D_{\Cb}(\Om^1_{r_1}, \Om^2_{r_2}, \Om^3_{r_3}, \Om^4_{r_4})\]
and
\[d|_{D(\Om^1_{r_1}, \Om^2_{r_2}, \Om^3_{r_3}, \Om^d_{r_d})} = \sum_{i=1}^4(-1)^{\left(\sum_{j<i}r_j\right)} D(\Id_{\Om^1_{r_1}}, \dots, d^i, \dots, \Id_{\Om^d_{r_d}}).\]

Thus the four components of the differential, corresponding to $i = 1,2,3,4$ in the above equation, are given, respectively, by
\[\begin{split}
D_\Cb \left(d^1, \Id_{\Om^2_{r_2}}, \Id_{\Om^3_{r_3}}, \Id_{\Om^4_{r_4}}\right)					&\co D_\Cb \left(\Om^1_{r_1}, \Om^2_{r_2}, \Om^3_{r_3}, \Om^4_{r_4}\right) \ra D_\Cb\left(\Om^1_{r_1+1}, \Om^2_{r_2}, \Om^3_{r_3}, \Om^4_{r_4}\right),\\
(-1)^{r_1} D_\Cb\left(\Id_{\Om^1_{r_1}}, d^2, \Id_{\Om^3_{r_3}}, \Id_{\Om^4_{r_4}}\right)			&\co D_\Cb \left(\Om^1_{r_1}, \Om^2_{r_2}, \Om^3_{r_3}, \Om^4_{r_4}\right) \ra D_\Cb \left(\Om^1_{r_1}, \Om^2_{r_2+1}, \Om^3_{r_3}, \Om^4_{r_4}\right),\\
(-1)^{r_1+r_2} D_\Cb \left(\Id_{\Om^1_{r_1}}, \Id_{\Om^2_{r_2}}, d^3, \Id_{\Om^4_{r_4}}\right) 		&\co D_\Cb \left(\Om^1_{r_1}, \Om^2_{r_2}, \Om^3_{r_3}, \Om^4_{r_4}\right) \ra D_\Cb \left(\Om^1_{r_1}, \Om^2_{r_2}, \Om^3_{r_3+1}, \Om^4_{r_4}\right),\\
(-1)^{r_1+r_2+r_3} D_\Cb \left(\Id_{\Om^1_{r_1}}, \Id_{\Om^2_{r_2}}, \Id_{\Om^3_{r_3}}, d^4\right) 	&\co D_\Cb \left(\Om^1_{r_1}, \Om^2_{r_2}, \Om^3_{r_3}, \Om^4_{r_4}\right) \ra D_\Cb \left(\Om^1_{r_1}, \Om^2_{r_2}, \Om^3_{r_3}, \Om^4_{r_4+1}\right).
\end{split}\]

Note that the power of $(-1)$ that multiplies the $4^{th}$ component of the differential does not depend on the complex $(\Omega^4, d^4)$.

We keep track of the signs in $\Cb(\Omega^4)$ by using the following diagram. There is a vertex for every triple $(r_1, r_2, r_3)$ for which each $\Om^i_{r_i}$ is nonzero, and there is an edge for each of the first three components of the differential, labelled by the map $(-1)^{\sum_{j<i}r_j}d^i$, for $i = 1, 2, 3$. A vertex's label is coloured magenta if $r_1+r_2+r_3$ is odd, in which case there is a factor of $-1$ multiplying the 4$^{th}$ component of the differential.

\[\begin{tikzcd}[cramped, sep=small]
(-2,0,0) \ar[rr, "D"] \ar[dr, "S"] \ar[dd, "S"]	&													& \mg{(-1,0,0)} \ar[dd, "-S" near start] \ar[dr, "-S"] \ar[rr, "S"]	&													& (0,0,0)	\ar[dr, "S"] \ar[dd,"S" near start]	&							\\
								&\mg{(-2, 1, 0)}  \ar[rr, "D" near start, crossing over]				&												& (-1, 1, 0) \ar[rr,"S" near start, crossing over]						&								& \mg{(0, 1, 0)} \ar[dd, "-S" near start]\\
\mg{(-2, 0, 1)}\ar[rr,"D" near start]\ar[dr,"S"]&													& (-1, 0, 1) \ar[dr,"-S"]\ar[rr,"S" near start]					&													& \mg{(0, 0, 1)} \ar[dr,"S"]				&							\\
								& (-2, 1, 1) \ar[rr,"D"]	\ar[from = uu, "-S" near start, crossing over]	&												& \mg{(-1, 1, 1)} \ar[from = uu,"S" near start, crossing over]\ar[rr, "S"]	&								& (0, 1, 1)
\end{tikzcd}\]

Placing the corresponding smoothings at the vertices $(r_1, r_2, r_3)$ of the above diagram, we can thus represent the operator by the following diagram, (where we also include the bigrading on the top-left object)\footnote{Specifying the bigrading at one vertex of a connected complex fixes it throughout, by \cref{thm:qDegZero}.}:

\[\begin{tikzcd}[cramped, sep=small]
\prescript{-3}{}{\opT_{-2}} \ar[rr, "D"] \ar[dr, "S"] \ar[dd, "S"]	&												& \opT \ar[dd, "-S" near start] \ar[dr, "-S"] \ar[rr, "S"]	&												& \opTL\ar[dr, "S"] \ar[dd, "S" near start]		&						\\
											&\opX  \ar[rr, "D" near start, crossing over]					&										& \opX \ar[rr, "S" near start, crossing over]						&									& \opT \ar[dd, "-S" near start]	\\
\opTt	 \ar[rr, "D" near start]\ar[dr, "S"]					&												& \opTt \ar[dr, "-S"]\ar[rr, "S" near start]			&												& \opT \ar[dr, "S"]						&						\\
											& \opR \ar[rr, "D"]\ar[from=uu, "-S" near start, crossing over]	&										& \opR \ar[from=uu, "S" near start, crossing over]\ar[rr, "S"]	&									& \opRt
\end{tikzcd}\]

Note that the cobordisms in the above diagram are not a priori well-defined intrinsically: they are shorthand for the morphisms in the complexes $\Omega^i$, but one still needs to look back at the planar algebra $D_{\Cb}$ to learn how the $\Omega^i$ sit inside each smoothing. However, with a little thought, one notices that the $S$ morphisms indicated are in fact uniquely determined (except for the ones involving the top-right corner) and the topmost $D$ vanishes, so we only need to be explicit about the three remaining $D$ morphisms. Doing this and delooping the top-right corner results in the following diagram:

\[\begin{tikzcd}[cramped, sep=small]
\prescript{-3}{}{\opT_{-2}} \ar[dr, "S"] \ar[dd, "S"]	&												& \opT \ar[dd, "-S" near start] \ar[dr, "-S"] \ar[rr, "\col{\text{Id}}{\ast}"]	&												& \arrP{\opT}\ar[dr, "\row{0}{\text{Id}}"] \ar[dd, "\row{\ast}{\text{Id}}" near start]		&						\\
											&\opX  \ar[rr, "\opDX" near start, crossing over]					&										& \opX \ar[rr, "S" near start, crossing over]					&									& \opT \ar[dd, "-S" near start]	\\
\opTt	 \ar[rr, "\opDTt" near start]\ar[dr, "S"]					&												& \opTt \ar[dr, "-S"]\ar[rr, "S" near start]			&												& \opT \ar[dr, "S"]						&						\\
											& \opR \ar[rr, "\opDR"]\ar[from=uu, "-S" near start, crossing over]	&										& \opR \ar[from=uu, "S" near start, crossing over]\ar[rr, "S"]	&									& \opRt
\end{tikzcd}\]

Applying the cancellation lemma to the isomorphisms $\text{Id}$ that appeared as a result of delooping yields the following diagram:

\[\begin{tikzcd}[cramped, sep=small]
\prescript{-3}{}{\opT_{-2}} \ar[dr, "S"] \ar[dd, "S"]	&									& 								&											& 				&	\\
									&\opX  \ar[rr, "\opDX"]					&								& \opX \ar[dr, bend left, "-S"]						&				& 	\\
\opTt	 \ar[rr, "\opDTt" near start]\ar[dr, "S"]		&									& \opTt \ar[dr, "-S"]\ar[rr, "S" near start]	&											& \opT \ar[dr, "S"]		&	\\
									& \opR \ar[rr, "\opDR"]\ar[from=uu, "-S" near start, crossing over]&					& \opR \ar[from=uu, "S" near start, crossing over]\ar[rr, "S"]&				& \opRt
\end{tikzcd}\]

For typographic purposes, we relabel the objects in the above diagram according to the dictionary in \cref{fig:resObjs}. 

\begin{figure}[h]
\[\begin{array}{cccccccc}
	\includegraphics[scale=0.12]{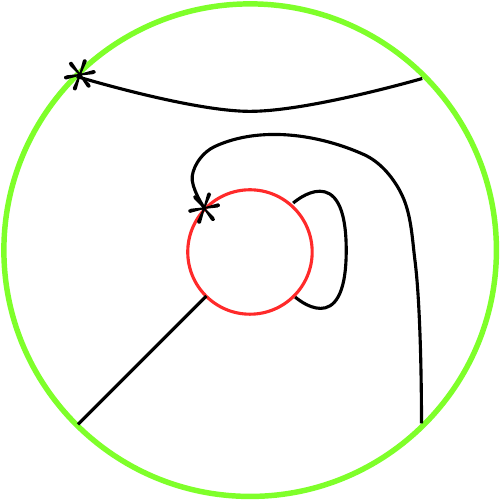}		& \includegraphics[scale=0.12]{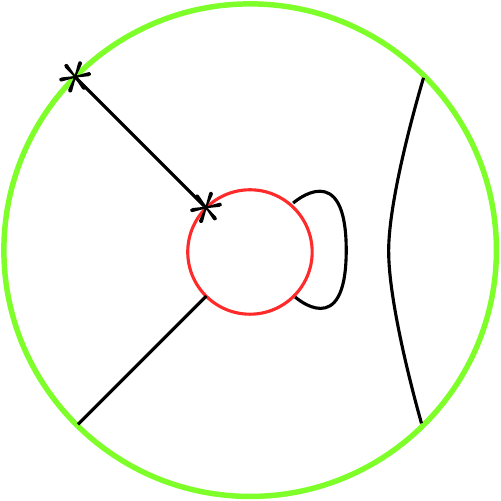}		& \includegraphics[scale=0.12]{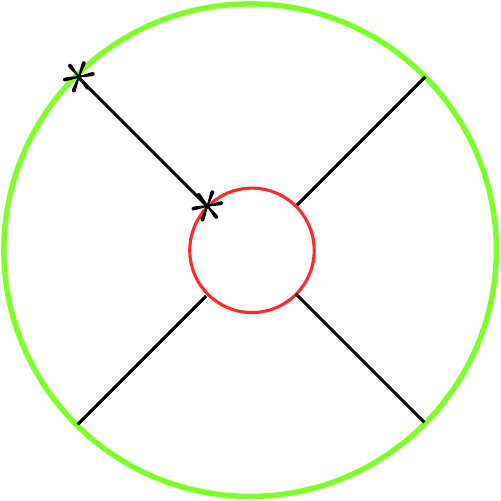}		& \includegraphics[scale=0.12]{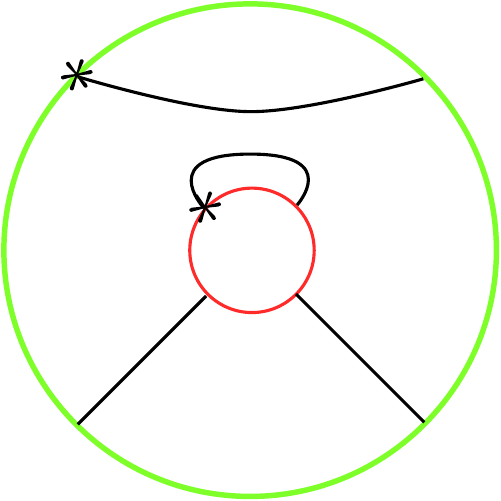}		& \includegraphics[scale=0.12]{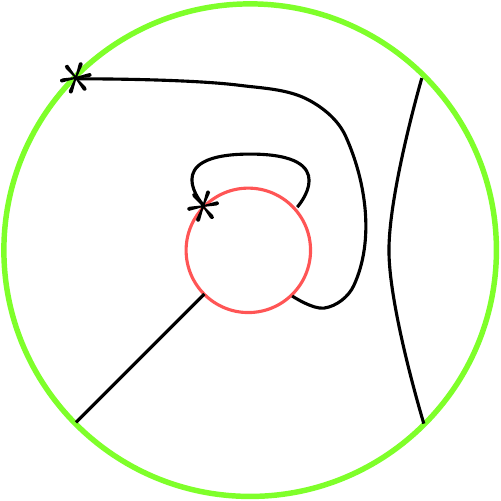}	&\includegraphics[scale=0.12]{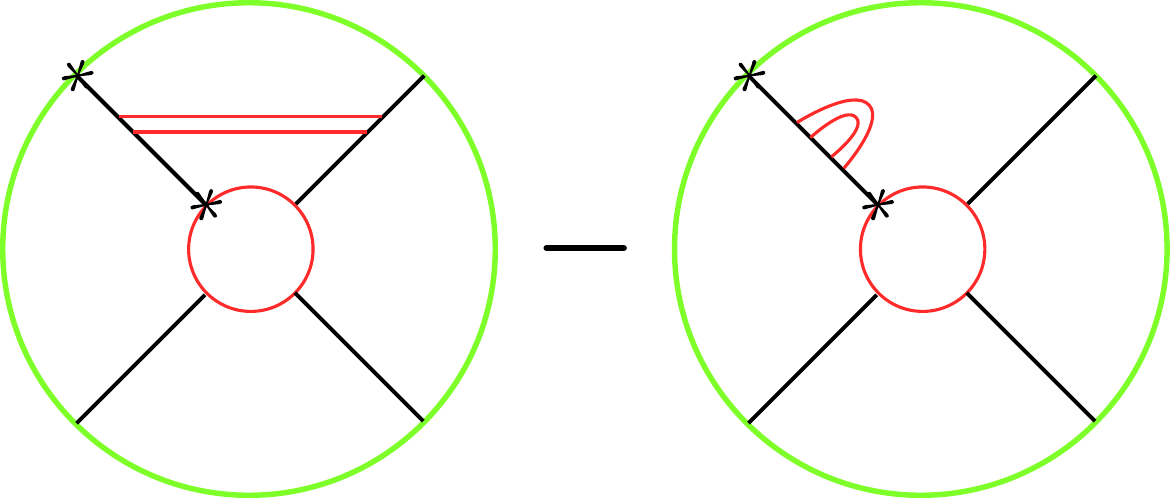}	&\includegraphics[scale=0.12]{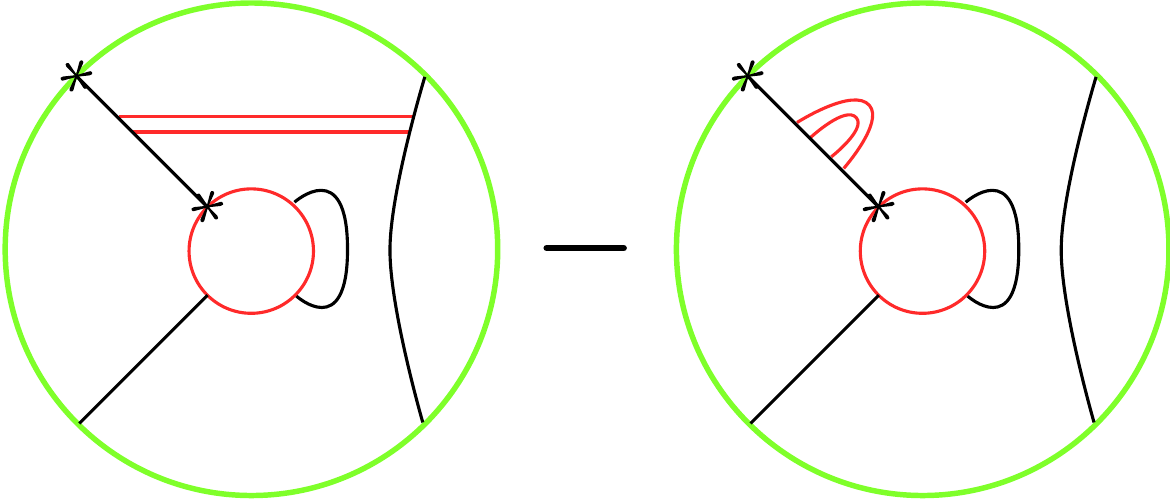}	& \includegraphics[scale=0.12]{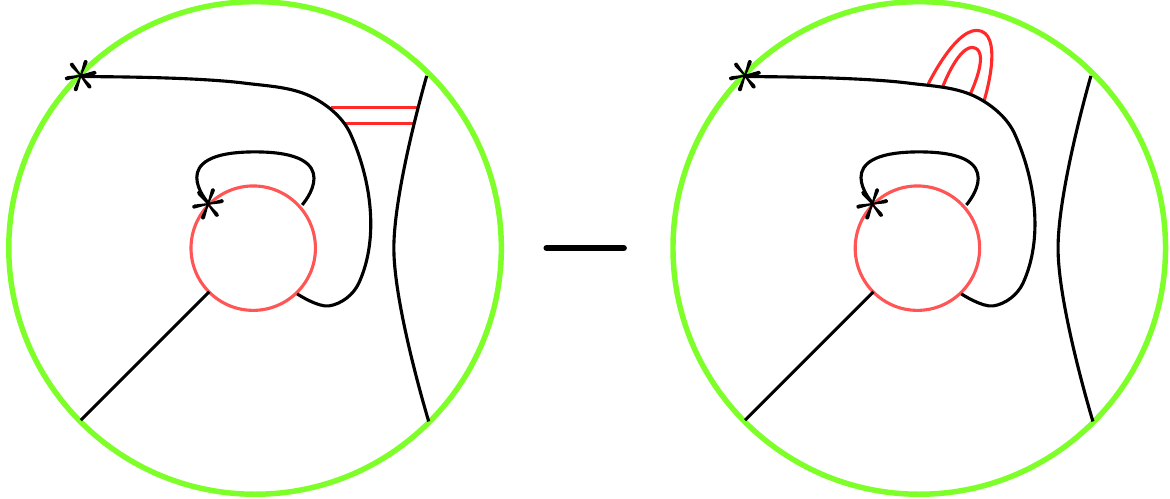} \\
	\vspace{0.1cm}
	\wt{R}		& R			& X			& T			& \wt{T}	&D_X	& D_R	& D_{\wt{T}} 
\end{array}\]
	\caption{Notation for the smoothings in $\br{\Cb_T}$ and the three $D$ morphisms therein.}
	\label{fig:resObjs}
\end{figure}

Rearranging, we have the following isomorphic representative of $\Cb_K$, where again, the magenta coloured objects represent the oddness of the sum $r_1+r_2+r_3$ at the vertex $(r_1,r_2,r_3)$:

\begin{equation}\label{eq:operator}
\Cb_K =
\begin{tikzcd}
\prescript{-3}{}{T_{-2}} \rar["S"] \ar[ddr, "S"]& \prescript{-2}{}{\mg{X}_{-1}}\rar["D_X"] \dar["{-S}"]	& \prescript{0}{}{X_0} \dar["S"] \ar[rdd, "{-S}", bend left = 90]	&		\\
								& \prescript{-1}{}{R_0} \rar["D_R"]						& \prescript{1}{}{\mg{R}_1} \rar["S"]					& \prescript{2}{}{\wt{R}_2}	\\
								& \prescript{-2}{}{\mg{\wt{T}}_{-1}} \uar["S"'] \rar["D_{\wt{T}}"]		& \prescript{0}{}{\wt{T}_0} \uar["{-S}"'] \rar["S"]	& \prescript{1}{}{\mg{T}_1} \uar["S"']
\end{tikzcd}.
\end{equation}

Finally, to compute $\Cb_K(\brl{T})$, we must compute $\Cb_K(C)$ for every cobordism $C$ isomorphic to path algebra element in $\mc{B}$. We do this in \cref{fig:cobTable} below.

\begin{figure}[h] 
\[\begin{array}{c | c c c c c}
	\begin{array}{cc}
			& \text{Obj}\\
	\text{Input}& \end{array}	&	\wt{R}		&	R		&	X	&	T		&	\wt{T}	\\
	\hline
	\oRes				& \oRes 			& \iRes		& \oRes		& \oTangLoop		& \iTangLoop\\
	\iRes					& \oTangLoop		& \iTangLoop	& \iRes		& \oRes			& \iRes		\\
	S_\bullet				& \oSadD			& \iSadL		& \oSad		& \oLpSadLpD		& \iLpSadLpL 	\\
	S_\circ				& \oLpSadLpD		& \iLpSadLpL	& \iSad		& \oSadD			& \iSadL		\\
	D_\bullet				& 0				& 0			& \oDot		& \oLpDotDLp		& \iLpDotLLp	\\
	D_\circ				& \oLpDotLpD		& \iLpDotLpL	& \iDot		& 0				& 0	\\
	SS_\bullet				& \oSadD^2		& \iSadL^2	& \oSad^2		& \oLpSadLpD^2	& \iLpSadLpL^2\\
	SS_\circ				& \oLpSadLpD^2	& \iLpSadLpL^2& \iSad^2		& \oSadD^2		& \iSadL^2	\\
	D_\bullet^k			& 0				& 0			& \oDot^k		& \oLpDotDLp^k	& \iLpDotLLp^k	\\
	D_\circ^k				& \oLpDotLpD^k	& \iLpDotLpL^k	& \iDot^k		& 0				& 0	\\
\end{array}\]
	\caption{Table of outputs of the resolution objects of $\Cb$. }
	\label{fig:cobTable}
\end{figure}

\textbf{A word of warning: the marked points describing some of the cobordisms in \cref{fig:cobTable} should be interpreted as convenient shorthand for linear combinations of cobordisms; the true basepoint is always on the top-left tangle end. For example, one should interpret $\oLpDotLpD$ as being equal to the linear combination $\moLpSSLpD - \moLpSSDD$. This will be important when we deloop, as we will need a consistent basepoint for all outputs of the operator.
}

Let us now perform some sanity checks.

\begin{ex}\label{eg:CoRes} It is easy to see that
\[\Cb_T(\oResOr) = \inputi,\]
and we have
\[\brl{\inputi} = 
\begin{tikzcd}[cramped]
	\prescript{-4}{}{\oRes_{-2}} \rar{\oSad}	& \prescript{-3}{}{\iRes_{-1}} \rar{\iDot}	& \prescript{-1}{}{\iRes_0}
\end{tikzcd}.\]
Let us verify that this agrees with $\Cb_K\left(\prescript{0}{}{\oRes_0}\right)$. According to \cref{eq:operator} and  \cref{fig:cobTable}, we have
\[\Cb_K\left(\prescript{q}{}{\oRes_h}\right) = 
\begin{tikzcd}[ampersand replacement=\&, sep=large]
\prescript{q-3}{}{\oTangLoop_{h-2}} \rar["\oLpSadUp"] \ar[ddr, "\oSadLp"']	\& \oRes \dar["-\oSad"]					\& \oRes \dar["\oSad"] \ar[ddr, bend left = 100, "-\oSadU"]	\& \\
										\& \iRes  \rar["\iDot"]						\& \iRes \rar["\iSad"]								\& \oRes\\
										\& \iTangLoop \ar[u, "\iLpSadLpL"']\rar["\iDotLp"]\& \iTangLoop \ar[u, "-\iLpSadLpL"']\rar["\iSadLp"]		\& \oTangLoop \ar[u, "\oLpSadLpD"']
\end{tikzcd}\]

By delooping, we obtain the following isomorphic complex:

\[
\Cb_K\left(\prescript{q}{}{\oRes_h}\right) \iso
	\begin{tikzcd}[ampersand replacement=\&, sep=large]
\arrPGr{\oRes}{h-2}{q-4}{h-2}{q-2} \rar["\oLpSadUp \circ L"] \ar[ddr, "L^{-1} \circ \oSadLp \circ L"']\& \oRes \dar["-\oSad"]									\& \oRes \dar["\oSad"] \ar[ddr, bend left = 100, "-L^{-1}\circ\oSadU"]	\& \\
														\& \iRes  \rar["\iDot"]										\& \iRes \rar["S"]										\& \oRes\\
														\& \arrP{\iRes} \ar[u, "\iLpSadLpL \circ L"']\rar["L^{-1} \circ \iDotLp \circ L"]\& \arrP{\iRes} \ar[u, "-\iLpSadLpL\circ L"']\rar["L^{-1}\circ \iSadLp \circ L"]	\& \arrP{\oRes} \ar[u, "\oLpSadLpD \circ L"']
	\end{tikzcd}\]
By applying the $/l$ relations, we can express all of the cobordisms in the diagram above in terms of the chosen basis of saddles and dots (where, to make sense of the symbols $D$ and $G$, we implicitly assume each top-left tangle end to be marked $\ast$). The result is
\[\begin{tikzcd}[ampersand replacement=\&]
\arrPGr{\oRes}{h-2}{q-4}{h-2}{q-2} \rar["\row{0}{1}"] \ar[ddr, "{\mat{S}{0}{0}{S}}"']		\& \oRes \dar["-S"]							\& \oRes \dar["S"] \ar[ddr, bend left = 100, "{\col{-1}{-G1}}"]	\& \\
										\& \iRes  \rar["D"]							\& \iRes \rar["S"]								\& \oRes\\
										\& \arr{\iRes} \ar[u, "(0\ 1)"']\rar["{\mat{D}{0}{0}{D}}"]	\& \arr{\iRes} \ar[u, "(0\ -1)"']\rar["{\mat{S}{0}{0}{S}}"]		\& \arr{\oRes} \ar[u, "(D\ 1)"']		
\end{tikzcd},\]
where we use the algebra elements $D$ and $S$ to represent honest cobordisms in $\cat{Cob}$. Let us gain notational alignment by applying the isomorphism $\Omega$ in \cref{def:cyrD} to work in the less cumbersome category $\cat{Mod}^\mc{B}$, and note the boxed identity morphisms, which we will soon get rid of:
\[\Om^{-1}\left( \Cb_K\left(\prescript{q}{}{\oRes_{h}}\right)\right) =
\begin{tikzcd}[ampersand replacement=\&]
\arrPGr{\bullet}{h-2}{q-4}{h-2}{q-2}\rar["\row{0}{\color{blue}\boxed{1}}"] \ar[ddr, "{\mat{S}{0}{0}{S}}"']	\& \bullet \dar["-S"]												\& \bullet \dar["S"] \ar[ddr, bend left = 100, "{\col{-1}{-G1}}"]			\& \\
												\& \circ  \rar["D"]												\& \circ \rar["S"]												\& \bullet\\
												\& \arPBb{\circ} \ar[u, "\row{0}{\color{blue}\boxed{1}}"']\rar["{\mat{D}{0}{0}{D}}"]	\& \arPBb{\circ} \ar[u, "\row{0}{\color{blue}\boxed{-1}}"']\rar["{\mat{S}{0}{0}{S}}"]	\& \arPBb{\bullet} \ar[u, "\row{D}{\color{blue}\boxed{1}}"']	
\end{tikzcd}\]
We can apply the cancellation lemma to the boxed morphisms to obtain the homotopy
\[\Om^{-1}\left( \Cb_K\left(\prescript{q}{}{\oRes_{h}}\right)\right) \simeq
\begin{tikzcd}[ampersand replacement=\&]
\prescript{q-4}{}{\bullet_{h-2}} \ar[ddr, "S"']	\&				\& {\color{blue}\boxed{\bullet}} \ar[ddr, bend left = 100, "{\color{blue}\boxed{-1}}"]	\& \\
								\& 				\&									\& \\
								\& \circ \rar["D"]		\& \circ \rar["S"]							\& \bullet 
\end{tikzcd}\]
Applying the cancellation lemma once again results in
\[\begin{tikzcd}[cramped, sep=small] 
\prescript{q-4}{}{\bullet_{h-2}} \rar{S}	& \prescript{q-3}{}{\circ_{h-1}} \rar{D}	& \prescript{q-1}{}{\circ_h},
\end{tikzcd}\]
which, after setting $q=h=0$, is isomorphic under $\Om$ to
\[\begin{tikzcd}[cramped]
	\prescript{-4}{}{\oRes_{-2}} \rar{\oSad}	& \prescript{-3}{}{\iRes_{-1}} \rar{\iDot}	& \prescript{-1}{}{\iRes_0},
\end{tikzcd}\]
as expected. Note that it would not have been possible to apply the cancellation lemma in one shot to all of the identity morphisms above because, in that lemma's wording, the collection of maps from $Y_1$ to $Y_2$ would have contained an extra non-identity morphism (the second vertical $S_\bullet$ in the top row of morphisms), and so the cancellation lemma's hypotheses would not have been satisfied.
\qed
\end{ex}

\begin{ex}\label{eg:CiRes}
Likewise, we can check that
\[\Cb_K\left(\prescript{0}{}{\iRes_0}\right) \simeq 
	\begin{tikzcd}
									& \prescript{0}{}{\iRes_0} \rar["\iDot"]		& \prescript{2}{}{\iRes_1} \rar["\iSad"]		& \prescript{3}{}{\oRes_2}	\\
	\prescript{-3}{}{\oRes_{-2}} \rar["\oSad"]	& \prescript{-2}{}{\iRes_{-1}} \rar["\iDot"]	& \prescript{0}{}{\iRes_0}				&
	\end{tikzcd},
\]
which is the Bar-Natan bracket invariant of 
\[\Cb_T\left(\iResOr\right) = \hopfTangleOr.\]
Again, we compute in $\cat{Mod}^\mc{B}$, and we blue-box the isomorphism that we apply the cancellation lemma to.
\[\begin{split}
\Om^{-1}\left(\Cb_K\left(\prescript{0}{}{\circ_0}\right)\right)	
&= \begin{tikzcd}[ampersand replacement=\&]
\prescript{-3}{}{\bullet_{-2}} \rar["S"] \ar[ddr, "S"]\& {\color{blue}\boxed{\circ}} \rar["D"] \dar["\col{\color{blue}\boxed{-1}}{-SS}"]		\& {\color{blue}\boxed{\circ}} \ar[ddr, "-S", bend left = 100] \dar["\col{\color{blue}\boxed{1}}{SS}"]	\& \\
						\& \arrP{\circ} \rar["{\mat{D}{0}{0}{D}}"]	\& \arrP{\circ} \rar["{\mat{S}{0}{0}{S}}"]				\& \arrP{\bullet} \\
						\& \circ \uar["\col{1}{G1}"'] \rar["D"]		\& \circ \uar["\col{-1}{-G1}"'] \rar["S"]					\& {\color{blue}\boxed{\bullet}} \uar["\col{\color{blue}\boxed{1}}{SS}"'] 
	\end{tikzcd}\\
&\simeq
\begin{tikzcd}[ampersand replacement=\&]
\prescript{-3}{}{\bullet_{-2}} \ar[ddr, "S"]	\&									\&							\& \\
					\& \circ \rar["D"]							\& \circ \rar["S"]					\& \bullet \\
					\& \circ \uar["G1 {\color{blue}-SS}"'] \rar["D"]	\& \circ \uar["-G1 {\color{blue}+SS}"'] 	\&
\end{tikzcd}\\
&=
\begin{tikzcd}[ampersand replacement=\&]
\prescript{-3}{}{\bullet_{-2}} \ar[ddr, "S"]	\&					\&				\& \\
					\& \circ \rar["D"]			\& \circ \rar["S"]		\& \bullet \\
					\& \circ \uar["-D"'] \rar["D"]	\& \circ \uar["D"'] 	\&
\end{tikzcd}\\
\end{split}\]
Finally, an application of the clean-up lemma along the green diagonal arrow in the diagram
\[\begin{tikzcd}[ampersand replacement=\&]
	\prescript{-3}{}{\bullet_{-2}} \ar[ddr, "S"]	\&					\&				\& \\
									\& \circ \rar["D"]			\& \circ \rar["S"]		\& \bullet \\
									\& \circ \uar["-D"'] \rar["D"]	\& \circ \uar["D"'] 	\&
	\ar[Leftarrow, from=3-3, to=2-2, ForestGreen, "-1"']
\end{tikzcd}\]
results in the following complex
\[\begin{tikzcd}[cramped, sep=small]
								& \circ			& \circ \rar["S"]					&\bullet\\
	\prescript{-3}{}{\bullet_{-2}} \rar["S"] 	& \circ \uar["-D"']	&\prescript{0}{}{\circ_0} \uar["D"']	&
\end{tikzcd},\]
which, we remark, is isomorphic to
\[\begin{tikzcd}[cramped, sep=small]
								& \circ			& \circ \rar["S"]				&\bullet\\
	\prescript{-3}{}{\bullet_{-2}} \rar["S"] 	& \circ \uar["D"']	&\prescript{0}{}{\circ_0} \uar["D"']	&
\end{tikzcd}.\]
\qed
\end{ex}

\subsection{On type D structures}\label{sec:typeDOperator}

\begin{defn} The operator $\Cb \co \cat{Mod}^\mc{B} \ra \cat{Mod}^\mc{B}$ is defined as the composite
\[\cat{Mod}^\mc{B} \xra{\Om} \cat{Kom}\left(\cat{End}_{/l}\left(\oRes \oplus \iRes \right) \right) \xra{\Cb_K} \cat{Kob}(4) \ra \cat{Mod}^\mc{B},\]
where the last functor is the equivalence of categories from \cref{thm:skeletalEquiv}.
\end{defn}

This subsection consists in the construction of a particular model, that we call $\Cb_D$, for the operator $\Cb$. Morally, this is nothing other than $\Om^{-1} \circ \Cb_K \circ \Omega$, but this expression does not make sense, since $\Cb_K$ is not guaranteed to land inside the category $\cat{Kom}\left(\cat{End}_{/l}\left(\oRes \oplus \iRes\right)\right)$. We systematically pass to this category by delooping. First, on objects, examples \ref{eg:CoRes} and \ref{eg:CiRes} produce isomorphic representatives for $\Cb_K\left(\oRes\right)$ and $\Cb_K\left(\iRes\right)$, which we take as definitions for $\Cb_D$.

\begin{defn} We define the following objects in $\cat{Mod}^\mc{B}$:
\begin{equation}\label{eq:CObj}
\begin{split}
\Cb_D(\prescript{q}{}{\bullet_h}) &:= 
	\begin{tikzcd}[ampersand replacement=\&]
\begin{array}{c} \prescript{q-4}{}{\bullet_{h-2}} \\ \oplus \\  \prescript{q-2}{}{\bullet_{h-2}}\end{array} \rar["\row{0}{1}"] \ar[ddr, "{\mat{S}{0}{0}{S}}"']	\& \bullet \dar["-S"]								\& \bullet \dar["S"] \ar[ddr, bend left = 100, "{\col{-1}{-G1}}"]	\& \\
											\& \circ  \rar["D"]								\& \circ \rar["S"]										\& \bullet\\
											\& \arrP{\circ} \ar[u, "\row{0}{1}"']\rar["{\mat{D}{0}{0}{D}}"]	\& \arrP{\circ} \ar[u, "\row{0}{-1}"']\rar["{\mat{S}{0}{0}{S}}"]		\& \arrP{\bullet} \ar[u, "\row{D}{1}"']	
	\end{tikzcd}\\
\Cb_D( \prescript{q}{}{\circ_h}) &:= 
	\begin{tikzcd}[ampersand replacement=\&]
	 \prescript{q-3}{}{\bullet_{h-2}} \rar["S"] \ar[ddr, "S"]\& \circ \rar["D"] \dar["\col{-1}{-SS}"]		\& \circ \ar[ddr, "-S", bend left = 100] \dar["\col{1}{SS}"]	\& \\
						\& \arrP{\circ} \rar["{\mat{D}{0}{0}{D}}"]	\& \arrP{\circ} \rar["{\mat{S}{0}{0}{S}}"]				\& \arrP{\bullet} \\
						\& \circ \uar["\col{1}{G1}"'] \rar["D"]		\& \circ \uar["\col{-1}{-G1}"'] \rar["S"]					\& \bullet \uar["\col{1}{SS}"'] 
	\end{tikzcd}
\end{split}
\end{equation}
Let now $z$ be a path in the path algebra $\mc{B}$, i.e. a string of composable $S$ and $D$ morphisms in the category $\mc{B}$, and let $x, y$ be the domain and codomain of $z$, respectively. We define $\Cb_D(z) \co \Cb_D(x) \ra \Cb_D(y)$ to be the morphism making the following diagram commute
\[\begin{tikzcd}
	\Cb_K(\Om(x))	\rar["\Cb_K(z)"] \dar["\iso"]	& \Cb_K(\Om(y)) \dar["\iso"]\\
	\Cb_D(x) \rar["\Cb_D(z)"]					& \Cb_D(y)
\end{tikzcd}\]
\end{defn}

We will now explicitly describe each $\Cb_D(z)$ with the following convention. The morphism $z \in \mc{B}$ induces a slew of cobordisms from $\Cb_D(x)$ to $\Cb_D(y)$. We fix the position in a $3\times4$-grid of each object in $\Cb_T$ according to \cref{eq:operator}, which allows us to display $\Cb(\bullet)$ and $\Cb(\circ)$ in a $3\times4$ grid as well, and we may also denote the collection of morphisms in $\Cb(x \xra{z} y)$ by a $3\times4$ matrix of cobordisms, where the $(i,j)$ entry in $\Cb_D(z)$ is the cobordism from the $(i,j)$ entry of $\Cb_D(x)$ to the $(i,j)$ entry of $\Cb_D(y)$.

%To write down the matrix $\Cb(z)$ we simply look up in \cref{fig:cobTable} the effect of each object in \cref{fig:resObjs}. This table yields objects other than $\iRes$ and $\oRes$, and as a result, the cobordisms are not expressible as elements in $\mc{B}$. Since we wish to work entirely with the Bar-Natan algebra $\mc{B}$, we use the delooped representatives of $\Cb(\bullet)$ and $\Cb(\circ)$. Therefore, we need to also write down delooped representatives of $\Cb(z)$, for $z \in \mc{B}$; this makes the matrix $\Cb(z)$ take the unfortunate form of a \textit{matrix of matrices of cobordisms}.
{
\newcommand{\loopedMatrixS}{\begin{pmatrix}
		\oLpSadLpD 	& -\oSad		& \oSad		& 0	\\
		0			& \iSadL		& -\iSadL		&  \oSadD \\
		0			& -\iLpSadLpL 	& \iLpSadLpL	& -\oLpSadLpD 
	\end{pmatrix}}
\newcommand{\deLoopedMatrixS}{\begin{pmatrix}
		\oLpSadLpD \circ L	& -\oSad			& \oSad			& 0	\\
		0				& L^{-1}\circ\iSadL	& -L^{-1} \circ \iSadL	& L^{-1} \circ \oSadD \\
		0				& -\iLpSadLpL \circ L	& \iLpSadLpL \circ L	& -\oLpSadLpD \circ L
	\end{pmatrix}}

\begin{prop}\label{prop:CCobs} The image under $\Cb_D$ of a generator $z \in \mc{B}$ is given in the following list.
\begin{equation}\label{eq:CCobs}\begin{split}
\Cb_D(S_\bullet) &=
	\begin{pmatrix}
		\row{D}{1}		& -S			& S			& 0	\\
		0			& \col{1}{G1}	& \col{-1}{-G1}	& \col{1}{SS} \\
		0			& \row{0}{-1}	& \row{0}{1}	& \row{-D}{-1}
	\end{pmatrix} \\
\Cb_D(S_\circ) &= 
	\begin{pmatrix}
		\col{1}{SS}	& -S			& S			& 0 \\
		0			& \row{0}{1}	& \row{0}{-1}	& \row{D}{1} \\
		0			& \col{-1}{-G1}	& \col{1}{G1}	& \col{-1}{-SS}
	\end{pmatrix} \\
\Cb_D(D_\bullet) &= 
	\begin{pmatrix}
		\mat{SS}{-1}{0}{D}	& -D				& D				& 0 \\
		0				& 0				& 0				& 0 \\
		0				& \mat{-G1}{1}{0}{0}	& \mat{G1}{-1}{0}{0}	& \mat{-SS}{1}{0}{-D}
	\end{pmatrix} \\
\Cb_D(D_\circ) &=
	\begin{pmatrix}
		0		& -D				& D				& 0 \\
		0		&\mat{-G1}{1}{0}{0}	&\mat{G1}{-1}{0}{0}	& \mat{-SS-D}{1}{0}{0} \\
		0		& 0				& 0				& 0
	\end{pmatrix} \\
\Cb_D(SS_\bullet) &= 
	\begin{pmatrix}
		\mat{D}{1}{0}{SS}	& -SS			& SS				& 0 		\\
		0				& G1				& -G1			& SS+D 	\\
		0				& -\mat{0}{1}{0}{G}	& \mat{0}{1}{0}{G}	& -\mat{D}{1}{0}{SS}
	\end{pmatrix} \\
\Cb_D(SS_\circ) &= 
	\begin{pmatrix}
		SS+D		& -SS				& SS				& 0 \\
		0			&\mat{0}{1}{0}{G1}	&\mat{0}{-1}{0}{-G1}	& \mat{D}{1}{0}{SS} \\
		0			& -G1				& G1				& -SS-D
	\end{pmatrix}
\end{split}\end{equation}

More generally\footnote{It is not hard to also compute $\Cb_D(z^k)$ for the other generators $z$ of $\mc{B}$, but, in view of \cref{lem:capTrivialConstraints}, we do not need them in this work, although they would be necessary to give a full description of the bimodule.}, for $k>1$, we have:
%
%\[\Cb_D(S_\bullet^{2k+1}) = 
%	\begin{pmatrix}
%		\row{D^{k+1}}{(SS+D)^k}	& -S^k		& S^k		& 0		\\
%		0					& \col{}{}		& \col{}{}		& \col{}{}	\\
%		0					& \row{}{}	& \row{}{}	& 0	
%	\end{pmatrix}
%\]
%
%\[\Cb_D((SS)_\bullet^{k}) = 
%	\begin{pmatrix}
%		\mat{}{}{}{}	& -S^k					& S^k					& 0		\\
%		0					& \col{}{}					& \col{}{}					& \col{}{}	\\
%		0					& -\mat{0}{G^{k-1}1}{0}{G^k1}	& \mat{0}{G^{k-1}1}{0}{G^k1}	& \mat{}{}{}{}
%	\end{pmatrix}
%\]
%
%\[\Cb_D(S_\circ^{2k+1}) = 
%	\begin{pmatrix}
%		\row{D^{k+1}}{(SS+d)^k}	& -S^k		& S^k		& 0		\\
%		0					& \col{}{}		& \col{}{}		& \col{}{}	\\
%		0					& \row{}{}	& \row{}{}	& 0	
%	\end{pmatrix}
%\]
%
%\[\Cb_D((SS)_\circ^{k}) = 
%	\begin{pmatrix}
%		\row{D^{k+1}}{(SS+d)^k}	& -S^k		& S^k		& 0		\\
%		0					& \col{}{}		& \col{}{}		& \col{}{}	\\
%		0					& \row{}{}	& \row{}{}	& 0	
%	\end{pmatrix}
%\]

\[\Cb_D(D_\bullet^{k}) = 
	\begin{cases}
	\DToThekEven	& k \text{ even}\\
	\DToThekOdd	& k \text{ odd}
	\end{cases}
\]

%\[\Cb_D(D_\circ^{k}) = 
%	\begin{pmatrix}
%		\row{D^{k+1}}{(SS+d)^k}	& -S^k		& S^k		& 0		\\
%		0					& \col{}{}		& \col{}{}		& \col{}{}	\\
%		0					& \row{}{}	& \row{}{}	& 0	
%	\end{pmatrix}
%\]
\begin{proof} 

We will only prove the result for $\Cb_D(S_\bullet)$, as a warm up, and for $\Cb_D(D_\bullet^k)$, since since all the minutiae of the computations are illustrated in this last case.

First, by directly dropping the cobordisms in \cref{fig:cobTable} inside the operator description in \cref{eq:operator}, we see that 
\[\Cb_K\big(\oRes \xra{\oSad} \iRes \big) =  \Cb_K\big(\oRes\big) \xra{\begin{pmatrix}
		\oLpSadLpD 	& -\oSad		& \oSad		& 0	\\
		0			& \iSadL		& -\iSadL		&  \oSadD \\
		0			& -\iLpSadLpL 	& \iLpSadLpL	& -\oLpSadLpD 
	\end{pmatrix}} \Cb_K\big(\iRes\big).\]
To obtain the representatives $\Cb_D(\bullet)$ and $\Cb_D(\circ)$ in \cref{eq:CObj}, we must deloop. We thus have the following commutative diagram, where the vertical maps are isomorphisms induced by delooping

\[\begin{tikzcd}[column sep=14em, row sep=5em]
\Cb_K\big(\oRes\big) \rar["\loopedMatrixS"] \dar["\cong"']	& \Cb_K\big(\iRes\big) \dar["\cong"] \\
\Om(\Cb_D(\bullet)) \rar["\deLoopedMatrixS"]								& \Om(\Cb_D(\circ))
\end{tikzcd}\]

All that is left is to compute all of the cobordisms in our preferred basis, which corresponds directly with the generators of the Bar-Natan path algebra $\mc{B}$:

\[
\oLpSadLpD \circ L 
= \oLpSadLpD \circ \row{\birthTri - G\cdot\birth\hspace{0.1cm}}{\hspace{0.1cm}\birth} 
= \row{\omTri-G \cdot 1_{\oRes}\hspace{0.1cm}}{\hspace{0.1cm}1_{\oRes}} 
\iso \row{SS_\bullet-G1_\bullet}{1_\bullet} = \row{D_\bullet}{1_\bullet}
\]

\[
L^{-1} \circ \iSadL 
= \col{\death \circ \imSadL}{\deathTri \circ \imSadL} 
= \col{1_{\iRes}}{\iH}
\iso \col{1_\circ}{H1_\circ}.
\]

Likewise,

\[\begin{split}
L^{-1} \circ \oSadD 	&= \col{1_{\oRes}}{\omSS} \iso \col{1_\bullet}{SS_\bullet}	\\
\iLpSadLpL \circ L	&= \row{\imTriL - G1_{\iRes} \hspace{0.1cm}}{\hspace{0.1cm} 1_{\iRes}} \iso \row{0}{1_\circ}	\\
\oLpSadLpD \circ L	&= \row{\omTri - G 1_{\oRes} \hspace{0.1cm}}{\hspace{0.1cm} 1_{\oRes}} \iso \row{D_\bullet}{1_\bullet}
\end{split}\]

This completes the computation and we have shown that
\[\Cb_D(S_\bullet) = 
\begin{pmatrix}
		\row{D_\bullet}{1_\bullet}	& -S_\bullet			& S_\bullet			& 0	\\
		0					& \col{1_\circ}{G1_\circ}	& \col{-1_\circ}{-G1_\circ}	& \col{1_\bullet}{SS_\bullet} \\
		0					& \row{0}{-1_\circ}		& \row{0}{1_\circ}		& \row{-D_\bullet}{-1_\bullet}
\end{pmatrix}.
\]

Let us now tackle the computation for $\Cb_D(D_\bullet^k)$. Again, we deloop to obtain the following array of cobordisms
\[\Cb_K(DD_\bullet^k) \cong 
\begin{pmatrix}
	L^{-1} \oLpDotDLp^k \circ L	& - \oDot^k	 				&  \oDot^k 				& 0	\\
	0						& 0							& 0						& 0	\\
	0						& -L^{-1}\circ\iLpDotLLp^k\circ L	& L^{-1}\circ\iLpDotLLp^k\circ L	& -L^{-1} \oLpDotDLp^k \circ L
\end{pmatrix}\]

We compute the three deloopings that appear, recalling our convention/warning following \cref{fig:cobTable}: the true marking (for the purpose of defining $G$ consistently) is always the top-left tangle end.

\[\begin{split}
L^{-1} \circ \oLpDotDLp^k \circ L 
&=  \col{\death \circ \left(\moLpSSLpD - \moLpSSLpLp \right)^k}{\deathTri \circ \left(\moLpSSLpD - \moLpSSLpLp\right)^k} \row{\birthTri - G\cdot\birth\hspace{0.1cm}}{\hspace{0.1cm}\birth} \\
&= \mat{\death \circ (-1)^{k-1}\left(\moLpSSLpD^k - \moLpSSLpLp^k \right) \circ \left( \birthTri - G \cdot \birth \right)}{\death \circ (-1)^{k-1}\left( \moLpSSLpD^k - \moLpSSLpLp^k \right) \circ \birth}{\hspace{0.1cm}\deathTri \circ (-1)^{k-1}\left( \moLpSSLpD^k - \moLpSSLpLp^k \right) \circ \left(\birthTri - G \cdot \birth \right)}{\hspace{0.1cm} \deathTri \circ (-1)^{k-1}\left( \moLpSSLpD^k - \moLpSSLpLp^k \right) \circ \birth}\\
&= \mat{(-1)^{k-1} \left(G^{k-1}\omSS - G \omSSDD^{k-1} - G^{k}1_{\oRes} + G \cdot 1_{\oRes} \sqcup \Sigma_k\right)}{(-1)^{k-1} \left( \omSSDD^{k-1} - 1_{\oRes} \sqcup \Sigma_k \right)}{(-1)^{k-1} \left(G^k\omSS - G^k \omSS + G^{k+1}1_{\oRes} - G^{k+1} \cdot 1_{\oRes} \right)}{(-1)^{k-1} \left(G^{k-1}\omSS - G^k 1_{\oRes} \right)},
\end{split}\]

where we use $\Sigma_g$ to denote the closed orientable surface of genus $g$. It is not hard to see that, in the presence of a marked component, the $/l$ relations imply
\[\Sigma_g = 	\begin{cases}
			 	0		&\text{ if } g \in 2\Z\\
				2G^{g-1}	&\text{ if } g \in 2\Z + 1
			\end{cases}
\]

We also have the following consequence of the 4-tubes (4Tu) relation:
\[ \oSSDD^k = \left(2\oSS - \oH\right)^k \simeq (SS_\bullet+D_\bullet)^k = SS_\bullet^k + D_\bullet^k = 
	\begin{cases}
		G^k1	_\bullet				&\text{ if } k \in 2\Z	\\
		G^{k-1}(SS_\bullet+D_\bullet)	&\text{ if } k \in 2\Z+1
	\end{cases}
\]

We have thus
\[\begin{split}
L^{-1} \circ \oLpDotDLp^k \circ L 
	&\iso \begin{cases}
			\mat{-\left(G^{k-1}SS - G^{k-1}(SS+D) - G^k1 + 0\right)}{-G^{k-2}(SS+D)}{0}{-\left( G^{k-1}SS - G^k1\right)} 	
				&\text{ if }	k \in 2\Z\\
			\mat{\left(G^{k-1}SS - G^k1 - G^k1 + 2G^k1 \right)}{G^{k-1}1- 2G^{k-1}1}{0}{\left( G^{k-1}SS - G^k1\right)} 		
				&\text{ if } k \in 2\Z+1,
	\end{cases}\\
	&= \begin{cases}
		\mat{G^{k-1}SS_\bullet}{-G^{k-2}(SS_\bullet+D_\bullet)}{0}{-G^{k-1}D_\bullet}
			&\text{ if } k \in 2\Z\\
		\mat{G^{k-1}SS_\bullet}{-G^{k-1}1_\bullet}{0}{G^{k-1}D_\bullet}
			&\text{ if } k \in 2\Z+1
	\end{cases}
\end{split}\]

Similarly,

\[\begin{split}
L^{-1} \circ \iLpDotLLp^k \circ L
	&=	\mat{\death \circ \left( \miLpSSLpL - \miLpSSLpLp \right)^k \circ \left(\birthTri - G\cdot\birth\right)}{\death \circ \left( \miLpSSLpL - \miLpSSLpLp \right)^k \circ \birth }{\deathTri \circ \left( \miLpSSLpL - \miLpSSLpLp \right)^k \circ \left(\birthTri - G\cdot\birth\right)}{\deathTri \circ \left( \miLpSSLpL - \miLpSSLpLp \right)^k \circ \birth} \\
	&= \mat{(-1)^{k-1}\left(G^k1_{\iRes} - G^k1_{\iRes} - G^k1_{\iRes} + G 1_{\iRes} \sqcup \Sigma_k \right)}{(-1)^{k-1} \left(G^{k-1}1_{\iRes} - 1_{\iRes} \sqcup \Sigma_k\right)}{0}{0} \\
	&\iso \mat{G^k1_\bullet}{-G^{k-1}1_\bullet}{}{},
\end{split}\]
which completes the computation of $\Cb_D(D^k_\bullet)$.

%\[\Cb_D(S_\circ) = 
%	\begin{pmatrix}
%		L^{-1} \circ \oSadD	& -\iSad			& \iSad			& 0	\\
%		0				& \iLpSadLpL \circ L	& \iLpSadLpL \circ L	& \oLpSadLpD \circ L \\
%		0				& -L^{-1} \circ \iSadL	& L^{-1} \circ \iSadL	& -\oSadD \circ L
%	\end{pmatrix}\]
%	
%\[\Cb_D(D_\bullet) = 
%	\begin{pmatrix}
%		L^{-1}\circ\oLpDotDLp\circ L 	& -\oDot 					& \oDot 					& 0 \\
%		0						& 0						& 0						& 0 \\
%		0						& -L^{-1}\circ\iDotToLp\circ L	& L^{-1}\circ\iDotToLp\circ L	& L^{-1}\circ \oLpDotDLp \circ L
%	\end{pmatrix}\]
%	
%\[\Cb_D(D_\circ) =
%	\begin{pmatrix}
%		0	& -\iDot		& \iDot		& 0 \\
%		0	& \iDotToLp	& -\iDotToLp	& \oLpDotDLp \\
%		0	& 0			& 0 			& 0
%	\end{pmatrix}\]
\end{proof}
\end{prop}
}

We can now describe $\Cb_D(\Rd(T))$ as follows. Using type D structure language, let 
\[ (\Rd(T), \delta) = \left(\bigoplus_i x_i ,  \big( \delta_{ij} \big) \right),\]
where $(\delta_{i,j})$ is the matrix of $\mc{B}$-coefficients of the type D differential $\delta$:
\[\delta(x_i) = \sum_{j} \delta_{ij} \otimes x_j.\] 
Let also ${\odot}_i$ be the idempotent label of $x_i$. In other words, ${\odot}_i = \bullet$ precisely when $1_\bullet(x_i) = x_i$, otherwise ${\odot}_i = \circ$. Then $\Cb_D(\Rd(T))$ is the type D structure
\[\left( \bigoplus_i \Cb_D({\odot}_i) , \big( \Cb_D(\delta_{i,j}) \big) \right).\]

Complexes such as the one above are complicated, and there is a unique curve-like complex in each chain homotopy class \cref{def:curvyD}. This is essentially the immersed curve invariant that is discussed in \cref{sec:curves}. Our aim is to describe the curve-like representative of $\Cb_D(\Rd(T))$, for all cap-trivial $T$. For that, we will use the following notation (``F" is for Fukaya): 

%%%%%%%%%%%%%%%%%%%%%%%%%%%%%%%%%%%
%% Notation for the Fukaya Category and Efficient computation    %%
%%%%%%%%%%%%%%%%%%%%%%%%%%%%%%%%%%%

\begin{notn} We use the subscript $F$ to denote the following complexes, which are chain homotopic to ones obtained by applying $\Cb_D$.
\[\begin{split}
\Cb_{F^{pre}}(\prescript{q}{}{\bullet_h}) &:= 
	\begin{tikzcd}[ampersand replacement=\&]
	\prescript{q-4}{}{\bullet_{h-2}} \ar[ddr, "S"']	\&				\& \bullet  \ar[ddr, bend left = 100, "-1"]	\&\\
									\& 				\&								\&\\
									\& \circ \rar["D"]		\& \circ  \rar["S"]					\& \bullet
	\end{tikzcd}\\
\Cb_F(\prescript{q}{}{\bullet_h}) &:= 
	\begin{tikzcd}[ampersand replacement=\&]
	\prescript{q-4}{}{\bullet_{h-2}} \ar[ddr, "S"']	\&				\& 		\\
									\& 				\&		\\
									\& \circ \rar["D"]		\& \circ
	\end{tikzcd}\\
\Cb_{F^{pre}}(\circ) &:=
	\begin{tikzcd}[ampersand replacement=\&]
	\prescript{q-3}{}{\bullet_{h-2}} \ar[ddr, "S"]	\&					\&				\& \\
									\& \circ \rar["D"]			\& \circ \rar["S"]		\& \bullet \\
									\& \circ \uar["-D"'] \rar["D"]	\& \circ \uar["D"'] 	\&
	\end{tikzcd}\\
\Cb_F(\circ) &:=
	\begin{tikzcd}[ampersand replacement=\&]
	\prescript{q-3}{}{\bullet_{h-2}} \ar[ddr, "S"]	\&				\&							\& \\
									\& \circ			\& \circ \rar["S"]					\& \bullet \\
									\& \circ \uar["-D"']	\& \prescript{q}{}{\circ_h} \uar["D"'] 	\&
	\end{tikzcd}
\end{split}\]
\end{notn}

\subsection{An example: the cables of the unknot}\label{subsec:unknot}

The (2,1) cable of the $k$-framed unknot is the $(2,2k+1)$ torus knot. Because the tangle associated to the strong inversion on the $(2, 2k+1)$ torus knot is algebraic, its immersed curve invariant may be computed by applying the tensor product method in \cite[Example 4.28]{KWZ19}. Our cabling operator offers an alternative computation. Recall that $\tau$ is the Dehn twist operator in \cref{fig:DehnTwist} and let $T_k$ denote the quotient tangle associated with the $(2,2k+1)$ torus knot. Then we have

\[\begin{split}
\Rd(T_k) = \Cb\left(\Rd\left(\tau^k \oResOr \right)\right) 
		&= 
		\begin{cases} 
			\Cb(\underbrace{\prescript{1}{}{\circ_0} \lra \cdots \xra{D_\circ} \circ}_{|k|} \xra{S_\circ} \prescript{-2k}{}{\bullet_{-k}})		&\text{ if } k < 0	\\
			\Cb(\prescript{-2k}{}{\bullet_{-k}} \xra{S_\bullet} \underbrace{\circ \xra{D_\circ} \cdots \lra \prescript{-1}{}{\circ_0}}_{k})		&\text{ if } k > 0
		\end{cases}	
%		&=
%		\begin{cases}
%			\Cb(\circ \ra \cdots \xra{D_\circ} \circ) \sqcup \Cb(\bullet)		&\text{ if } k < 0 \\
%			\Cb(\bullet) \sqcup \Cb(\circ \xra{D_\circ} \cdots \ra \circ)		&\text{ if } k > 0
%		\end{cases}
\end{split}\]

Consider the case $k > 0$: a homotopic representative is given by
\[\begin{split}
\Rd(T_k) 	&\simeq \Cb_D(\prescript{-2k}{}{\bullet_{-k}}) \xra{\Cb_D(S_\bullet)} \Cb_D(\circ) \xra{\Cb_D(D_\circ)} \Cb_D(\circ) \xra{SS}  \Cb_D(\circ) \xra{\Cb_D(D)} \cdots,
\end{split}\]
which is equal to
\begin{figure}[h]
\centering
\includegraphics[scale=0.85]{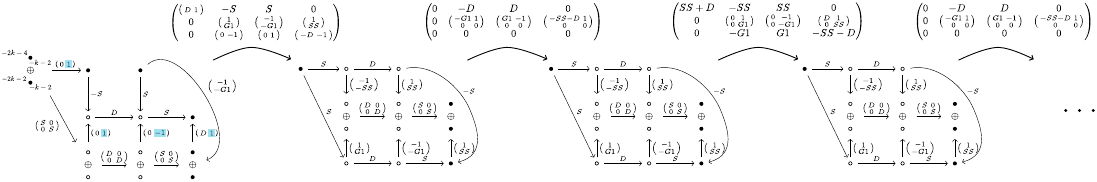}
\end{figure}

We apply the Cancellation Lemma to the highlighted isomorphisms to obtain the following complex, where we again cancel the highlighted isomorphism.
\begin{figure}[h]
\centering
\includegraphics[scale=0.85]{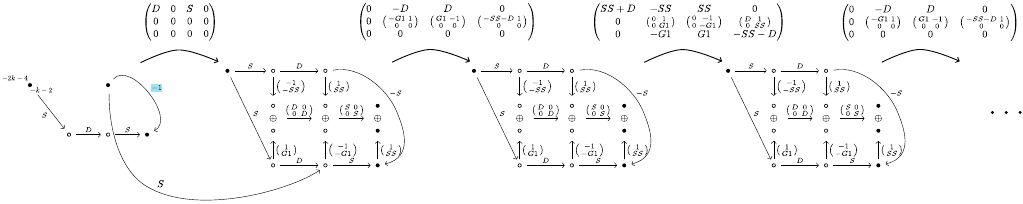}
\includegraphics[scale=0.85]{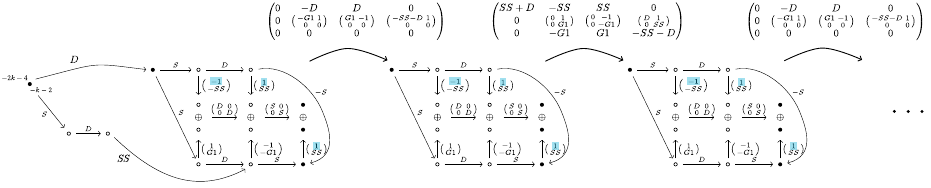}
\end{figure}

We denote the maps induced by cancellation with a lower asterisk, so that the previous computation may be compactly written as
\[\begin{split}
\Rd(T_k)	&\simeq \Cb_{F^{pre}}(\bullet) \xra{{\Cb_D(S_\bullet)}_\ast} \Cb_D(\circ) \xra{\Cb_D(D)} \Cb_D(\circ) \xra{\Cb_D(SS)} \Cb_D(\circ) \xra{\Cb_D(D)} \cdots \\
		&\simeq \Cb_F(\bullet) \xra{{\Cb_D(S_\bullet)}_\ast} \Cb_D(\circ) \xra{\Cb_D(D)} \Cb_D(\circ) \xra{\Cb_D(SS)} \Cb_D(\circ) \xra{\Cb_D(D)} \cdots
\end{split}\]

In principle, cancellation could result in ``long" differentials (nonzero differentials on the second page of a spectral sequence), but these do not occur in this particular example. Finally, by cancelling the isomorphisms highlighted in the previous diagram, we obtain the complex
\[\Cb_F(\bullet) \xra{\Cb_D(\bullet)_*} \Cb_{F^{pre}}(\circ) \xra{ \Cb_{F^{pre}}(D_\circ)} \Cb_{F^{pre}}(\circ) \xra{ \Cb_{F^{pre}}(SS_\circ)} \Cb_{F^{pre}}(\circ) \xra{ \Cb_{F^{pre}}(D_\circ)},\]
where $\Cb_{F^{pre}}(D_\circ)$ denotes the induced map ${\Cb_D(D_\circ)}_*$, and likewise for $\Cb_{F^{pre}}(SS_\circ)$. Explicitly, this type D structure is given by
%\[\Cb_F(\bullet) \xra{\Cb_F^{pre}(S)} \Cb_F^{pre}(\circ) \xra{\Cb_F^{pre}(D)} \Cb_F^{pre}(\circ) \xra{\Cb_F^{pre}(SS)} \Cb_F^{pre}(\circ) \xra{\Cb_F^{pre}(D)} \cdots\]
%which is precisely equal to the following:
\begin{figure}[h]
\centering
\includegraphics{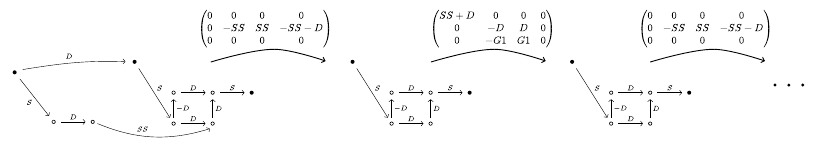}
\end{figure}

We simplify the type D structure further by applying the Clean-Up Lemma with respect to the morphism indicated by the doubled green arrows (it is not hard to check that the hypotheses of the lemma are satisfied):
\begin{figure}[h]
\centering
\includegraphics[scale=0.95]{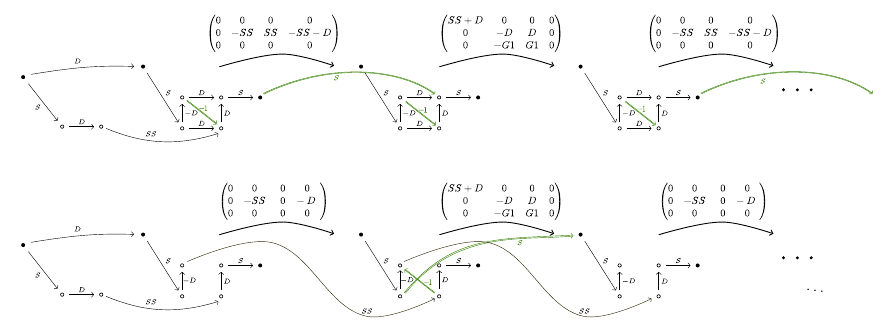}
\end{figure} 

This results in the following complex (minus the doubled green arrows, which we will use to clean up in a second). 
\begin{figure}[h]
\centering
\includegraphics{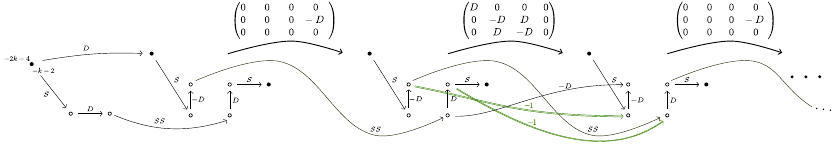}
\end{figure} 

At this point, we see already that
\[\Rd(T_1) \simeq 
\begin{tikzcd}[row sep = small, column sep = small]
\prescript{-6}{}{\bullet_{-3}} \ar[ddr, "S"] \ar[rrrr,"D"]	&			&						&	& \bullet\ar[ddr, "S"]	&				&			&	\\
						&			&						&	&				&\circ			& \circ \rar["S"]	& \bullet \\
						& \circ \rar["D"]	& \circ \ar[rrrr, bend right, "SS"]	&	&				&\circ \uar["-D"']		& \circ \uar["D"']	&
\end{tikzcd},\]
as expected.

\begin{notn} Let us define the following bigraded type D structure
\begin{equation}\label{eq:2Compact}
\prescript{q}{}{C_h} := 
\begin{tikzcd}[column sep = small]
\prescript{q}{}{\bullet_h} \rar["S_\bullet"]\dar["D_\bullet"]	&\circ \rar["D_\circ"]	&\circ \rar["SS_\circ"]	&\circ \rar["D_\circ"]	&\circ \rar["S_\circ"]	&\bullet \dar["D_\bullet"]\\
\bullet \rar["S_\bullet"]				&\circ \rar["D_\circ"]	&\circ \rar["SS_\circ"]	&\circ \rar["D_\circ"]	&\circ \rar["S_\circ"]	&\bullet
\end{tikzcd}.\end{equation}
\end{notn}

Then we also see from our computation so far that
\[\Rd(T_2) \simeq \prescript{-8}{}{C_{-4}} \oplus
\begin{tikzcd}[row sep = small, column sep = small]
\prescript{-4}{}{\bullet_{-2}} \rar["S"]	& \circ  \rar["-D"]	& \circ
\end{tikzcd}.\]

Continuing with the clean-up indicated above, we obtain
\begin{figure}[h]
\centering
\includegraphics{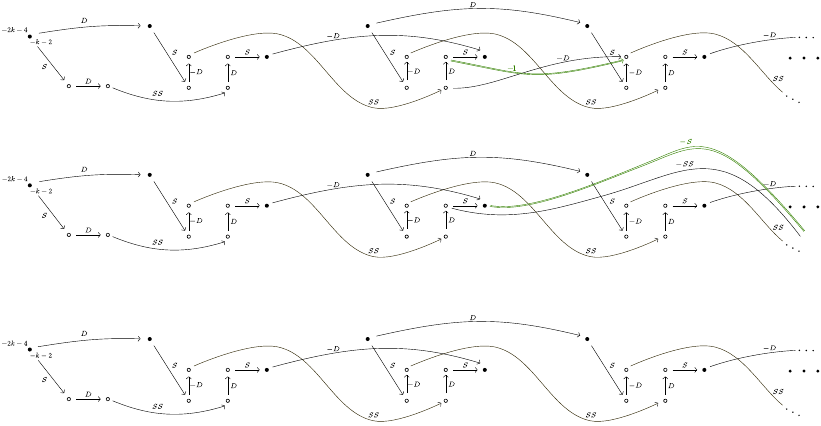}
\end{figure} 

And it is clear that, in general, we have

\[\Rd(T_k) = 
\begin{cases} 
		\underbrace{\prescript{-2k-4}{}{C_{-k-2}} \oplus \prescript{-2k}{}{C_{-k}} \oplus \cdots \oplus \prescript{-10}{}{C_{-5}}}_{\frac{k-1}{2}} \oplus 
		\begin{tikzcd}[row sep = small, column sep = small, ampersand	replacement = \&]
												\&	\bullet \rar["S"]	\& \circ \rar["D"]		\&\circ		\&			\&		\\
			\prescript{-6}{}{\bullet_{-3}} \ar[dr, "S"]\ar[ur,"D"]	\& 				\&				\&			\&			\&		\\
												\&\circ \rar["D"]		\&\circ \rar["SS"]	\&\circ \rar["D"]	\&\circ \rar["S"]	\&\bullet
		\end{tikzcd}
	&\text{ if } k \in 2\N+1	\\
		\underbrace{\prescript{-2k-4}{}{C_{-k-2}} \oplus \prescript{-2k}{}{C_{-k}} \oplus \cdots \oplus \prescript{-8}{}{C_{-4}}}_\frac{k}{2} \oplus 
		\begin{tikzcd}[row sep = small, column sep = small, ampersand	replacement = \&]
			\prescript{-4}{}{\bullet_{-2}} \rar["S"]	\&\circ \rar["D"]	\&\circ
		\end{tikzcd}
	&\text{ if } k \in 2\N
\end{cases}\]

It is no more difficult to compute $\Cb_D(T_k)$ for $k<0$: in this case, we start with the type D structure
\[\xra{\Cb_D(D_\circ)} \Cb_D(\circ) \xra{\Cb_D(SS_\circ)} \Cb_D(\circ) \xra{\Cb_D(D_\circ)} \Cb_D(\circ) \xra{\Cb_D(S_\circ)} \Cb_D(\prescript{-2k}{}{\bullet_{-k}}).\]
cancelling the isomorphisms in $\Cb_D(\bullet)$ results in 
\[\xra{\Cb_D(D_\circ)} \Cb_D(\circ) \xra{\Cb_D(SS_\circ)} \Cb_D(\circ) \xra{\Cb_D(D_\circ)} \Cb_D(\circ) \xra{\Cb_F(S_\circ)} \Cb_F(\prescript{-2k}{}{\bullet_{-k}}),\]
where the induced $\Cb_F(S_\circ)$ is equal to
\[\begin{pmatrix} 1&0&0&0\\ 0&0&0&0\\ 0&-1&1&0 \end{pmatrix}.\]
This type D structure may be cleaned-up similarly to how we proceeded above, as shown in the following figure:

\begin{figure}[h]
\centering
\includegraphics[scale=1.1]{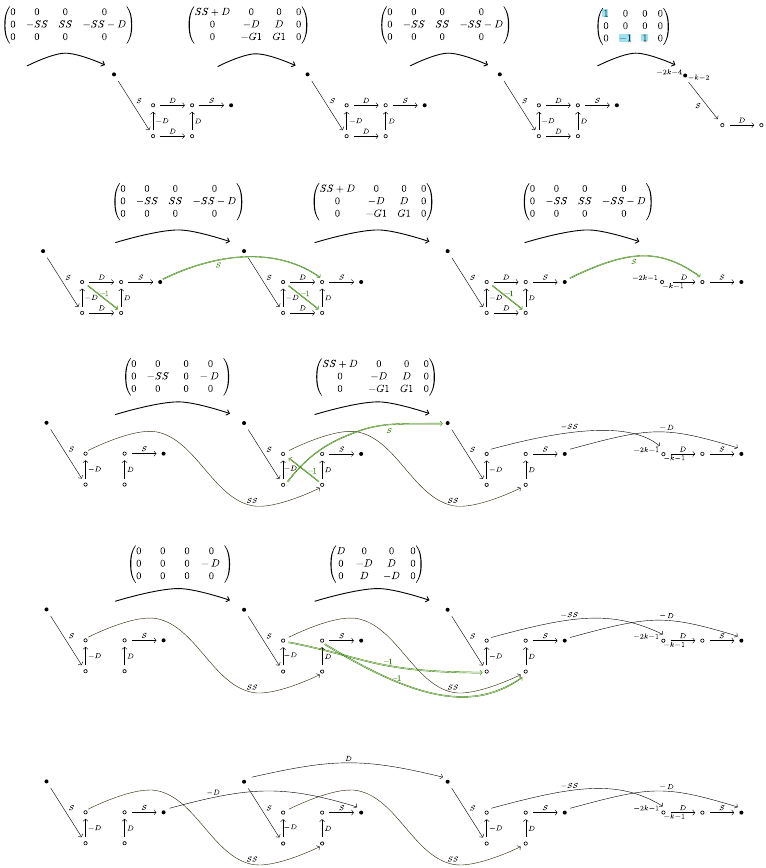}
\end{figure}

We see then that, if $k < 0$, then we have
\[\Rd(T_k) = 
\begin{cases} 
		\begin{tikzcd}[row sep = small, column sep = small, ampersand	replacement = \&]
			\prescript{1}{}{\circ_{0}} \rar["D"]	\&\circ \rar["S"]	\&\bullet
		\end{tikzcd}
		\oplus 
		\underbrace{\prescript{-3}{}{C_{-2}} \oplus \prescript{-7}{}{C_{-4}} \oplus \cdots \oplus \prescript{-2k-8}{}{C_{-k-5}}}_\frac{-k-1}{2} 
	&\text{ if } k \in 2\Z+1	\\
		\begin{tikzcd}[row sep = small, column sep = small, ampersand	replacement = \&]
				\&				\&			\& \prescript{1}{}{\circ_0} \rar["D"]	\& \circ \rar["S"]	\& \bullet\ar[dr,"D"]	\&			\\
				\& 				\&			\&							\&			\&				\&\bullet 		\\
				\&\bullet \rar["S"]	\&\circ \rar["D"]	\&\circ \rar["SS"]				\&\circ \rar["D"]	\&\circ \ar[ur,"S"]	\&
		\end{tikzcd}
		\oplus
		\underbrace{\prescript{0}{}{C_{-1}} \oplus \prescript{4}{}{C_{2}} \oplus \cdots \oplus \prescript{-2k-8}{}{C_{-k-5}}}_{\frac{-k}{2}}  
	&\text{ if } k \in 2\Z
\end{cases}\]

This completes the computation of the curve-like representative of $\Rd(T_k)$ (\cref{def:curvyD}).

\subsection{On Seifert framing}

Let us end this section by fulfilling a promise made in the introduction and show that $\Cb^0$ fixes the class of Seifert-framed cap-trivial tangles. We will do this by checking that the 0-closure of $\Cb^0(T)$ has determinant 0. As mentioned above, the determinant of a link agrees up to sign with either the reduced Jones or the Alexander polynomial evaluated at -1. %One way to see this is to note that the skein relations that characterize these polynomials become equivalent when we set $t=-1$ and take the norm. 
Moreover, the Jones polynomial is the graded Euler characteristic of Khovanov homology:
\[V_L(t) = \sum_{q,h} \mathit{rk}\left(\prescript{q}{}{\wt{\Kh}_h}(L)\right) \cdot (-1)^h t^{q/2},\]
where $\prescript{q}{}{\wt{\Kh}_h}(L)$ is the rank of the reduced Khovanov homology of $L$ in quantum grading $q$ and homological grading $h$, and $L$ is a pointed link. Alternatively, since the Euler characteristic is invariant under quasi-isomorphisms, we can, in place of the rank of the Khovanov homology, use $\mathit{rk}\left(\prescript{q}{}{\wt{\CKh}_h}(L)\right) = \mathit{rk}\left(\prescript{q}{}{\wt{\CBN}^{H=0}_h}(L)\right)$, \cite[\S3]{KWZ19}. It is this latter chain complex that we will use, computed using the following theorem.

\begin{thm}[\!\!\cite{KWZ19}, Proposition 4.31] \label{thm:hom} Let $\mc{L} = \mc{L}(T_1, T_2)$, in the notation of \cref{def:cutPaste}. Then there is a chain homotopy
\begin{equation} \label{eq:hom} 
\wt{\CBN}(\mc{L})\{-1\} \iso \mathrm{Mor}(\Rd(mT_1), \Rd(T_2))
\end{equation}
as bigraded chain complexes of $R[H]$-modules, where $m$ denotes the mirror, the $H$-action on the right is given by $-G$, and the differential on $\mathrm{Mor}(\Rd_1, \Rd_2)$ is given by
\[\del f = f \circ d_1 - (-1)^{h(f)} d_2 \circ f.\] \qed
\end{thm}

\begin{prop}\label{prop:detCb} Let $T$ be Seifert-framed cap trivial. Then $\det(\Cb^0(T)(0)) = 0$.
\end{prop}
\begin{proof} Since $T$ is Seifert framed, $\det(T(0)) = 0$. Equivalently,
\[\begin{split}
|V_{T(0)}(-1)| = \left| \chi_{gr} \left(\wt{\CBN}^{H=0}(T(0))\right)\big|_{t = -1} \right| = \left| \chi_{gr}\left(\mathrm{Mor}_{\mc{B}_0}(\prescript{0}{}{\bullet_0}, \Rd(T)) \{+1\}\right)\big|_{t=-1}\right| = 0,
\end{split}\]
where $\mc{B}_0$ is the quotient of the algebra $\mc{B}$ by the relations $G_\circ = G_\bullet = 0$. The differential of $\Rd(T)$ does not contribute to the Euler characteristic, so we only need to keep track of the generators. We have %Precisely, the Euler characteristic depends on the bigradings of the generators of $\Rd(T)$ and whether they are in $\mathrm{im} 1_\bullet$ or in $\mathrm{im} 1_\circ$. In other words,
\[\chi_{gr}\mathrm{Mor}_{\mc{B}_0}(\prescript{0}{}{\bullet_0}, \Rd(T))\{+1\} = \chi_{gr} \mathrm{Mor}_{\mc{B}_0}(\prescript{0}{}{\bullet_0}, 1_\bullet \cdot \Rd(T))\{+1\} + \chi_{gr} \mathrm{Mor}_{\mc{B}_0}(\prescript{0}{}{\bullet_0}, 1_\circ \cdot \Rd(T))\{+1\}.\]
(Recall that $1_\bullet, 1_\circ$ generate the idempotent subring $\mc{I}_\mc{B} \subset \mc{B}$, over which $\Rd(T)$ is a module). Note first that, for every generator $\prescript{q}{}{\bullet_h} \in \Rd(T)$, we have
\[\chi_{gr} \mathrm{Mor}_{\mc{B}_0}(\prescript{0}{}{\bullet_0}, \prescript{q}{}{\bullet_h}) = \chi_{gr}R\gp{\prescript{q}{}{{1_\bullet}_h}, \prescript{q-2}{}{{D_\bullet}_h}} = (-1)^h t^{\frac{q}{2}} + (-1)^h t^\frac{q-2}{2},\]
which vanishes when evaluated at $t=-1$. Thus, only $\mathrm{Mor}_{\mc{B}_0}(\prescript{0}{}{\bullet_0}, 1_\circ \cdot \Rd(T))$ contributes to the determinant. It is easiest to keep track of this contribution by using the $\delta$-grading, which is defined to be $\delta = \frac{q}{2}-h$. Now, for a generator $\prescript{q}{}{\circ_h}$ in $1_\circ \cdot \Rd(T)$, we have
\[\chi_{gr} \left(\mathrm{Mor}_{\mc{B}_0}(\prescript{0}{}{\bullet_0}, \prescript{q}{}{\circ_h})\right)\big|_{t=-1} = (-1)^h t^{\frac{q-1}{2}}\big|_{t=-1} = (-1)^{h + \frac{q-1}{2}} = (-1)^{\delta -\frac{1}{2}}.\]
Since the quantum grading takes values either in $2\Z$ or in $2\Z + 1$, the $\delta$ grading takes values either in $\Z + \frac{1}{2}$ or in $\Z$. In either case, there is a partition of the $\delta$-gradings into two subsets, $Z^\text{odd}$ and $Z^\text{even}$, such that 
\[(-1)^{\delta - \frac{1}{2}} = \begin{cases} \zeta &\text{ if } \delta \in Z^\text{even}\\ - \zeta &\text{ if } \delta \in Z^\text{odd} \end{cases},\]
where $\zeta$ is either $i$ or $1$. It is not hard to see that, in our case, because $T(0)$ is a 2-component link, we have $\zeta = i$, but this is unimportant. The point is that we can also partition $1_\circ \cdot \Rd(T)$ into $\Rd(T)_\circ^\text{odd}$ and $\Rd(T)_\circ^\text{even}$ so that
\[\det(T(0)) = \left| \mathrm{rk} \Rd(T)_\circ^\text{odd} - \mathrm{rk} \Rd(T)_\circ^\text{even} \right|.\]
Thus, $T$ being Seifert framed is equivalent to the statement that $1_\circ \cdot \Rd(T)$ has as many generators in even as in odd $\delta$-grading (for our generalization of ``even" and ``odd"). To complete the proof of the proposition, we only need to show that $\Cb^0$ also results in a type D structure with this property. We have not computed the planar algebra operation nor the operator $\Cb^0_D$ on $\cat{Mod}^\mc{B}$ induced by $\Cb^0$, but they both exist, and it is clear that
\[\Cb^0_D(\bullet) = \bullet.\]
Thus, the module $1_\circ \cdot \Cb^0_D(T)$ depends only on $1_\circ \cdot \Rd(T)$. It follows then that, if $\Rd(T)$ has the same number of generators in even and odd $\delta$ grading, then so does the type D structure $\Cb^0_D(T)$.
\end{proof}

By the Montesinos trick, the proposition implies that $\Cb^0(T)$ is Seifert framed.

%!TEX root = Cabling_Operator.tex

\section{Structural properties} \label{sec:properties}

In this section we present a fairly tight set of constraints on the Bar-Natan invariant of a cap-trivial tangle $T$, and we then apply these to prove structural properties of $\Cb(T)$. For that, we make use of the immersed curve theory in \cite{KWZ19}, so we start by providing its key aspects in the following subsection.

\subsection{Immersed curve invariants of cap-trivial tangles}\label{sec:curves}

We now need to introduce the immersed curve theory of \cite{KWZ19}. The following results require that we take coefficients in a field $\k$. Recall that $S^2_{4,*}$ is 4-punctured $S^2$, with one puncture marked $\ast$. 

\begin{defn} We treat the special puncture $* \in S^2_{4,*}$ as a basepoint and let $S^1_*$ be a pointed circle. A \textbf{parametrization} of $S^2_{4,*}$ consists of the following
\begin{enumerate}
	\item A choice of two pointed embeddings $i_\circ, i_\bullet \co S^1_* \ra S^2_{4,*}$ such that $S^2_{4,*} \setminus \mathrm{im}(i_\bullet) \cup \mathrm{im}(i_\circ)$ is a disjoint union of three punctured discs $D_1, D_2, D_3$.
	\item An orientation-preserving embedding $i$ of the quiver from \cref{def:quiv} (thought of as a 1-dimensional CW complex)
%	\[\begin{tikzcd}
%	\bullet \ar[loop left, "D_\bullet"] \rar[bend right, "S_\bullet"']	&\circ \ar[l, bend right, "S_\circ"'] \ar[loop right, "D_\circ"]
%\end{tikzcd}\]
	 into $S^2_{4,*}$, so that $i(\bullet) \in \mathrm{im}(i_\bullet)$, $i(\circ) \in \mathrm{im}(i_\circ)$ and each arrow embeds orientation-preservingly into one of the discs $D_i$.
\end{enumerate}
See \cref{fig:S2param}. We call the image of $i_\bullet$ (resp. $i_\circ$) the $\bullet$ (resp. $\circ$) parametrizing arc.
\end{defn}

\begin{figure}[h]
	\labellist
	\hair 2pt
	\pinlabel $D_\bullet$ at 36 16
	\pinlabel $S_\circ$ at 71 75
	\pinlabel $D_\circ$ at 122 100
	\pinlabel $S_\bullet$ at 121 10
	%\pinlabel (4Tu) at 182 53
	\endlabellist
	\centering
	\includegraphics[scale=0.8]{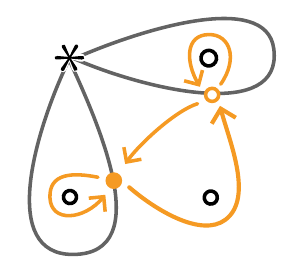}
	\caption{The surface $S^2_{4,*}$, parametrized by choice of arcs and an embedding of the quiver defining $\mc{B}$.}
	\label{fig:S2param}
\end{figure}
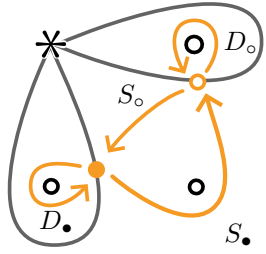

\begin{con}\label{con:immersion} Let $T$ be a pointed 4-ended tangle. The parametrization in \cref{fig:S2param} determines how to immerse of the graph of $\Rd(T)$ into $S^2_{4,*}$ as one of two train tracks, $i(\Rd(T))$ or $i^{\spir}(\Rd(T))$. Fix a metric on $S^2_{4,*}$ that agrees (locally) with the flat metric on this page, so that we can draw geodesics as straight lines. First, embed the vertices on the corresponding parametrizing arcs, i.e. place every vertex of $\Rd(T)$ that is labelled $\circ$ (resp.\ $\bullet$) on the $\circ$ (resp.\ $\bullet$) parametrizing arc. Second, given an edge $\xra{P}$ in $\Rd(T)$ that is labelled by a path algebra element $P \in \mc{B}$, let $P = Z_1\dots Z_k$, where each $Z_i$ is either $S_{\odot}$ or $D_\odot$; immerse the edge $\xra{P}$ so that it is homotopic to the path in $S^2_{4,*}$ that is the concatenation $i(Z_1)\dots i(Z_k)$ and so that it intersects the parametrizing arcs in a right angle. Finally, for each leaf of $\Rd(T)$, i.e. vertex of valence 1, the embedding is obtained by connecting the image of the vertex to a puncture so that no new intersections with the parametrizing arcs are generated and the immersed curve still intersects the parametrizing arcs transversely. To obtain he immersion $i(\Rd(T))$, connect leaves to punctures via shortest paths, while the embedding $i^{\spir}(\Rd(T))$ is defined so as to connect leaves to punctures via curves that twist infinitely counterclockwise around the puncture; see \cref{fig:wrapPair} for both kinds of immersions.

The resulting 1-manifold in $S^2_{4,*}$ is a train track whose switches all occur in a small neighbourhood of the parametrizing arcs.
\end{con}

As an example, for the following representative of $\Rd(T_{3_1})$ (where $T_{3_1}$ is the quotient tangle associated with the unique strong inversion on the Seifert-framed right-handed trefoil)
\[\begin{tikzcd}[column sep = 1em, row sep = 0.5em]
	\circ \rar["D"]	&\circ \rar["SS"]	&\circ \rar["D"]	&\circ \ar[dr, "S"]			&			&				&			&			&	\\
				&			&			&						&\bullet		&				&			&			&	\\		
				&			&			&\bullet \ar[ur, "D"]\ar[dr, "S"]	&			&				&			&			&	\\
				&			&			&						& \circ\rar["D"]	&\circ \rar["SS"]		&\circ \rar["D"]	&\circ\rar["S"]	&\bullet
\end{tikzcd}\]
the embedding $i(\Rd(T_{3_1})$ is the curve in \cref{fig:singint2} below.

\begin{figure}[h]
	\centering
	\includegraphics[height=0.4\textwidth]{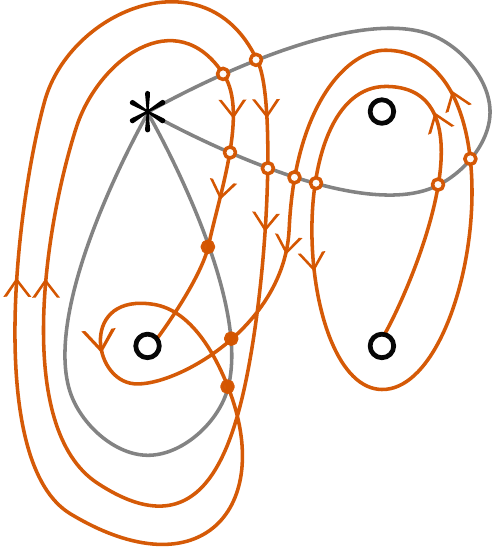}
	\caption{The immersed curve $\wt{\BN}(T_{3_1})$.}
	\label{fig:singint2}
\end{figure}

\begin{thm}[\cite{KWZ19}, Theorem 1.5] Let $T$ be a 4-ended framed pointed tangle. There is a (possibly disconnected) curve $\wt{\BN}(T; \k)$, immersed in the pointed 4-punctured sphere $S^2_{4,*}$, with a possibly non-trivial local system. This curve with local system is a tangle invariant when considered up to homotopy of the curve and matrix similarity of the local system.
\end{thm}

We will often drop the field $\k$ from the notation. One may think of the invariant $\wt{\BN}(T)$ as a reformulation of the type D structure $\Rd(T)$: by intersecting $\wt{\BN}(T)$ with the parametrizing arcs and running \cref{con:immersion} backwards, we recover a graph that is, as a type D structure, chain homotopy equivalent to $\Rd(T)$; see \cite[Theorem 5.14]{KWZ19}.

\begin{defn}\label{def:curvyD} Given $\wt{\BN}(T)$, we call the type D structure obtained by intersecting $\wt{\BN}(T)$ with the parametrizing arcs the \textbf{curve-like} representative of $\Rd(T)$.
\end{defn}

\begin{thm}[{\!\!\cite[Theorem 1.9]{KWZ19}}] Suppose that $T_1$, $T_2$ are two 4-ended framed pointed tangles and let $\mc{L}(T_1, T_2)$ be the link obtained by gluing $T_1$ and $T_2$ together as in \cref{fig:glue}. Then the Bar-Natan homology of $\mc{L}(T_1, T_2)$ is isomorphic to the wrapped Lagrangian Floer homology of $\wt{\BN}(mT_1)$ and $\wt{\BN}(T_2)$:
\[\wt{\BN}(\mc{L}(T_1, T_2); \k) \iso \HF(\wt{\BN}(mT_1; \k), \wt{\BN}(T_2; \k)),\]
where $m$ denotes the mirror.
\end{thm}

\begin{figure}[h]
	\centering
	\includegraphics[height=0.4\textwidth]{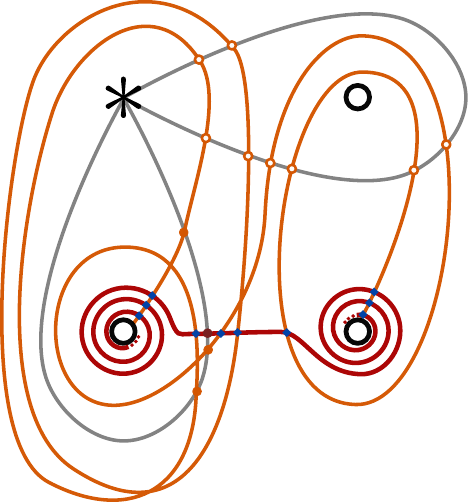}
	\caption{The Lagrangians $\wt{\BN}(Q(0))^{\protect\spir}$ (in red) and $\wt{\BN}(T_{3_1})$ (in orange), and their intersection points (in blue). This computes $\HF(\wt{\BN}(mQ(0)), \wt{\BN}(T_{3_1})) \iso \k^4 \oplus \k[H]^2$.}
	\label{fig:wrapPair}
\end{figure}

\begin{thm}[\!\!\cite{KWZ19}, Theorem 8.1]\label{thm:MCG} Let $\rho$ be an element of the mapping class group of $S^2_{4,*}$ that fixes $\ast$. Then, with coefficients in $\F_2$, the immersed curve invariant is natural with respect to the mapping class group action:
	\[\wt{\BN}(\rho T; \F_2) = \rho \wt{\BN}(T ; \F_2).\]
\end{thm}

\begin{wrapfigure}[24]{R}{0.32\textwidth}
\centering
\includegraphics[width=0.24\textwidth]{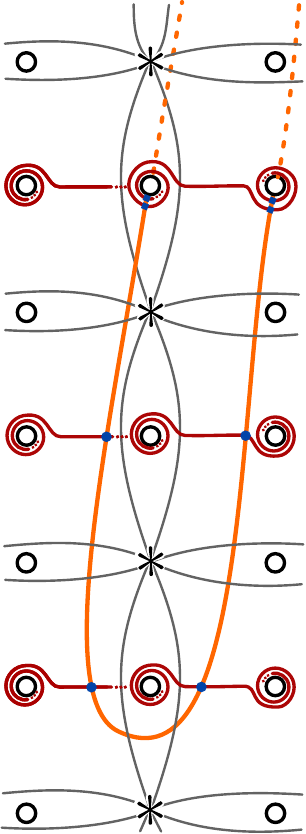}
\caption{Lift of the Lagrangians in \cref{fig:wrapPair} to the covering space $\R^2 \setminus \left(\frac{1}{2}\Z \right)^2$.}
\label{fig:wrapPairLift}
\end{wrapfigure}

Two remarks about the above theorem. First, the proof of naturality with coefficients in an arbitrary field $\k$ is set to appear in a forthcoming article by the same authors, \cite{KWZ22b}, and we will make use of this latter version. Second, although we have equality of bigraded invariants, it is not clear how an arbitrary mapping class group element $\rho$ acts on the bigradings; often, knowledge of the underlying 1-manifold $\rho \wt{BN}(T;\F_2)$ allows one to deduce the bigrading information.

\begin{rmk} Here is how we use the immersed curve theory in practice. First, we will mainly worry about non-compact curves. Such curves have trivial local systems \cite{KWZ19}, so we do not need to worry about this additional structure.
Second, in practice, the chain complex for Lagrangain intersection homology has trivial differential: indeed, any bigon contributing to the differential in $\CF(\gamma_1, \gamma_2)$ can be eliminated by an isotopy. Thus, for appropriately chosen representatives of the immersed curves, the group $\HF(\wt{\BN}(mT_1), \wt{\BN}(T_2))$ is, as a $\k$-module, freely generated by the intersection points of $\wt{\BN}(mT_1 ; \k)$ and $\wt{\BN}(T_2 ; \k)$.
Third, the adjective ``wrapped" in the theorem indicates that, in $\HF(\gamma_1, \gamma_2)$, the curve $\gamma_1$ must have every non-compact end wrapping infinitely around its corresponding puncture, as in \cref{fig:wrapPair}. Thus, the wrapped immersion $i^\spir$ in \cref{con:immersion} is chosen for the first entry in $\HF(\wt{\BN}(T_1), \wt{\BN}(T_2))$. We will drop the nonstandard $\spir$ notation henceforth.
Finally, it is convenient to lift the curves $\wt{\BN}(mT_1 ; \k)$ and $\wt{\BN}(T_2 ; \k)$ to the following covering space:
\[\R^2\setminus ({\scalebox{1}{$\frac{1}{2}$}} \Z)^2  \xra{\alpha} T^2_{4,*} \xra{\beta} S^2_{4,*},\]
where $\beta$ the usual double cover corresponding to the hyperelliptic involution and $\alpha$ is the universal Abelian cover. Intersection points between two curves $\gamma_1$ and $\gamma_2$ in $S^2_{4,*}$ correspond to intersection points of one lift of $\gamma_2$ with the union of all the lifts of $\gamma_1$. See \cref{fig:wrapPairLift}.
\end{rmk}

\begin{notn} For $n \in \Q \cup \{\infty\}$, let $\bf{a}_n$ denote the line of slope $n$ through a non-special half-integer lattice point in the covering space $\R^2 \setminus \left(\frac{1}{2}\Z\right)^2 \ra S^2_{4,*}$.
\end{notn}

Recall that $Q(n)$ is the $n$-rational tangle in \cref{fig:nTangle}.

\begin{prop}[\!\!\cite{KWZ19}] Modulo grading, $\wt{\BN}(Q(n) ; \k) = \bf{a}_n$.
\end{prop}

We are now ready to analyze the immersed curve invariants of cap-trivial tangles. 

\begin{notn} Given curve $\gamma \in \wrapFuk(S^2_{4,*})$, let $\gamma^a$ denote the non-compact components of $\gamma$. Likewise, given a curve-like type D structure $\Rd$, let $\Rd^a$ be the type D structure corresponding to the non-compact curve.
\end{notn}

Since Khovanov homology detects the unknot \cite{KM11}, cap-triviality of $T$ is equivalent to the immersed curve $\wt{\BN}(T; \k)$ having the property 
\begin{equation}\label{eq:singTow}
\HF(\wt{\BN}(Q_\infty), \wt{\BN}(T)) = \HF({\bf{a}}_\infty, \wt{\BN}^a(T)) \iso \k[H],
\end{equation}
the reduced Bar-Natan homology of the unknot.

\begin{lem}\label{lem:rdleaf} Suppose $T$ is cap-trivial. Then, up to action by $\tau$, the graph of $\Rd^{a}(T)$ is either equal to $\bullet$ or else it has a leaf that is connected to the rest of $\Rd^{a}(T)$ by an edge labelled $D$:
\[ \cdots \xra{D} \bullet \hspace{1cm} \text{ or} \hspace{1cm} \bullet \xra{D} \cdots.\]
The other leaf of $\Rd^{a}(\tau)$ is $\bullet$. 
\end{lem}
\begin{proof} This follows from \cref{eq:singTow} and a straightforward combinatorial analysis of intersections in $S^2_{4,*}$.
\end{proof}
%\[ \cdots \xra{S} \bullet \hspace{1cm} \text{ or} \hspace{1cm} \bullet \xra{S} \cdots.\]
%\begin{proof} Since $\Hom(\bullet, \Rd^{nc}(\tau)) \simeq \k[H] \simeq \HF(\mathbf{a}(\infty), \wt{\BN}(\tau))$, it follows that $\Rd^{nc}(\tau)$ intersects $\del \mathrm{nbhd}(\mathbf{a}(\infty))$ in a single point; call it $x$. Therefore, inside the neighbourhood, sufficient applications of $\rho$ or $\rho^{-1}$ reduce the picture to one that looks like \cref{fig:singint} inside $\mathrm{nbhd}(\mathbf{a}(\infty))$. Suppose $\rho$ has been applied enough times so that the rank of $\Rd(\tau)$ is minimal, or, equivalently, so that the number of intersection points between $\wt{\BN}(\tau)$ and the two parametrizing arcs is minimal.
%
%The first claim of the lemma is now that the arc of immersed curve $S^2_{4,*} \setminus \mathrm{nbhd}(\mathbf{a}(\infty))$ that leaves $x$ must intersect the parametrizing arc for $\bullet$ before intersecting the parametrizing arc for $\circ$. Suppose not. Then the arc must necessarily intersect the parametrizing arc for $\circ$ in two vertices, as in \cref{fig:taucase}. But this contradicts the minimality of $\mathrm{rk}\Rd(\tau)$.
%\begin{figure}[h]
%	\centering
%	\includegraphics[scale=0.7]{./Figs/taucase}
%	\label{fig:taucase}
%\end{figure}
%For the second claim, by uniqueness of the intersection point $x$, it follows that the other end of the immersed curve must be at the puncture that is completely surrounded by the parametrizing arc for $\bullet$. The curve is then restricted to intersect the parametrizing arc for $\circ$ next.
%\end{proof}

\begin{lem}\label{lem:capTrivialConstraints} Let $T$ be a cap-trivial tangle and let $\Rd^{a}(T;\k)$ be its curve-like non-compact type D structure. Then the following are forbidden configurations of $\Rd^{a}(T;\k)$, by which we mean that no subquotient of $\Rd^{a}(T;\k)$ can belong to the following list:
\[\begin{split}
								\hspace{0.5cm}&\hspace{0.5cm}	\circ \xra{G^nSS_\circ} \circ		\\
	\circ \xra{G^nS_\circ} \bullet		\hspace{0.5cm}&\hspace{0.5cm}	\bullet \xra{G^nS_\bullet} 	\circ 		\\
	\circ \xra{G^nD_\circ} \circ			\hspace{0.5cm}&\hspace{0.5cm}	\bullet \xra{G^{n-1}SS_\bullet}\bullet	\\
	{\odot} \ra \circ \leftarrow {\odot}		\hspace{0.5cm}&\hspace{0.5cm}	{\odot} \leftarrow \circ \ra {\odot}
\end{split}\]
for $n \geq 1$, and where ${\odot}$ stands in for either $\circ$ or  $\bullet$.
\end{lem}

\begin{notn} We use the term $\circ$-elbows for configurations of the form $\begin{tikzcd}[sep=small, cramped] {}\rar & \circ & {}\lar \end{tikzcd}$ or $\begin{tikzcd}[sep=small, cramped] {} &\lar \circ\rar & {} \end{tikzcd}$, and $\bullet$-elbows are defined likewise. Thus, the above lemma says that the curve-like representative of  $\Rd^{a}(T;\k)$ may not contain any $\circ$-elbows and can contain only arrows labelled with  $D_\bullet, S_\circ, S_\bullet, SS_\circ,$ or $D^n_\bullet$, for $k \geq 1$. 
\end{notn}

\begin{proof} This follows by considering the pairing with $\bf{a}_\infty$. Note that, for $n \geq 2$, we have 
\[\begin{array}{ccc} D^n = (-G)^{n-1}D  	&S^{2n-1} = G^{n-1}S	&SS^n = G^{n-1}SS\end{array},\]
so we consider segments of immersed curves corresponding to high powers of $S_\odot, D_\circ$ and $SS_\odot$. \Cref{fig:verbot1} shows that such a segment generates an essential intersection point between $\bf{a}_\infty = \wt{\BN}(m \iRes)$ and $\wt{\BN}(T)$, so that the resulting homology cannot be $\k[H]$, contradicting cap-triviality.

\begin{figure}[h]
	\labellist
	\hair 2pt
	\pinlabel $S_\circ^3$ at 86 -10
	\pinlabel $SS_\bullet$ at 310 -10
	\pinlabel $D_\circ^2$ at 525 -10
	\endlabellist
	\centering
	\includegraphics[width=0.7\textwidth]{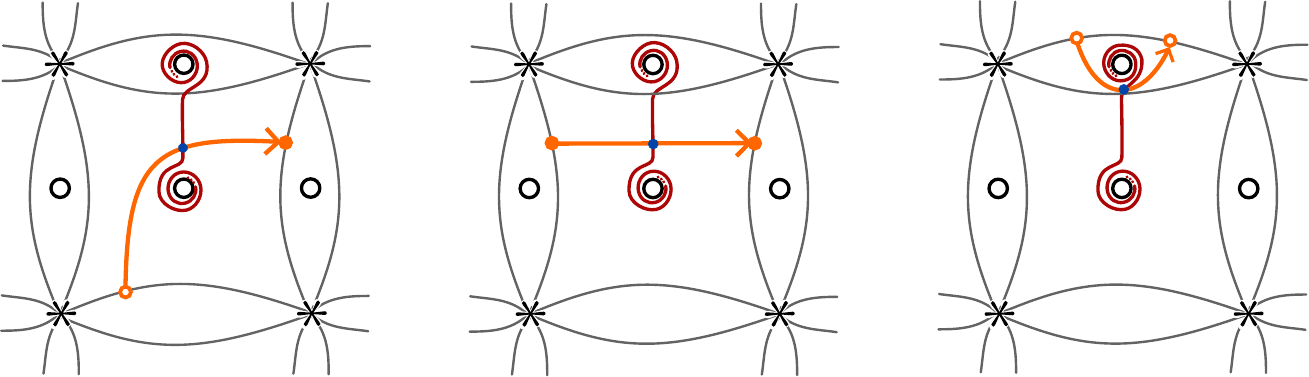}
	\caption{Forbidden algebra elements for cap-trivial curves.}
	\label{fig:verbot1}
\end{figure}

The analysis of $\circ$-elbows is similar. Note first that, since $\Rd^{a}(T)$ immerses as a curve that intersect the parametrizing arcs transversely, the $\circ$-elbow cannot be 
\[\odot \xra{Z^k} \circ \xla{Z^l} \odot \hspace{0.5cm} \text{or} \hspace{0.5cm} \odot \xla{Z^k} \circ \xra{Z^l} \odot,\]
where $k, l \geq 0$. The two arrows therefore have different labels. We are thus left to consider the following four cases
\[
\circ \xra{D^k} \circ \xla{SS^l} \circ 	\quad  \circ \xla{D^k} \circ \xra{SS^l} \circ  	\quad \bullet \xra{S^{2l-1}} \circ \xla{D^k} \circ	\quad  \bullet \xla{S^{2l-1}} \circ \xra{D^k} \circ.
\] 
Moreover, by our work so far, we only need to consider the case $l = k = 1$. \Cref{fig:verbot2} shows that each elbow in the above list generates an essential intersection point with $\bf{a}_\infty$.
\end{proof}

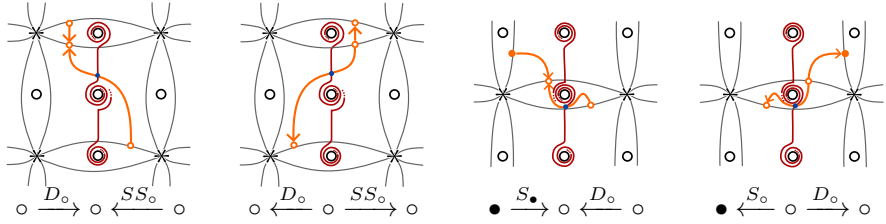
\begin{figure}[h]
	\labellist
	\hair 2pt
	\pinlabel {$\circ \xra{D_\circ} \circ \xleftarrow{SS_\circ} \circ$} at 91 -13
	\pinlabel {$\circ \xleftarrow{D_\circ} \circ \xra{SS_\circ} \circ$} at 316 -13
	\pinlabel {$\bullet \xra{S_\bullet} \circ \xleftarrow{D_\circ} \circ$} at 541 -13
	\pinlabel {$\bullet\xleftarrow{S_\circ}\circ\xra{D_\circ}\circ$} at 760 -13
	\endlabellist
	\centering
	\includegraphics[width=0.8\textwidth]{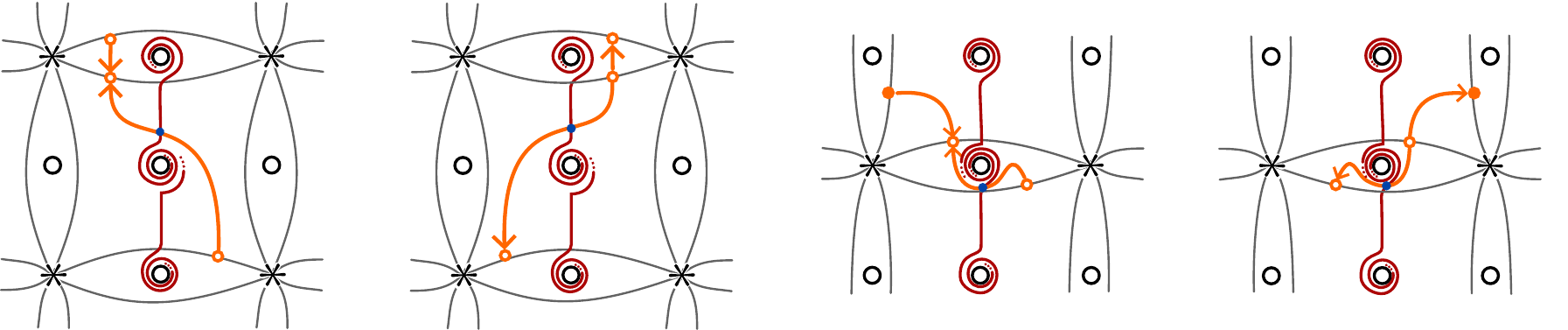}
	\caption{Forbidden elbows for cap-trivial curves.}
	\label{fig:verbot2}
\end{figure}

\subsection{Structural implications for $\Cb$}

In this subsection we prove what we call the ``elbow" and the ``homogeneous chain" lemmas. These are, computationally, the most involved arguments and form the technical core of the paper.

\begin{lem}[Elbow Lemma]\label{lem:elbow} Suppose $T$ is cap-trivial and $\Rd(T)$ is a curve-like type D structure that contains an elbow such as
\[\begin{tikzcd}[column sep = small, row sep = tiny, cramped] 
	\circ \ar[dr, "S_\circ"]			&	\\
							& \bullet \\	
	\bullet \ar[ur, "D_\bullet^k"]	& 
\end{tikzcd}
	\hspace{0.5cm} \text{ or } \hspace{0.5cm} 
\begin{tikzcd}[column sep = small, row sep = tiny, cramped] 		
									& \circ \\	
\bullet\ar[dr, "D_\bullet^k"] \ar[ur, "S_\bullet"] 	&	   \\ 		
									& \bullet	
\end{tikzcd}\]
Then the type D structure $\Cb(\Rd(T))$ splits along the $D^k_\bullet$ edge, by which we mean that we have a chain homotopy
\[\Cb	\left[ 
		\begin{tikzcd}[column sep=small, row sep = tiny, cramped] 
		 	\circ \ar[dr, "S_\circ"]		&	\\
								& \bullet \\	
			\bullet \ar[ur, "D_\bullet^k"]& 
		\end{tikzcd} 
	\right] \htp 
\Cb \left[ 
		\begin{tikzcd}[column sep=small, row sep = tiny, cramped] 
		\circ \ar[dr, "S_\circ"]	&	\\
						& \bullet \\	
		\bullet 			& 
		\end{tikzcd} 
	\right], \hspace{0.5cm} \text{ or } \hspace{0.5cm}
\Cb \left[ 
		\begin{tikzcd}[column sep=small, row sep = tiny, cramped] 
											& \circ \\	
		\bullet\ar[dr, "D_\bullet^k"] \ar[ur, "S_\bullet"] 	&	   \\ 		
											& \bullet
		\end{tikzcd} 
	\right] \htp
\Cb \left[ 
		\begin{tikzcd}[column sep=small, row sep = tiny, cramped] 
		 					& \circ	\\
		\bullet \ar[ur, "S_\bullet"]	&  \\	
							& \bullet
		\end{tikzcd} 
	\right].
\]
\end{lem}

We stress that the statement is not about the effect of the operator $\Cb$ on the specific type D structures $\bullet \xra{D^k_\bullet} \bullet \xleftarrow{S_\circ} \circ$ or $\bullet \xleftarrow{D^k_\bullet} \bullet \xra{S_\circ} \circ$ (which are simple instances of our lemma and, where the splitting follows from a straightforward computation), but rather about the effect of $\Cb$ on a type D structure that contains one of the aforementioned type D structures as a subgraph.

%%%%%%%%%%%%%%%%%%%%%%
% 	ELBOW LEMMA
%%%%%%%%%%%%%%%%%%%%%%

\begin{proof} The proof is a case verification. First, we walk through the graph of $\Rd(T)$ to begin generating the rest the type D structure. If we do not encounter any leaves, then we claim that the graph of $\Rd(T)$ near the elbow is one of the following:
\begin{enumerate}
	\item[Case (i)] \[\begin{tikzcd}[column sep = small, row sep = tiny]
					& \odot \rar["S/SS"]	& \circ \rar["D"]	& \circ \ar[dr, "S"]		&	\\
					&					&			&					& \bullet \\
	\odot\rar["S/SS"]	& \circ \rar["D"]		& \circ \rar["S"]	& \bullet \ar[ur, "D^k"]	&
	\end{tikzcd}\]
	\item[Case (ii)]  \[\begin{tikzcd}[column sep = small, row sep = tiny]
					&			&						& \circ \rar["D"]	& \circ \rar["S/SS"]	& \odot	\\
					&			& \bullet \ar[dr, "D^k"]\ar[ur, "S"]	&			&				& 		\\
					&			&						& \bullet 		&				& 		\\
	\odot \rar["S/SS"]	& \circ \rar["D"]	&\circ \ar[ur, "S"]			&			&				&
	\end{tikzcd}\]
	\item[Case (iii)] \[\begin{tikzcd}[column sep = small, row sep = tiny]
								& \circ \rar["D"]		& \circ \rar["S/SS"]	&\Cb_D(\odot)		& \\
		\bullet \ar[ur, "S"]\ar[dr, "D^k"]	&				&				& 				& \\
								& \bullet \rar["S"]	&\circ \rar["D"]		& \circ \rar["S/SS"]	&\Cb_D(\odot)	
	\end{tikzcd}\]
\end{enumerate}
If we do encounter leaves, then we use the same model, but with one of the arrows above labelled with the 0 cobordism; the analysis in this case is simpler. The claim follows from \cref{lem:capTrivialConstraints} and the fact that $\Rd(T)$ is curve-like, i.e. its graph has degree $2$ at each vertex that is not a leaf. Let us spell this out in one instance and trust that the reader can fill in the missing casework. Suppose that we have on our hands an elbow of the form 
\[\begin{tikzcd}[column sep = small, row sep = tiny, cramped] 
	\circ \ar[dr, "S_\circ"]			&	\\
							& \bullet \\	
	\bullet \ar[ur, "D_\bullet^k"]	& 
\end{tikzcd}\]
and let us walk through the graph $\Rd(T)$ starting at the vertex $\circ$. Since there are no $\circ$-elbows and the differential squares to 0, every subgraph $\circ \xra{S_\circ} \bullet$ sits inside a subgraph $ \xra{*} \circ \xra{D_\circ} \circ \xra{S_\circ} \bullet$, where $* \in \{S_\bullet, SS_\circ\}$, assuming that neither of the two $\circ$ vertices are leaves. Likewise subgraphs $\bullet \xra{S_\bullet} \circ$ sit inside subgraphs $\bullet \xra{S_\bullet} \circ \xra{D_\circ} \circ \xra{*}$. Therefore, our elbow sits inside a subgraph
\[\begin{tikzcd}[column sep = small, row sep = tiny, cramped] 
					& \odot \rar["S/SS"]	& \circ \rar["D"]	& \circ \ar[dr, "S"]		&	\\
					&				&			&					& \bullet \\
					&				&			& \bullet \ar[ur,"D^k"]		&
\end{tikzcd}\]
Let us now walk away from the other end of the elbow, namely the $\bullet$ end. This vertex is either an elbow or not. Suppose not (this is where the argument splits into casework that we entrust to the reader). Then the same \cref{lem:capTrivialConstraints} tells us that there do not exist arrows labelled $SS_\bullet$ inside type D invariants of cap-trivial tangles, so our elbow necessarily sits as a subgraph of
\[\begin{tikzcd}[column sep = small, row sep = tiny, cramped] 
					& \odot \rar["S/SS"]	& \circ \rar["D"]	& \circ \ar[dr, "S"]		&	\\
					&				&			&					& \bullet \\
					&				& \circ \rar["S"]	& \bullet \ar[ur,"D^k"]		&
\end{tikzcd}\]
Applying the same reasoning we did for the walk away from the $\circ$ vertex in the elbow, we see that we are in case (i). 

Assuming the claim, the rest of the proof is a simplification of $\Cb_D(\Rd(T))$ in each of the three cases, by using Lemmas \ref{lem:cancel} and \ref{lem:cleanUp} and the time-tested technique of the diagram chase.

\subsubsection*{\underline{Case (i)}:}

In this case, $\Cb_D(\Rd(T))$ contains the subquotient complex
\[\begin{tikzcd}[row sep = tiny]
	& 							& \Cb_D(\odot)\rar["\Cb_D(S/SS)"]	& \Cb_D(\circ) \rar["\Cb_D(D)"]	& \Cb_D(\circ) \ar[dr, "\Cb_D(S)"]	& \\
	&							&							&						&							& \Cb_D(\bullet) \\
	&\Cb_D(\odot)\rar["\Cb_D(S/SS)"]	& \Cb_D(\circ) \rar["\Cb_D(D)"]		& \Cb_D(\circ) \rar["\Cb_D(S)"]	& \Cb_D(\bullet) \ar[ur, "\Cb_D(D^k)"']&
\end{tikzcd}\]

We use \cref{lem:cancel} to cancel the isomorphisms internal to each copy of $\Cb_D(\circ)$, like we did in \cref{eg:CoRes}. The resulting diagram chase yields the following type D structure (note in particular, that the induced long differential from $\Cb_D(\odot)$ to $\Cb_{F^{pre}}(\circ)$ vanishes):

\[\begin{tikzcd}[row sep = tiny]
	&	&							& \Cb_D(\odot)\rar["\Cb_D(S/SS)_*"] 							& \Cb_{F^{pre}}(\circ) \rar["\Cb_{F^{pre}}(D)"]	\ar[rrd, bend right=20]	& \Cb_{F^{pre}}(\circ) \ar[dr, "\Cb_D(S)_*"]	& \\
	&	&							&															&													&								& \Cb_D(\bullet) \\
	&	&\Cb_D(\odot)\rar["\Cb_D(S/SS)_*"]	& \Cb_{F^{pre}}(\circ) \rar["\Cb_{F^{pre}}(D)"]\ar[rr, bend right, "\Cb_D(S|D)_*"]	& \Cb_{F^{pre}}(\circ) \rar["\Cb_D(S)_*"]						& \Cb_D(\bullet) \ar[ur, "\Cb_D(D^k)"']	&
\end{tikzcd}\]

Continuing with a cancellation along the isomorphism in the subgraphs $\Cb_D(\bullet)$, we obtain the following type D structures, where the morphisms are described in \cref{fig:elbowi1}; note importantly that the map $\beta$ therein is either 0 or $\col{0}{0}$, depending on the value of $\odot$.

\begin{equation}\begin{tikzcd}[row sep = tiny]\label{eq:elbowi1}
	&	&						&\rar["\Cb_D(S/SS)_*"]									& \Cb_{F^{pre}}(\circ) \rar["\Cb_{F^{pre}}(D)"]								& \Cb_{F^{pre}}(\circ) \ar[rd, "\Cb_F(S)"]		&	\\
	&	&						&													&															&									& \Cb_F(\bullet) \\
	&	& \rar["\Cb_{F^{pre}}(S/SS)"]	& \Cb_{F^{pre}}(\circ) \rar["\Cb_{F^{pre}}(D)"]	\ar[rrru, bend right=55]	& \Cb_{F^{pre}}(\circ) \rar["\Cb_F(S)"]\ar[rru, "\Cb_F(D^k|S)", bend right=45]	& \Cb_F(\bullet)	 \ar[ur, "\Cb_F(D^k)"']		&
\end{tikzcd}\end{equation}

\begin{figure}[h] 
\centering
\includegraphics[width=\textwidth]{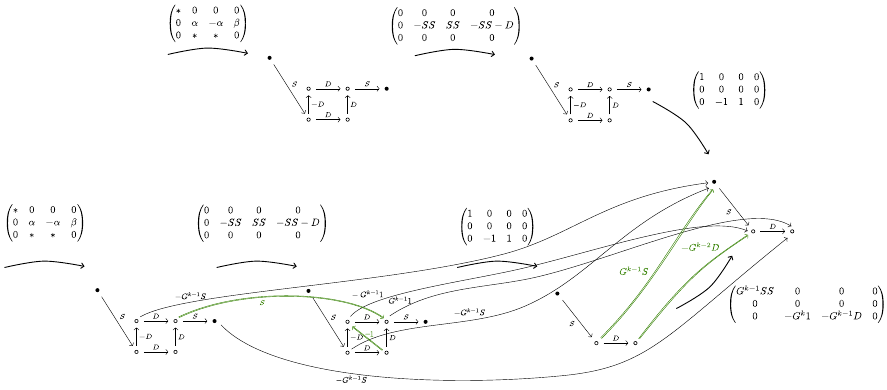}
\caption{Full description of the type D structure in \cref{eq:elbowi1}; the map $\beta$ is $0$ or $\left(\begin{array}{c}0\\ 0\end{array}\right)$.}
\label{fig:elbowi1}
\end{figure}

Cleaning-up along the doubled green arrows indicated in \cref{fig:elbowi1}, and following with the clean-ups indicated in \cref{fig:elbowi2} %results in a type D structure isomorphic to the following, which 
completes the proof of case (i).

%\begin{equation}\begin{tikzcd}[row sep = tiny]\label{eq:elbowiFinal}
%	&	&								&\Cb_D(\odot) \rar["\Cb_D(S/SS)_*"]			& \Cb_{F^{pre}}(\circ) \rar["\Cb_{F^{pre}}(D)"]		& \Cb_{F^{pre}}(\circ) \ar[rd, "\Cb_F(S)"]	&			\\
%	&	&								&									&									&								& \Cb_F(\bullet) \\
%	&	& \Cb_D(\odot) \rar["\Cb_{F^{pre}}(S/SS)"]	& \Cb_F(\circ) \rar["\Cb_F(D)"]				& \Cb_F(\circ) \rar["\Cb_F(S)"]				& \Cb_F(\bullet)						&
%\end{tikzcd}\end{equation}

\begin{figure}
\centering
\includegraphics[width=\textwidth]{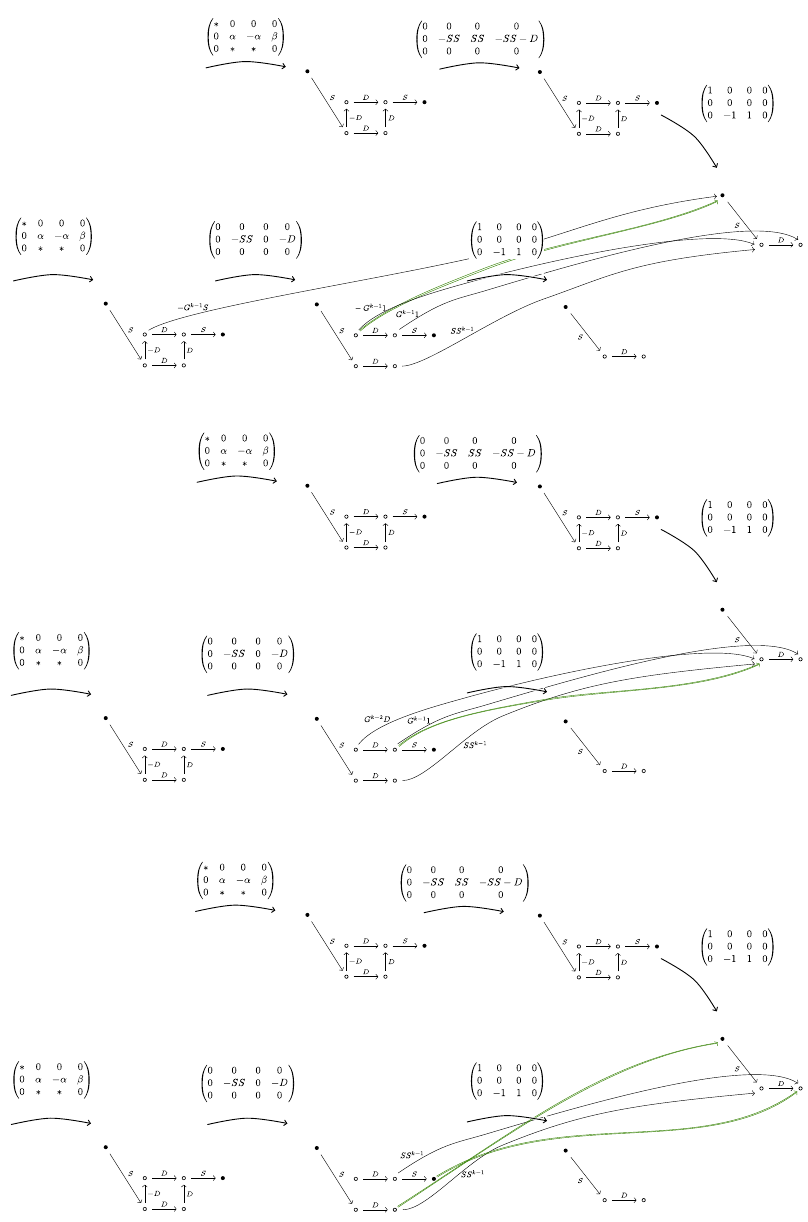}
\caption{The clean-up required to prove case (i).}
\label{fig:elbowi2}
\end{figure}

\pagebreak
\subsubsection*{\underline{Case (ii)}:}

In this case, $\Cb_D(T)$ contains the subquotient
 \[\begin{tikzcd}[row sep = tiny]
							&						&											& \Cb_D(\circ) \rar["\Cb_D(D)"]	& \Cb_D(\circ) \rar["\Cb_D(S/SS)"]	& \Cb_D(\odot)\\
							&						& \Cb_D(\bullet) \ar[dr, "\Cb_D(D^k)"]\ar[ur, "\Cb_D(S)"]	&						&							& \\
							&						&											& \Cb_D(\bullet) 			&							& \\
\Cb_D(\odot)\rar["\Cb_D(S/SS)"]	& \Cb_D(\circ) \rar["\Cb_D(D)"]	& \Cb_D(\circ) \ar[ur, "\Cb_D(S)"]					&						&							&
\end{tikzcd}\]

It turns out that for this case, we only need to work with the subquotient

 \[\begin{tikzcd}[row sep = tiny]
												& \Cb_D(\circ)		\\
	 \Cb_D(\bullet) \ar[dr, "\Cb_D(D^k)"]\ar[ur, "\Cb_D(S)"]	&				\\
												& \Cb_D(\bullet) 	\\
	 \Cb_D(\circ) \ar[ur, "\Cb_D(S)"]						&
\end{tikzcd}\]

Cancelling (in two steps, as always) the isomorphisms internal to $\Cb_D(\bullet)$ results in the type D structure in \cref{fig:elbowii}:

\begin{figure}[h]
	\centering
	\includegraphics[width=\textwidth]{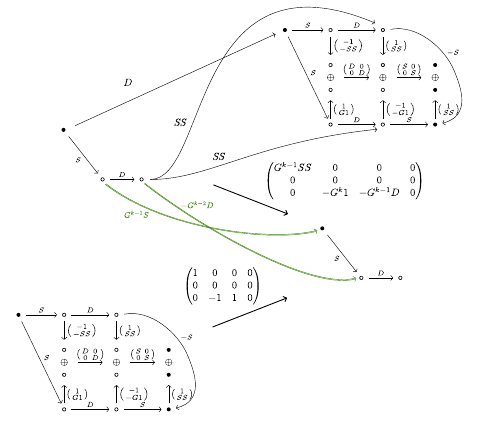}
	\caption{Clean-up of the type D structure in case (ii).}
	\label{fig:elbowii}
\end{figure}

In case $k\geq 2$, cleaning up along the indicated doubled green arrows completes the proof of case (ii) of the elbow lemma. If $k=1$, replacing the morphism $-D^{k-2}$ with the identity yields the same clean-up result.

\subsubsection*{\underline{Case (iii)}:}

In this case, $\Cb_D(T)$ contains the subquotient
\[\begin{tikzcd}[row sep = tiny]
													& \Cb_D(\circ) \rar["\Cb_D(D)"]		& \Cb_D(\circ) \rar["\Cb_D(S/SS)"]	&\Cb_D(\odot)					& 			\\
		\Cb_D(\bullet) \ar[ur, "\Cb_D(S)"]\ar[dr, "\Cb_D(D^k)"]	&							&							& 							& 			\\
													& \Cb_D(\bullet) \rar["\Cb_D(S)"]	& \Cb_D(\circ) \rar["\Cb_D(D)"]		& \Cb_D(\circ) \rar["\Cb_D(S/SS)"]	&\Cb_D(\odot)	
\end{tikzcd}\]

As in case (ii), we can restrict our attention to a smaller subquotient. Cancelling the isomorphisms internal to $\Cb_D(\bullet)$ yields the following type D structure (\cref{fig:elbowiii}), which cleans up as in case (ii):

\begin{figure}[h]
	\centering
	\includegraphics[width=\textwidth]{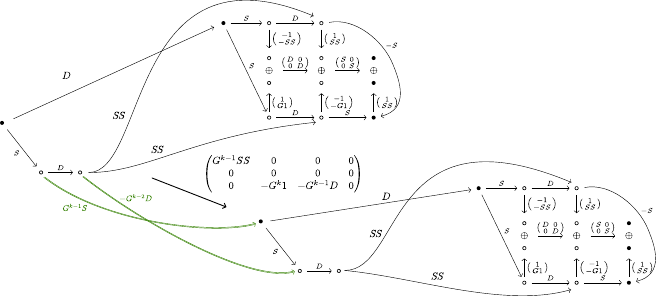}
	\caption{Clean-up of the type D structure in case (iii).}
	\label{fig:elbowiii}
\end{figure}

This completes the proof of the elbow lemma.
\end{proof}

%%%%%%%
%% Below I have the first step in the computation for the trefoil tangle
%%%%%%%
%\begin{cor} As an example, consider the type D structure associated with the strong inversion on the left-handed trefoil
%\[\Rd(T_{3_1}) := 
%\begin{tikzcd}[column sep = small, row sep = small]
%						&				&				&					&\bullet \ar[dr, "D_\bullet"]	&			\\
%						&				&				&					&					& \bullet		\\
%	\bullet \rar["S_\bullet"]	&\circ \rar["D_\circ"]	&\circ \rar["SS_\circ"]	&\circ \rar["D_\circ"]		&\circ \ar[ur, "S_\circ"]		&
%	\end{tikzcd}\]
%By the elbow lemma, we have
%\[\Cb(T_{3_1}) = \Cb(\bullet) \sqcup 
%\begin{tikzcd}[column sep = small]
%	\Cb \bigg( \bullet \rar["S_\bullet"]	&\circ \rar["D_\circ"]	&\circ \rar["SS_\circ"]	&\circ \rar["D_\circ"]	&\circ \rar["S_\circ"]	&\bullet \bigg).
%\end{tikzcd}\]
%\end{cor}

%\[\Cb(C_4) \simeq 
%\begin{tikzcd}[column sep = small]
%	\Cb \bigg(\bullet \rar["S_\bullet"]	&\circ \rar["D_\circ"]	&\circ \rar["SS_\circ"]	&\circ \rar["D_\circ"]	&\circ \rar["S_\circ"]	&\bullet \bigg)	\\
%	\Cb \bigg( \bullet \rar["S_\bullet"]	&\circ \rar["D_\circ"]	&\circ \rar["SS_\circ"]	&\circ \rar["D_\circ"]	&\circ \rar["S_\circ"]	&\bullet \bigg)
%\end{tikzcd}\]

We can restate the elbow lemma in terms of a simpler category of type D structures. \label{def:tightAtD} Given a type D structure $\Rd$, consider the type D structure $\Rd|_{D_\bullet = 0}$, obtained from $\Rd$ by formally setting the $D_\bullet$-labelled arrows to 0. Then the elbow lemma is equivalent to:

\begin{prop}\label{thm:factorization} Let $T$ be cap-trivial. Then
\[\Cb(\Rd(T)) = \Cb(\Rd(T)|_{D_\bullet = 0}).\]
\end{prop}

\begin{defn}\label{def:tightAtD} Given an immersed curve $\wt{\BN}(T) \in \Fuk(S^2_{4,*})$, let $\Rd(T)$ be the corresponding curve-like type D structure. The immersed curve $\wt{\BN}(T)|_{D_\bullet=0}$ is defined to be the embedding of $\Rd(T)|_{D_\bullet = 0}$ in $S^2_{4,*}$ according to \cref{con:immersion}. 
\end{defn}

It is not hard to see that the above definition agrees with the one given in the introduction. The previous proposition implies \cref{thm:curveFactorization}:

\begin{cor*}%[{\cref{thm:curveFactorization}}] 
Let $T$ be cap-trivial, then
\[\Cb(\wt{\BN}(T)) = \Cb(\wt{\BN}(T)|_{D_\bullet=0}).\]
\end{cor*}

\begin{defn} Let $\Rd(T)$ be the curve-like type D structure invariant of a cap-trivial tangle $T$. A \textbf{homogeneous chain} is a connected subgraph of $\Rd(T)$ that contains no $\bullet$ vertices. The number of vertices in this subgraph is the length of the homogeneous chain.
\end{defn}

\begin{rmk} Since there are no $\circ$-elbows, maximal homogeneous chains are necessarily of the form
\[\circ \xra{D_\circ} \cdots \xra{D_\circ} \circ,\]
where the edge labels alternate between $D_\circ$ and $SS_\circ$ and the number of vertices is even. Moreover, by \cref{lem:capTrivialConstraints}, every homogeneous chain that does not contain a leaf is a subgraph of the following
\[\bullet \xra{S_\bullet} \circ \xra{D_\circ} \cdots \xra{D_\circ} \circ \xra{S_\circ} \bullet,\]
where each $\bullet$ endpoint is the corner of a $\bullet$-elbow.
\end{rmk}

We can now state the second lemma of this section. We continue using the notation in \cref{eq:2Compact}.

\begin{lem}[Homogeneous chain lemma]\label{lem:homogChain} Let $\Rd(T)$ be the curve-like type D structure of a cap-trivial tangle. Suppose we have a homogeneous chain contained in the following type D subgraph
\[\begin{tikzcd}[row sep=tiny]
	\prescript{q}{}{\bullet_h} \rar["S"] 	&\circ \rar["D"] 	&\dots \rar["D"]	 &\circ \rar["S"]	&\bullet
\end{tikzcd}\]
Then the cabling operator applied to this subgraph yields a type D structure that is homotopic to
\[\underbrace{\prescript{q+4}{}{C_{h+2}} \oplus \prescript{q+8}{}{C_{h+4}}\dots \oplus \prescript{q+2k}{}{C_{h+k}}}_{\frac{k}{2}}, \]
where $k$ is the length of the homogeneous chain.
\end{lem}
\begin{proof} The proof is as in \cref{subsec:unknot}: a diagram chase using \cref{lem:cancel,lem:cleanUp}. By the preceding remark, we may apply the elbow lemma to conclude that the type D structure splits as
\[\Cb(\Rd(T)) \simeq \Cb(\prescript{q}{}{\bullet_h} \xra{S} 	\circ \xra{D} 	\dots \xra{D}	 \circ \xra{S}	\bullet) \oplus \Rd'(T).\]
Thus, up to homotopy, $\Cb(\prescript{q}{}{\bullet_h} \xra{S} 	\circ \xra{D} 	\dots \xra{D}	 \circ \xra{S}	\bullet)$ is a sub-object of $\Rd(T)$. We apply our model $\Cb_D$ and notice that, in the type D structure
\[\begin{tikzcd}[row sep=tiny]
	\Cb_D(\prescript{q}{}{\bullet_h}) \rar["\Cb_D(S)"] 	&\Cb_D(\circ) \rar["\Cb_D(D)"] 	&\dots \rar["\Cb_D(D)"]	 &\Cb_D(\circ) \rar["\Cb_D(S)"]	&\Cb_D(\bullet),
\end{tikzcd}\]
canceling isomorphisms internal to $\Cb_D(\circ)$ does not result in any long induced maps out of $\Cb_D(\bullet)$, so the above type D structure is homotopic to
\[\begin{tikzcd}[row sep=tiny]
	\Cb_F(\prescript{q}{}{\bullet_h}) \rar["{\Cb_F(S)}"] 	&\Cb_F^{pre}(\circ) \rar["\Cb_F^{pre}(D)"] 	&\dots \rar["\Cb_F^{pre}(D)"] \ar[rr, bend right]	 &\Cb_F^{pre}(\circ) \rar["{\Cb_F(S)}"]	&\Cb_F(\bullet),
\end{tikzcd}\]
where the indicated maps are given by
\[\begin{split}
	\Cb_F(S_\bullet)		&= \begin{pmatrix} D_\bullet	&0	&0	&0	\\ 0	&0	&0	&0	\\	0	&0	&SS_\circ	&0\end{pmatrix}	\\
	\Cb_F^{pre}(D_\circ)		&= \begin{pmatrix} 0	&0	&0	&0	\\ 0	&-SS		&SS	&-SS-D	\\	0	&0	&0	&0\end{pmatrix}\\
	\Cb_F^{pre}(SS_\circ)	&= \begin{pmatrix} SS+D	&0	&0	&0	\\ 0	&-D	&D	&0	\\	0	&-G1	&G1	&0\end{pmatrix}\\
	\Cb_F(S_\circ)			&= \begin{pmatrix} 1	&0	&0	&0	\\ 0	&0	&0	&0	\\	0	&-1	&1	&0\end{pmatrix}
\end{split}\]
Cancelling the isomorphism $\Cb_F(S_\circ)$ and cleaning up along the doubled green arrows indicated below (in the case that the homogeneous chain has length 6) concludes the proof of the homogeneous chain lemma. 
\begin{figure}[h]
	\centering
	\includegraphics[width=\textwidth]{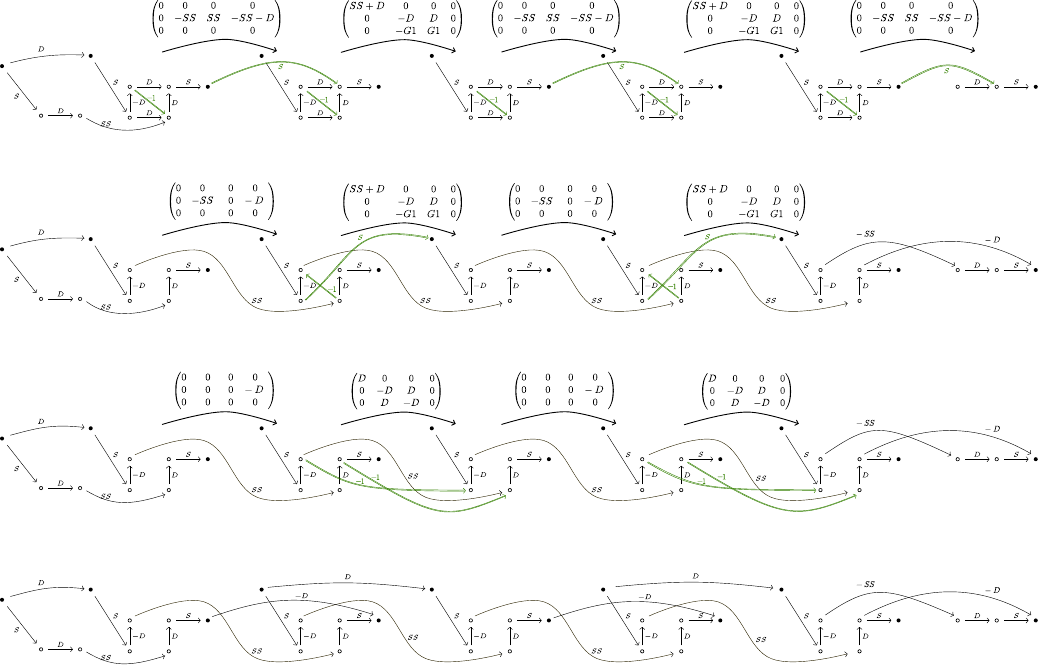}
	\caption{Clean-up in the proof of \cref{lem:homogChain}.}
\end{figure}
\end{proof}
\pagebreak
\subsection{Another example: the $(2,1)$-cables of the trefoil} \label{sec:trefoilComputation}

\begin{wrapfigure}[10]{R}{0.32\textwidth}
\includegraphics[scale=0.5]{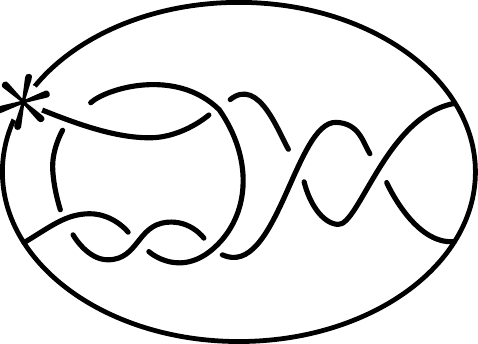}\caption{The tangle $T$.}
\label{fig:trefoilTangleMin}
\end{wrapfigure}

Let $T$ be the tangle associated with the strong inversion on the left-handed trefoil, framed so that $\Rd(T)$ has minimal rank; this particular framing is depicted in \cref{fig:trefoilTangleMin}. Let us compute $\Cb(\tau^nT)$, the Bar-Natan invariant of the tangle associated to the strong inversion on the $(2,1)$-cables of the left-handed trefoil, for $n\in \Z$; recall that $\tau$ is the Dehn twist in \cref{fig:DehnTwist}.

We have the following curve-like representative of $\Rd(\tau^nT)$, by \cref{thm:MCG}:
\[\Rd(\tau^nT) \simeq 
	\begin{cases}
		\begin{tikzcd}[column sep = small]
	&\circ \rar &\cdots \rar["SS"]	&\circ \rar["D"]	&\circ \rar["S"]	&\bullet				&			&				&			&			&	\\
	&		&				&			&			&\bullet \rar["S"] \uar["D"]	&\circ \rar["D"]	&\circ \rar["SS"]		&\circ \rar["D"]	&\circ \rar["S"]	&\bullet	
		\end{tikzcd} &\text{ if } n < 0\\
		\begin{tikzcd}[column sep = small]
			\bullet \rar["S"]			&\circ \rar["D"]	&\circ \rar["SS"]		& \cdots \rar	& \circ		&		\\
			\bullet \rar["S"] \uar["D"]	&\circ \rar["D"]	&\circ \rar["SS"]		&\circ \rar["D"]	&\circ \rar["S"]	&\bullet			
		\end{tikzcd} &\text{ if } n > 0\\
	\end{cases}		
\]
where the ambiguous homogeneous chain has length $|n|$.
By the elbow lemma, we have
\[\Cb(\Rd(\tau^nT)) \simeq 
\begin{tikzcd}
	\Cb(\bullet) \rar["\Cb(S_\bullet)"]	&\Cb(\circ) \rar["\Cb(D_\circ)"]	&\Cb(\circ) \rar["\Cb(SS_\circ)"]	&\Cb(\circ) \rar["\Cb(D_\circ)"]	&\Cb(\circ) \rar["\Cb(S_\circ)"]	&\Cb(\bullet)
\end{tikzcd}
\oplus \Cb(\tau^n \bullet).\]

By the homogeneous chain lemma, up to grading, the resulting type D structure consists of 2 copies of the compact curve $C$, together with the invariant associated with a $(2,1)$-cable of the $n$-framed unknot. 

In terms of immersed curves, the action of our operator $\Cb$ is given in \cref{fig:cabTrefoilFactorization}.

Finally, we include a proof of \cref{thm:geography}.

\begin{thm} Given a cap-trivial tangle $T$, the invariant $\Cb(\wt{\BN}(T; \F_c))$ is equal to a number of compact curves and one non-compact curve that is, up to mirroring and framing, one of
\[\wt{\BN}(T_{3_1}) \qquad \text{or} \qquad \wt{\BN}(\oRes).\]
\end{thm}
\begin{proof} By \cref{lem:rdleaf,lem:capTrivialConstraints}, the invariant $\wt{\BN}(T)|_{D_\bullet = 0}$ consists of a number (which may be 0) of arcs that pull tight in $\R^2 \setminus \left(\frac{1}{2}\Z \right)^2$ to the column containing the special puncture, together with $\wt{\BN}(Q_n)$, for some $n$. The cabling operator applied to this latter component is precisely the content of \cref{subsec:unknot}.
\end{proof}

\subsection{Lewark--Zibrowius' $\vartheta$-invariant}

In this section we discuss the $\vartheta$ invariant defined in \cite{LewZib22} and prove \cref{prop:CbRational}. 

As mentioned in the introduction, $\vartheta$ is $\Z$-valued concordance invariant that is related to Rasmussen's $s$-invariant (but which contains different information from it). As with the $s$-invariant, there is actually an infinite family of invariants $\{\vartheta_c : c \text{prime}\}$, which contains more information than any individual member: there are knots for which different choices of $c$ yield different values of $\vartheta_c$. It is thus generally interesting to keep track of the field characteristic. However, the following result is independent of the characteristic, and so we will omit it. The key property allowing for computation of $\vartheta(K)$ is that it can be extracted in a straightforward manner from the non-compact immersed curve invariant $\wt{\BN}^a(T_K)$, where $T_K$ is the Seifert-framed tangle associated with the obvious strong inversion on $K \# K$.

%It is a fact that 4-ended tangles without closed components have precisely one non-compact component in their immersed curve invariant. This is essentially because the non-compact components are a bordered version of Lee homology; see \cite{KWZ19}.

\begin{defn}[\!\!\cite{LewZib22}, Corollary 5.14] Given a knot $K$ in $S^3$,  the invariant is given by
\[\vartheta_c(K) = \lceil \sigma_c/2 \rceil,\]
where $\sigma_c$ is the slope of $\wt{\BN}^a(T_K; \F_c)$ near the bottom right tangle end.
\end{defn} 

It turns out that the following property is critical for the behaviour of $\vartheta_c$:

\begin{defn} A knot $K$ is said to be $\vartheta_c$-rational if $\wt{\BN}^a(T_K; \F_c)$ is rational, in the sense that, up to an overall shift in bigrading, it is equal to the invariant $\wt{\BN}(Q_n; \F_c)$, for some choice of $n$.
\end{defn}

In \cite{Mar25}, it is shown that $\vartheta_c$-rational knots $K$ in fact have $\vartheta_c(K) = 0$ for all $c$, a property conjectured in \cite{LewZib22}, where the authors also argue that most knots should not be $\vartheta_c$-rational, essentially because the property $\wt{\BN}^a(T) = \wt{\BN}(\oRes)$ should be rare. Consider the following expansion of the definition:

\begin{defn} We say that a cap-trivial tangle $T$ is $\vartheta$-rational if 
\[\wt{\BN}^a(T) = \wt{\BN}(\oRes).\]
\end{defn} 

Thus, a $\vartheta$-rational knot $K$ is one such that the tangle associated with the strong inversion on $K\#K$ is $\vartheta$-rational. Using our cabling operator, we find that $\vartheta$-rational tangles are easy to construct in infinite families:

\begin{proof}[Proof of \cref{prop:CbRational}] Let $T$ be any Seifert-framed cap-trivial tangle. Then, since $T(0)$ is a link, it follows that $\wt{\BN}^a(T)$ has its two ends at the bottom left and bottom right corners. In particular, $\wt{\BN}^a(T)\big|_{D_\bullet=0}$ consists of some homogeneous chain and the invariant $\wt{\BN}(Q_{2n})$, for some $n\in \Z$. From our model computation in \cref{subsec:unknot}, we see that $\Cb(T)$ is $\vartheta$-rational if $n \geq 0$ and $\Cb(\Cb(T))$ is $\vartheta$-rational if $n < 0$. This completes the proof.
\end{proof}

By considering how $\Cb^0(\wt{\BN}(T))$ differs from $\Cb(\wt{\BN}(T))$ on the tail of the invariant, the following is immediate:

\begin{cor} Let $T$ be cap-trivial Seifert framed. Then $\Cb^0(\Cb(0(T)))$ is Seifert framed $\vartheta$-rational.
\end{cor}

One is tempted to ask whether it is possible to construct an infinite family of $\vartheta$-rational knots using this procedure. The na\"ive idea that cabling may preserve $\vartheta$-rationality is indeed too na\"ive, because cabling does not commute with connected sum, but it is amusing to think of less na\"ive construction.

\bibliographystyle{alphaurl}
\bibliography{Candid}  

\end{document}